\documentclass[a4paper,reqno,12pt]{amsart}

\usepackage[all, cmtip, 2cell, knot,]{xy} \UseAllTwocells \SilentMatrices
\usepackage{xspace}
\usepackage{amsmath}
\usepackage{amstext}
\usepackage{amsfonts}
\usepackage[mathscr]{euscript}
\usepackage{amscd}
\usepackage{latexsym}
\usepackage{amssymb}
\usepackage{layout}
\usepackage{amsthm}
\usepackage{amsfonts}
\usepackage{amsxtra}
\usepackage{color}
\usepackage[usenames,dvipsnames,svgnames,table]{xcolor}
\usepackage{graphics}
\usepackage{bm}
\usepackage{bbm}
\usepackage{tabulary}
\usepackage{etoolbox}
\usepackage{fancyhdr}
\usepackage[bookmarks=true,hyperfootnotes=false]{hyperref}
\usepackage{tikz-cd}
\usepackage{tikz}
\usepackage{xargs}
\usepackage{standalone}
\usepackage{relsize}
\usepackage{chngcntr}
\usepackage{longtable}
\usepackage[strict]{changepage}
\usepackage{amsrefs}
\usepackage{marginnote}
\usepackage{mathtools}

\hypersetup{
			colorlinks=true,
			linkcolor=blue,
			anchorcolor=black,
			citecolor=green,
			urlcolor=black,
}

\usetikzlibrary{arrows}
\usetikzlibrary{arrows.meta}
\usepgflibrary{plotmarks}
\pgfsetplotmarksize{0.035cm}
\usetikzlibrary{positioning}
\usetikzlibrary{calc} 
\usetikzlibrary{graphs}

\tikzset{
  commutative diagrams/.cd,
  arrow style=tikz,
  diagrams={->}
}

\pgfsetplotmarksize{0.045cm}

\usetikzlibrary{decorations.pathreplacing,
                calligraphy}

\theoremstyle{plain}
\newtheorem{theorem}{Theorem}[section]
\newtheorem*{theorem*}{Theorem}  
\newtheorem{lemma}[theorem]{Lemma}
\newtheorem{proposition}[theorem]{Proposition}
\newtheorem{corollary}[theorem]{Corollary}
\theoremstyle{definition}
\newtheorem{definition}[theorem]{Definition}
\newtheorem{example}[theorem]{Example}

\newtheorem{remark}[theorem]{Remark}

\newcommand{\clA}{\mathcal{A}}

\newcommand{\clC}{\mathcal{C}}
\newcommand{\clD}{\mathcal{D}}

\newcommand{\clO}{\mathcal{O}}

\newcommand{\clU}{\mathcal{U}}

\newcommand{\za}{\alpha}
\newcommand{\zb}{\beta}
\newcommand{\zg}{\gamma}

\newcommand{\zd}{\delta}
\newcommand{\zD}{\Delta}
\newcommand{\zDm}{\zD_{mono}}
\newcommand{\zDopm}{\zD^{op}_{mono}}
\newcommand{\zuDop}{\zuD^{op}}
\newcommand{\zuDopm}{\zuD^{op}_{mono}}

\newcommand{\zve}{\varepsilon}

\newcommand{\hmu}{\hat\mu}
\newcommand{\muk}{\mu_k}
\newcommand{\hmuk}{\hat{\mu}_k}

\newcommand{\zs}{\sigma}
\newcommand{\zst}{\tilde{\zs}}

\newcommand{\pt}{\partial}
\newcommand{\ptt}{\tilde{\pt}}

\newcommand{\uf}{\underline{f}}

\newcommand{\ueta}{\underline{\eta}}
\newcommand{\oeta}{\overline{\eta}}
\newcommand{\uzve}{\underline{\zve}}
\newcommand{\zuD}{\underline{\Delta}}

\newcommand{\Sc}[1]{\scriptstyle{#1}}

\newcommand{\Sz}[1]{\scriptsize{#1}}

\newcommand{\Id}{\operatorname{Id}}

\newcommand{\pr}{\operatorname{pr}}

\newcommand{\bsim}{/\!\!\sim}
\newcommand{\bbsim}{/\!\sim}

\newcommand{\nid}{\noindent}
\newcommand{\ds}{\displaystyle}
\newcommand{\bk}{\bigskip}
\newcommand{\mk}{\medskip}

\newcommand{\ovl}[1]{\overline{#1}}

\newcommand{\up}[1]{^{(#1)}}

\newcommand{\rw}{\rightarrow}
\newcommand{\Rw}{\Rightarrow}
\newcommand{\xRw}{\xRightarrow}
\newcommand{\lw}{\leftarrow}
\newcommand{\Lw}{\Leftarrow}
\newcommand{\xrw}{\xrightarrow} 
\newcommand{\hxrw}[1]{\xymatrix{\ \ar@{^{(}->}^{#1}[r] & \ }}
\newcommand{\tiund}[1]{{\times}_{#1}\:}
\newcommand{\pro}[3]{#1\tiund{#2}\overset{#3}{\cdots}\tiund{#2}#1}

\newcommand{\tenx}[2]{#1\,\tiund{#2}\,#1}
\newcommand{\uset}[2]{\underset{#1}{#2}}
\newcommand{\oset}[2]{\overset{#1}{#2}}
\newcommand{\ob}{obj\,}

\newcommand{\mi}{\text{-}}

\newcommand{\cop}{\textstyle{\,\coprod\,}}

\newlength{\myline}
\newcommandx*{\doublearrow}[4][1=0, 2=1]{
  \draw[black,line width=\myline,double distance=3pt,#3] #4;
}

\newcommandx*{\triplearrow}[4][1=0, 2=3]{
  \draw[black,line width=\myline,double distance=4pt,#3] #4;
  \draw[black,line width=0.7pt] #4;
}

\newcommand{\Dop}{\Delta^{op}}

\newcommand{\Cop}{\clC^{op}}

\newcommand{\Delmono}{\Delta_{mono}^{^{op}}}

\newcommand{\Cat}{\mbox{$\mathsf{Cat}\,$}}

\newcommand{\fair}[1]{\mbox{$\mathsf{Fair^{#1}}$}}
\newcommand{\gfair}[1]{\mbox{$\mathsf{GFair^{#1}}$}}
\newcommand{\fairwg}[1]{\mbox{$\mathsf{Fair_{wg}^{#1}}$}}

\newcommand{\Gpd}{\mbox{$\mathsf{Gpd}$}}

\newcommand{\epid}{\mbox{$\mathsf{Epi}\zD$}}

\newcommand{\cathd}[1]{\mbox{$\mathsf{Cat_{hd}^{#1}}$}}

\newcommand{\catwg}[1]{\mbox{$\mathsf{Cat_{wg}^{#1}}$}}
\newcommand{\gcatwg}[1]{\mbox{$\mathsf{GCat_{wg}^{#1}}$}}
\newcommand{\tawg}[1]{\mbox{$\mathsf{Ta_{wg}^{#1}}$}}

\newcommand{\segps}{\mbox{$\mathsf{SegPs}$}}
\newcommand{\segpsc}[2]{\mbox{$\mathsf{SegPs}$}\funcat{#1}{#2}}
\newcommand{\Ssegpsc}[2]{\mbox{$\mathsf{SSegPs}$}\funcat{#1}{#2}}
\newcommand{\Ps}{\mbox{$\mathsf{Ps}$}}
\newcommand{\psc}[2]{\mbox{$\mathsf{Ps}$}\funcat{#1}{#2}}
\newcommand{\pscmon}[2]{\mbox{$\mathsf{Ps}$}\funcatmon{#1}{#2}}
\newcommand{\PsTalg}{\mbox{\sf Ps-}T\mbox{\sf -alg}}

\newcommand{\wt}{\widetilde}

\newcommand{\p}[1]{p^{(#1)}}

\newcommand{\ta}[1]{\mbox{$\mathsf{Ta^{#1}}$}}

\newcommand{\Set}{\mbox{$\mathsf{Set}$}}
\newcommand{\St}{St\,}

\newcommand{\tr}[1]{Tr_{#1}}

\newcommand{\Disc}[1]{Disc_{#1}}

\newcommand{\funcat}[2]{[\Delta^{{#1}^{op}},#2]}
\newcommand{\funcatc}[1]{[\clC^{op},#1]}
\newcommand{\fatcat}[2]{[#1,#2]}
\newcommand{\funcatmon}[2]{[\Delta_{mono}^{{#1}^{op}},#2]}

\newcommand{\dblar}[2]{({#1} \rightrightarrows {#2})}

\newcommand{\pcir}{circle[radius=0.045cm]}

\newcommand{\DT}[1]{%
\begin{tikzpicture}[#1]%
\filldraw (0,0) circle[radius=0.0cm] (0,0.1)  circle[radius=0.04cm]; 
\end{tikzpicture}%
}
\newcommand{\dts}{_{\DT{scale=0.7}}}

\tikzset{commutative diagrams/row sep/normal=1cm}

\newcommand{\dotp}{\oset{\dts}{+}}

\begin{document}

\title {Weakly globular double categories and weak units}

\author{Simona Paoli}
\address{\small{Department of Mathematics, School of Computing and Natural Sciences, University of Aberdeen, UK}}
\email{simona.paoli@abdn.ac.uk}

\date{\today}

\begin{abstract}
Weakly globular double categories are a model of weak $2$-categories based on the notion of weak globularity, and they are known to be suitably equivalent to Tamsamani $2$-categories. Fair $2$-categories, introduced by J. Kock, model weak $2$-categories with strictly associative compositions and weak unit laws. In this paper we establish a direct comparison between weakly globular double categories and fair $2$-categories and prove they are equivalent after localisation with respect to the $2$-equivalences. This comparison sheds new light on weakly globular double categories as encoding a strictly associative, though not strictly unital, composition, as well as the category of weak units via the weak globularity condition.
\end{abstract}

\maketitle

\section{Introduction}\label{intro}

Higher category theory is a rapidly developing field with applications to disparate areas, from homotopy theory, mathematical physics, algebraic geometry to, more recently, logic and computer science.

Higher categories comprise not only objects and morphisms (like in a category) but also higher morphism, which compose and have identities. A key point in higher category theory is the behaviour of these compositions. In a category, composition of morphisms is associative and unital. Higher categories in which these rules for compositions hold for morphisms in all dimensions are called strict higher categories:
they are not difficult to formalize, but they are of limited use in applications. A striking example is the case of strict $n$-groupoids, which are strict $n$-categories with invertible higher morphisms. These are algebraic models for the building blocks of topological spaces (the $n$-types) only when $n=0,1,2$, see \cite{S2} for a counterexample.

To model $n$-types for all $n$ (that is, to satisfy the 'homotopy hypothesis'), a more complex class of higher structures is needed, the weak $n$-categories. In a weak $n$-category, compositions are associative and unital only up to an invertible cell in the next dimension, in a coherent way.

In this paper we concentrate on the case $n=2$. In \cite{PP} we introduced a new model $\catwg{2}$ of weak 2-categories, called weakly globular double categories, based on a new paradigm to weaken higher categorical structures, which is the notion of weak globularity.

In \cite{Kock2006} Kock introduced the category $\fair{2}$ of fair 2-categories, to model weak 2-categories with strict composition laws. This model is based on the 'fat delta' category $\zuD$, which plays a prominent role in this work. In this paper we establish a direct comparison between $\catwg{2}$ and $\fair{2}$: we build a pair of functors between these categories and show they induce an equivalence of categories after localization with respect to the 2-equivalences. The proof of this result is completely independent on the equivalence of $\catwg{2}$ and $\fair{2}$ with bicategories (as in \cite{PP} and \cite{Kock2006} respectively) and highlights new features of weakly globular double categories: the fact that the weak globularity condition encodes the category of weak units and the fact that it is possible to extract from a weakly globular double category a strictly associative (though not strictly unital) composition.

The proof of our comparison result is highly non-trivial: the construction of the functor from $\fair{2}$ to $\catwg{2}$ uses several new properties of the fat delta category $\zuD$
which we establish in this work. These properties allow to functorially build from a fair 2-category a pseudo-functor from $\Dop$ to $\Cat$ of a special type, namely a Segalic pseudo-functor \cite{PP}. From the latter we functorially build a fair 2-category as in \cite{PP}.

Another point of novelty is that to establish the zig-zags of $2$-equivalences giving rise to the equivalence of categories after localization between $\catwg{2}$ and $\fair{2}$, we need to enlarge the context by introducing two new players: the category of Segalic pseudo-functors from the opposite of the 'fat delta' category to $\Cat$ and the category $\fairwg{2}$ of weakly globular fair $2$-categories.

Although this paper is about the case $n=2$, we envisage that the techniques developed here will be useful in the case of general dimension $n$. This will be tackled in future work, but we explain here the general set up for motivation.

The category $\catwg{2}$ was generalized in \cite{PBook2019} to the category $\catwg{n}$ of weakly globular $n$-fold categories and it was shown in \cite[Theorem 12.3.11]{PBook2019} that it satisfies the homotopy hypothesis: there is a subcategory $\gcatwg{n}\subset \catwg{n}$ called groupoidal weakly globular $n$-fold categories which is an algebraic model of $n$-types.

The category of Segalic pseudo-functors has been generalized to higher $n$ in \cite{PBook2019} and like for $n=2$, it is closely connected to $\catwg{n}$.

The category $\fair{2}$ was generalized in \cite{Kock2006} to $\fair{n}$ for any $n$. The latter encodes higher categories where all compositions are strictly associative but not strictly unital. For $n>2$, to date it is not known if $\fair{n}$ satisfies the homotopy hypothesis, except for the special case of 1-connected 3-types \cite{JoyalKock2007}.

For general $n$, one would seek comparison functors between $\catwg{n}$ and $\fair{n}$, factoring through the category of $n$-dimensional Segalic pseudo-functors, inducing an equivalence of categories after localization with respect to the $n$-equivalences.

As in the case $n=2$ (see Corollary \ref{compar-cor2}) we envisage this to restrict to an equivalence (after localization) between $\gcatwg{n}$ and $\gfair{n}$, the latter being a groupoidal version of fair $n$-categories. Since by \cite{PBook2019} $\gcatwg{n}$ is an algebraic model of $n$-types, this would mean that fair $n$-categories satisfy the homotopy hypothesis. This would give a proof of Simpson's weak units conjecture \cite{Simp}.

We envisage the case of general $n$ to be based on induction, the present paper for $n=2$ being the first step.

\bk

\nid \textbf{Organization of the paper} Sections \ref{backgr} to \ref{fa2cat} cover the necessary background: $2$-categorical techniques (Section \ref{backgr}), weakly globular double categories (Section \ref{sbs-wg-doubcat}), the fat delta (Section \ref{fatdelta_first}), fair $2$-categories (Section \ref{fa2cat}). These sections are expository although we adopt a different definition of the fat delta than the one of \cite{Kock2006} and we introduce corresponding new notation.

In Section \ref{fatdels} we establish some new properties of the 'fat delta' category $\zuD$ which are needed later on.
The comparison between $\catwg{2}$ and $\fair{2}$ is made of two parts. In Section \ref{weak2fair} we explain the passage from weakly globular double categories to fair $2$-categories. We construct in Theorem \ref{StrongpseThe-1} the functor
  \begin{equation*}
  F_2:\catwg{2}\rw \fair{2}
  \end{equation*}
   and in Proposition \ref{fairtowg2Cor-2} a natural transformation in $[\zuDop,\Cat]$
  \begin{equation*}
    S_2(X):F_2(X)\rw\tilde\pi^* X
  \end{equation*}
 (with $\tilde\pi^* X$ as in Definition \ref{fairtowg2Def-2}) which is a levelwise equivalence of categories.

In Section \ref{fairtowg2} we treat the other direction, from fair $2$-categories to weakly globular double categories. We define the functor (Definition \ref{fairtowg2Def-5})
\begin{equation*}
R_2:\fair{2}\rw\catwg{2}
\end{equation*}

We show in Section \ref{comparison} our main result Theorem \ref{compar-the1} that the functors $F_2$ and $R_2$ induce an equivalence of categories after localization with respect to the $2$-equivalences.
We prove this result by constructing for each $X\in \catwg{2}$ a $2$-equivalence $R_2 F_2 X \rw X$ in $\catwg{2}$  and a zig-zag of $2$-equivalences in $\fair{2}$ between $Y$ and $F_2 R_2 Y$ for each $Y\in \fair{2}$.
The construction of this zig-zag requires new notions and results developed in Section \ref{wgfair2}: the category of Segalic pseudo-functors $\segps\fatcat{\zuDop}{\Cat}$, the category $\fairwg{2}$ of weakly globular fair $2$-categories and Theorem \ref{fairtowg2The-1} relating the two.\medskip

\nid \textbf{Acknowledgements} This paper is partially based upon work supported by the National Science Foundation under Grant No. DMS-1440140 while the author was in residence at the Mathematical Sciences Research Institute in Berkeley, California, during the 'Higher Categories and Categorification' program in Spring 2020. I thank the organizers for their invitation to this program. I also thank the referee for many helpful comments.

\section{Techniques from $2$-category theory}\label{backgr}

In this Section we recall two techniques from $2$-category theory. The first is the strictification of pseudo-functors: this plays an important role in the theory of weakly globular double categories, as recalled in Section \ref{sbs-wg-doubcat}, and it will also be used in Section \ref{wgfair2} whose results are crucial to the proof of our main Theorem \ref{compar-the1}. The second technique is the transport of structure along an adjunction, which will be used in Proposition \ref{fairtowg2-pro1}, leading to the functor $T_2$ of Theorem \ref{fairtowg2-the1}. 

\subsection{Strictification of pseudo-functors}\label{sbs-strict-psfun}

Let $\mathcal{C}$ be a small category. The $2$-category of 2-functors, 2-natural transformations and modifications  $\funcatc{\Cat}$ is $2$-monadic over $[\ob(\Cop),\;\Cat]$ where $\ob(\Cop)$ is the set of objects of $\Cop$. Let
\begin{equation*}
    U:\funcatc{\Cat}\rw [\ob(\Cop),\Cat]
\end{equation*}
be the forgetful functor given by $(UX)_{k}=X_{k}$ for each $k\in\Cop$ and $X\in\funcatc{\Cat}$. Its left adjoint $F$ is given on objects by
\begin{equation*}
    (FY)_{k}=\underset{r\in \ob(\Cop)}{\cop}\Cop(r,k)\times Y_{r}
\end{equation*}
for $Y\in [\ob(\Cop),\Cat]$, $k\in \Cop$. If $T$ is the monad corresponding to the adjunction $F\dashv U$, then
\begin{equation*}
    (TY)_{k}= (UFY)_{k}= \underset{r\in \ob(\Cop)}{\cop}\Cop(r,k)\times Y_{r}\;.
\end{equation*}
A pseudo $T$-algebra is given by $Y\in [\ob(\Cop),\Cat]$, functors
\begin{equation*}
    h_{k}: \underset{r\in \ob(\Cop)}{\cop}\Cop(r,k)\times Y_{r} \rw Y_{k}
\end{equation*}
and additional data given by the axioms of pseudo $T$-algebra (see for instance \cite{PW}). This amounts precisely to functors from $\Cop$ to $\Cat$ and the $2$-category $\PsTalg$ of pseudo $T$-algebras  corresponds to the $2$-category $\Ps\funcatc{\Cat}$ of pseudo-functors, pseudo-natural transformations and modifications. Note that there is a commuting diagram
\begin{equation*}
\xymatrix{
\funcatc{\Cat} \ar@{^{(}->}[r] \ar_{U}[d] & \Ps\funcatc{\Cat} \ar_{U}[dl]\\
[\ob(\Cop),\Cat]
}
\end{equation*}
Recalling that, if $X$ is a set, $X\times \clC \cong \uset{X}{\cop}\clC$, we see that the pseudo $T$-algebra corresponding to $H \in \Ps\funcatc{\Cat}$ has structure map $h: TUH\rw UH$ as follows. Denoting
\begin{equation*}
(TUH)_{k}=\underset{r\in\clC}{\cop}\clC(k,r)\times H_{r}=\underset{r\in\clC}{\cop}\underset{\clC(k,r)}{\cop} H_{r}\;.
\end{equation*}
and, if $f\in\clC(k,r)$, denoting
\begin{equation*}
i_{r}=\underset{\clC(k,r)}{\cop}H_{r}\rw \underset{r\in\clC}{\cop}\underset{\clC(k,r)}{\cop} H_{r}=(TUH)_{r}\;,
 \end{equation*}
 \begin{equation*}
   j_f :H_{r}\rw \underset{\clC(k,r)}{\cop}H_{r}
   \end{equation*}
the coproduct inclusions, then the map $h$ is the unique map satisfying
\begin{equation}\label{eq-sbs-strict-psfun-1}
   h_{k}\,i_{r}\,j_f= H(f)\;.
\end{equation}
The structure map $TUH\, \rw\, UH$ carries a canonical enrichment to a pseudo-natural transformation $FUH \,\rw\,  H$.\; The strictification of pseudo-algebras result proved in \cite{PW} yields that every pseudo-functor from $\Cop$\! to $\Cat$ is equivalent, in $\Ps\funcatc{\Cat}$, to a $2$-functor, that is, an object of $\funcatc{\Cat}$.

Given a pseudo $T$-algebra as above, \cite{PW} considers the factorization  of\linebreak $h:TUH\rw UH$ as
\begin{equation*}
    TUH\xrw{v}L\xrw{g}UH
\end{equation*}
with $v_{k}$ bijective on objects and $g_{k}$ fully faithful, for each $k\in\Cop$. It is shown in \cite{PW} that $g$ is a pseudo-natural transformation and it is possible to give a strict $T$-algebra structure $TL\rw L$ such that $(g,Tg)$ is an equivalence of pseudo $T$-algebras. It is immediate to see that, for each $k\in\Cop$, $g_{k}$ is an equivalence of categories.

We define
\begin{equation}\label{eq-sbs-strict-psfun-2}
L=\St H\;.
\end{equation}
The above constructions are natural in \cite{PW} so given a morphism $H\rw H'$ in $\Ps[\clC^{op},\Cat]$ this gives a morphism $\St H=L\rw L'=\St H'$ in  $[\clC^{op},\Cat]$.
Further, it is shown in \cite{Lack} that $\St:\Ps\funcatc{\Cat}\rw\funcatc{\Cat}$ as defined in \eqref{eq-sbs-strict-psfun-2} is left adjoint to the inclusion
\begin{equation*}
  J:\funcatc{\Cat}\rw\Ps\funcatc{\Cat}
\end{equation*}
 and that the components of the units are equivalences in $\Ps\funcatc{\Cat}$.
\begin{remark}\label{strict-psfun-Rem-1}
  It is straightforward from \cite{PW} that if $H\in\funcatc{\Cat}$ the pseudo-natural transformation $g:\St H \rw H$ is a 2-natural transformation.
\end{remark}

In this work we use the strictification of pseudo-functors in Section \ref{wgfair2} (in the case where $\clC=\zuD$) and in Section \ref{fairtowg2} (in the case  where $\clC=\zD$). As we recall in Section \ref{sbs-wg-doubcat}, this technique also plays a crucial role in the theory of weakly globular double categories.


\subsection{Transport of structure}\label{sbs-trans}
The following $2$-categorical technique will be used in Section \ref{fairtowg2}. Its proof relies on \cite[Theorem 6.1]{lk}.

\begin{lemma}{\rm{\cite{PP}}}\label{lem-PP}
     Let $\clC$ be a small $2$-category, $F,F':\clC\rw\Cat$ be $2$-functors, and $\mu:F\rw F'$
    a $2$-natural transformation. Suppose that, for all objects $C$ of $\clC$, the
    following conditions hold:
\begin{enumerate}
  \item $G(C),\;G'(C)$ are objects of $\Cat$ and there are adjoint equivalences of
  categories $\zb_C\vdash\za_C$, $\zb'_C\vdash\za'_C$,
\begin{equation*}
  \zb_C:G(C)\;\rightleftarrows\;F(C):\za_C\qquad\qquad
  \zb'_C:G'(C)\;\rightleftarrows\;F'(C):\za'_C,
\end{equation*}

  \item  there are functors $\xi_C:G(C)\rw G'(C),$\mk

  \item  there is an invertible $2$-cell
\begin{equation*}
  \gamma_C:\xi_C\,\za_C\Rightarrow\za'_C\,\mu_C.
\end{equation*}
\end{enumerate}
Then
\begin{itemize}
  \item [a)] There exists a pseudo-functor $G:\clC\rw\Cat$ given on objects by $G(C)$,
  and pseudo-natural transformations $\za:F\rw G$, $\zb:G\rw F$ with
  $\za(C)=\za_C$, $\zb(C)=\zb_C$; these are part of an adjoint equivalence
  $\zb\vdash\za$ in the $2$-category $\Ps[\clC,\Cat]$. Similarly there is a pseudo-functor $G':\clC\rw\Cat$ and pseudo-natural transformations $\za':F'\rw G'$ and $\zb':G'\rw F'$.\mk

  \item [b)] There is a pseudo-natural transformation $\xi:G\rw G'$ with
  $\xi(C)=\xi_C$ and an invertible $2$-cell in $\Ps[\clC,\Cat]$,
  $\gamma:\xi\za\Rightarrow\za'\mu$ with $\gamma(C)=\gamma_C$.
\end{itemize}
\end{lemma}
\begin{proof}\

Recall \cite{PW} that the functor 2-category $[\clC,\Cat]$ is 2-monadic over
$[\ob(\clC),\Cat]$, where $\ob(\clC)$ is the set of objects in $\clC$. Let
\begin{equation*}
  \clU:[\clC,\Cat]\rw[\ob(\clC),\Cat]
\end{equation*}
be the forgetful functor. Let $T$ be the 2-monad; then the pseudo-$T$-algebras are precisely the pseudo-functors from
$\clC$ to $\Cat$.

The adjoint equivalences $\zb_C\vdash\za_C$ amount precisely to an adjoint equivalence in $[\ob(\clC),\Cat]$, $\;\zb_0\vdash\za_0$, $\;\zb_0:G_0\;\;\rightleftarrows\;\;\clU F:\za_0$, where $\;G_0(C)=G(C)$ for
all $C\in \ob(\clC)$. By \cite[Theorem 6.1]{lk} this equivalence enriches to an adjoint equivalence $\zb\vdash\za$ in $\Ps[\clC,\Cat]$
\begin{equation*}
  \zb:G\;\rightleftarrows\; F:\za
\end{equation*}
between $F$ and a pseudo-functor $G$; we have $\clU G=G_0$, $\;\clU\za=\za_0$,
$\;\clU\zb=\zb_0$; hence on objects $G$ is given by $G(C)$, and
$\za(C)=\clU\za(C)=\za_C$, $\;\zb(C)=\clU\zb(C)=\zb_C$.

Let $\nu_C:\Id_{G(C)}\Rw\za_C\zb_C$ and $\zve_C:\zb_C\za_C\Rw\Id_{F(C)}$ be
the unit and counit of the adjunction $\zb_C\vdash\za_C$. Given a morphism $f:C\rw D$ in $\clC$, we have
\begin{equation*}
  G(f)=\za_D F(f)\zb_C
\end{equation*}
Given morphisms $C\xrw{f}D\xrw{g}E$ in $\clC$, the 2-cell
\begin{equation*}
\xymatrix{
G(C) \ar@/_2.7pc/[rrrr]_{G(gf)} \ar^{G(f)}[rr] && G(D) \ar@{}[d]|{\big\Downarrow} \ar^{G(g)}[rr] && G(E)\\
&& \
}
\end{equation*}
is obtained by the following pasting diagram
\begin{equation*}
\xymatrix{
G(C) \ar^{G(f)}[rr]\ar_{\zb_{C}}[d] && G(D)\ar@{}|{\big\Downarrow\zve_D}[d]\ar^{G(g)}[rr]\ar^{\zb_{D}}[dr] && G(E)\\
F(C)\ar_{F(f)}[r] & F(D) \ar^{\za_D}[ur]\ar@{}|{=}[rr] && F(D) \ar_{F(g)}[r] & F(E)\ar_{\za_E}[u]
}
\end{equation*}
while, for each $C\in\clC$, the 2-cell
\begin{equation*}
\xymatrix{
G(C)  \ar@/_2.7pc/[rr]_{\Id_{G(C)}} \ar^{G(\Id_C)}[rr]  & \ar@{}|{\big\Downarrow}[d] & G(C)\\
& \ &
}
\end{equation*}
is given by $\nu_C^{-1}:\za_C\zb_C=G(\Id_G)\Rw\Id_{G(C)}$. These data satisfy the coherence axioms for pseudo-functors by  \cite[Theorem 6.1]{lk}.

We have natural isomorphisms:
\begin{eqnarray*}
   \za_f &:& G(f)\za_C=\za_D F(f)\zb_C\za_C\overset{\za_D F(f)\zve_C}{=\!=\!=\!=\!\Rw} \za_D F(f)\\
   \zb_f &:& F(f)\zb_C\overset{\nu_{F(f)}\zb_C}{=\!=\!=\!\Rw}\zb_D\za_D F(f)\zb_C=\zb_D
   G(f).
\end{eqnarray*}
Also, the natural isomorphism
\begin{equation*}
\xi_f: G'(f)\xi_C\Rw\xi_D G(f)
\end{equation*}
is the result of the following pasting
{\labelmargin-{1pt}
\begin{equation*}
    \xy
    0;/r.16pc/:
    (-40,40)*+{G(C)}="1";
    (40,40)*+{G'(C)}="2";
    (-40,-40)*+{G(D)}="3";
    (40,-40)*+{G'(D)}="4";
    (-20,20)*+{F(C)}="5";
    (20,20)*+{F'(C)}="6";
    (-20,-20)*+{F(D)}="7";
    (20,-20)*+{F'(D)}="8";
    {\ar^{\xi_C}"1";"2"};
    {\ar_{G(f)}"1";"3"};
    {\ar^{G'(f)}"2";"4"};
    {\ar_{\xi_D}"3";"4"};
    {\ar^{\mu_C}"5";"6"};
    {\ar^{F(f)}"5";"7"};
    {\ar_{F'(f)}"6";"8"};
    {\ar_{\mu_D}"7";"8"};
    {\ar_{\za_C}"5";"1"};
    {\ar^{\za_D}"7";"3"};
    {\ar^{\za'_C}"6";"2"};
    {\ar_{\za'_D}"8";"4"};
    {\ar@{=>}^{\gamma_C}(0,33);(0,27)};
    {\ar@{=>}^{\gamma_D^{-1}}(0,-27);(0,-33)};
    {\ar@{=>}^{\za_f'}(30,3);(30,-3)};
    {\ar@{=>}^{\za_f}(-30,3);(-30,-3)};
\endxy
\end{equation*}}
\end{proof}
This diagram gives invertible 2-cells:
\begin{equation*}
\begin{split}
\zg_c\Id_{\za_C^{-1}} &: \xi_C\Rw\za'_C\mu_C\za_C^{-1} \\
  \za'_f &: G'(f)\za_C^{'}\Rw\za'_D F'(f) \\
  \zg^{-1}_D &: \za'_D\mu_D\Rw\xi_D\za_D \\
  \za^{-1}_f & : \za_D F(f)\Rw G(f)\za_C\;.
\end{split}
\end{equation*}
Using the fact that $F'(f)\mu_C=\mu_D F(f)$, we obtain the 2-cell $\xi_f$ by composition of the following invertible 2-cells:
\begin{equation*}
\begin{split}
   & G'(f)\xi_C\xRw{\Id_{G'(f)}\zg_C\Id_{\za^{-1}_C}} G'(f)\za'_C\mu_C\za_C^{-1} \xRw{\za_{f'}\Id_{\mu_C\za^{-1}_{C}}} \za'_D F'(f)\mu_C\za^{-1}_{C}=\\
   =& \za'_D\mu_DF(f)\za^{-1}_{C}\xRw{\zg^{-1}_{D}\Id_{F(f)\za^{-1}_C}} \xi_D\za_D F(f)\za^{-1}_C \xRw{\Id_{\xi_D}\za^{-1}_f\Id_{\za^{-1}_C}}\xi_D G(f)\za_C\za^{-1}_C=\xi_D G(f) \\
\end{split}
\end{equation*}

\begin{remark}\label{sbs-trans-rem1}
We can specialize Lemma \ref{lem-PP} to the case where $F=F',\,\mu=\Id$ and, for all $C\in\ob\clC$, $G(C)=G'(C)$, $\xi_C=\za'_C\zb_C$ and $\zg_C:\xi_C\za_C=\za'_C\zb_C\za_C\Rightarrow \za'_C$ is given by $\za'_C\zve_C$. This amounts to constructing pseudo-functors $G,G':\clC\rw\Cat$ (with $G(C)=G'(C)$) from the functor $F$ using two distinct choices of adjoint equivalences of categories $\zb_C\vdash \za_C$ and $\zb'_C\dashv \za'_C$. This yields a pseudo-natural transformation $\xi:G\rw G'$ with $\xi(C)=\xi_C$.

Similarly, applying Lemma  \ref{lem-PP} when $F=F'$, $\mu=\Id$, $G(C)=G'(C)$ and $\xi'_C:G'(C)\rw G(C)$ is given by $\xi'_C=\za_C\zb'_C$ and $\zg'_C:\xi'_C\za'_C=\za_C\zb'_C\za'_C\Rightarrow \za_C$ is given by $\za_C\zve'_C$. This yields a pseudo-natural transformation $\xi':G'\rw G$ with $\xi'(C)=\xi'_C$.

The two pseudo-natural transformations $\xi$ and $\xi'$ are an equivalence between pseudo-functors $G$ and $G'$ in the 2-category $\Ps[\clC,\Cat]$ since $\xi'_C\xi_C=\za_C\zb'_C\za'_C\zb_C\cong\za_C\zb_C\cong\Id$ and $\xi_C\xi'_C=\za'_C\zb_C\za_C\zb'_C\cong\za'_C\zb'_C\cong\Id$.
\end{remark}


\section{Weakly globular double categories and Segalic pseudo-functors}\label{sbs-wg-doubcat}
We recall the theory of weakly globular double categories and of Segalic pseudo-functors, originally introduced in \cite{PP} and further developed in \cite{PBook2019}.

\subsection{Weakly globular double categories}\label{sbs-wg-d} We first need the notion of Segal maps and of induced Segal maps.
\begin{definition}\label{wg-doubcat-def-1}
    Let ${X\in\funcat{}{\clC}}$ be a simplicial object in a category $\clC$ with pullbacks. For each ${1\leq j\leq k}$ and $k\geq 2$, let ${\nu_j:X_k\rw X_1}$ be induced by the map  $[1]\rw[k]$ in $\Delta$ sending $0$ to ${j-1}$ and $1$ to $j$. Then the following diagram commutes:
\small
\begin{equation}\label{wg-doubcat-eq-1}
\begin{gathered}
\xymatrix@C=18pt{
&&&& X\sb{k} \ar[llld]_{\nu\sb{1}} \ar[ld]_{\nu\sb{2}} \ar[rrd]^{\nu\sb{k}} &&&& \\
& X\sb{1} \ar[ld]_{d\sb{1}} \ar[rd]^{d\sb{0}} &&
X\sb{1} \ar[ld]_{d\sb{1}} \ar[rd]^{d\sb{0}} && \dotsc &
X\sb{1} \ar[ld]_{d\sb{1}} \ar[rd]^{d\sb{0}} & \\
X\sb{0} && X\sb{0} && X\sb{0} &\dotsc X\sb{0} && X\sb{0}
}
\end{gathered}
\end{equation}
\normalsize
If  ${\pro{X_1}{X_0}{k}}$ denotes the limit of the lower part of the
diagram \eqref{wg-doubcat-eq-1}, the \emph{$k$-th Segal map of $X$} is the unique map
$$
\muk:X\sb{k}~\rw~\pro{X\sb{1}}{X\sb{0}}{k}
$$
\noindent such that ${\pr_j\,\muk=\nu\sb{j}}$ where
${\pr\sb{j}}$ is the $j^{th}$ projection.
\end{definition}
Let $\Cat\clC$ be the category of internal categories in $\clC$ and internal functors \cite{Borc} and let $N:\Cat\clC\rw\funcat{}{\clC}$ be the nerve functor. The Segal maps characterize the essential image of $N$. Namely, an object $X\in\funcat{}{\clC}$ is in the essential image of $N$ if and only if its Segal maps $X_k\rw\pro{X_1}{X_0}{k}$ are isomorphisms for all $k\geq 2$.
\begin{remark}\label{wg-doubcat-rem-00}
When $\clC=\Set$, $N:\Cat\rw\funcat{}{\Set}$ is the nerve of small categories. This functor is fully faithful, so we can identify $\Cat$ with the essential image of $N$. We will make this identification throughout this work.
\end{remark}
\begin{definition}\label{wg-doubcat-def-2}

    Let ${X\in\funcat{}{\clC}}$ and suppose that there is a map in $\clC$
     \begin{equation*}
       \zg: X_0 \rw Y
     \end{equation*}
      such that the limit of the diagram
\begin{equation*}
\xymatrix@R25pt@C16pt{
& X\sb{1} \ar[ld]_{\zg d\sb{1}} \ar[rd]^{\zg d\sb{0}} &&
X\sb{1} \ar[ld]_{\zg d\sb{1}} \ar[rd]^{\zg d\sb{0}} &\cdots& k &\cdots&
X\sb{1} \ar[ld]_{\zg d\sb{1}} \ar[rd]^{\zg d\sb{0}} & \\
Y && Y && Y\cdots &&\cdots Y && Y
    }
\end{equation*}
exists; denote the latter by $\pro{X_1}{Y}{k}$. Then the following diagram commutes, where $\nu_j$ is as in Definition \ref{wg-doubcat-def-1}, and $k\geq 2$
\begin{equation*}
\xymatrix@C=20pt{
&&&& X\sb{k} \ar[llld]_{\nu\sb{1}} \ar[ld]_{\nu\sb{2}} \ar[rrd]^{\nu\sb{k}} &&&& \\
& X\sb{1} \ar[ld]_{\zg d\sb{1}} \ar[rd]^{\zg d\sb{0}} &&
X\sb{1} \ar[ld]_{\zg d\sb{1}} \ar[rd]^{\zg d\sb{0}} && \dotsc &
X\sb{1} \ar[ld]_{\zg d\sb{1}} \ar[rd]^{\zg d\sb{0}} & \\
Y && Y && Y &\dotsc Y && Y
}
\end{equation*}
The \emph{$k$-th induced Segal map of $X$} is the unique map
\begin{equation*}
\hmuk:X\sb{k}~\rw~\pro{X\sb{1}}{Y}{k}
\end{equation*}
such that ${\pr_j\,\hmuk=\nu\sb{j}}$ where ${\pr\sb{j}}$ is the $j^{th}$ projection. If $Y=X_0$ and $\gamma$ is the identity, the induced Segal map coincides with the Segal map of Definition \ref{wg-doubcat-def-1}.
\end{definition}
\begin{definition}\label{wg-doubcat-def-3}
  A homotopically discrete category is an equivalence relation, that is a groupoid with no non-trivial loops. We denote by $\cathd{}$ the category of homotopically discrete categories.
\end{definition}
Let $p:\Cat\rw\Set$ be the isomorphism classes of objects functor. As discussed for instance in \cite[Lemma 4.1.4]{PBook2019} $p$ preserves pullbacks over discrete objects and sends equivalences of categories to isomorphisms.

\begin{definition}\label{discr}
If $X\in\cathd{}$, we denote by $X^d$ the discrete category on the set $p X$. There is a map $\zg: X\rw X^d$, called discretization, which is an equivalence of categories.
\end{definition}

The category of weakly globular double categories was originally introduced in \cite{PP} and further studied in \cite{PBook2019}:

\begin{definition}\label{wg-doubcat-def-4}
The category $\catwg{2}$ of weakly globular double categories is the full subcategory of $[\Dop,\Cat]$ whose objects $X$ are such that
\begin{itemize}
  \item [a)] $X_0\in\cathd{}$.\mk

  \item [b)] For each $k\geq 2$ the Segal maps
  \begin{equation*}
    \muk:X_k\rw \pro{X_1}{X_0}{k}
  \end{equation*}
are isomorphisms.\mk

  \item [c)] For each $k\geq 2$ the induced Segal maps
  \begin{equation*}
    \hmuk:X_k\rw \pro{X_1}{X_0^d}{k}
  \end{equation*}
  which are induced by the discretization map $\zg: X_0\rw X_0^d$ are equivalences of categories.
\end{itemize}
\end{definition}
Note that because of condition b), $\catwg{2}$ is a full subcategory of the category $\Cat^2$ of double categories, that is of internal categories in $\Cat$.
\begin{remark}\label{wg-doubcat-rem-1}
 Let $\p{1}:\catwg{2}\rw\funcat{}{\Set}$ be given by $(\p{1}X)_k=p X_k$ for all $k\geq 0$. Then $\p{1}X$ is the nerve of a category. In fact, since $p$ sends equivalences of categories to isomorphisms and preserves pullbacks over discrete objects, for each $X\in\catwg{2}$ and $k\geq 2$ there are isomorphisms
 \begin{equation*}
 \begin{split}
 & (\p{1}X)_k=p(X_k)\cong p(\pro{X_1}{X_0^d}{k})\cong \\
 & \cong \pro{p(X_1)}{p(X_0^d)}{k}\cong \pro{p(X_1)}{p(X_0)}{k}\;.
 \end{split}
 \end{equation*}
Thus, using the notational convention of remark \ref{wg-doubcat-rem-00} we can write
\begin{equation*}
  \p{1}:\catwg{2}\rw \Cat\;.
\end{equation*}
\end{remark}
In what follows, given $X\in\catwg{2}$ and $a,b\in X_0^d$ we denote by $X(a,b)$ the fiber at $(a,b)$ of the map given by the composite
\begin{equation*}
  X_1\xrw{(\pt_0,\pt_1)}X_0\times X_0\xrw{\zg\times\zg}X^d_0\times X^d_0\;.
\end{equation*}
where $\zg: X_0\rw {X_0}^d$ is the discretization map as in Definition \ref{discr}.
The category $X(a,b)$ plays the role of 'hom-category'.
\begin{definition}\label{wg-doubcat-def-5}
  A morphism $F:X\rw Y$ in $\catwg{2}$ is a $2$-equivalence if:
  \begin{itemize}
    \item [i)] For all $a,b\in X_0^d$ the morphism $F_{(a,b)}:X(a,b)\rw Y(Fa,Fb)$ is an equivalence of categories.\mk

    \item [ii)] $\p{1}F$ is an equivalence of categories.
  \end{itemize}
\end{definition}
\begin{remark}\label{wg-doubcat-rem-2} The following properties were shown in \cite{PBook2019}:

\begin{itemize}
  \item [a)]If $F$ is a levelwise equivalence of categories it is in particular a $2$-equivalence. When $F_0=\Id$, the two notions coincide.\mk

  \item [b)] Condition ii) in Definition \ref{wg-doubcat-def-5} can be relaxed to requiring that $p\p{1}X$ is surjective.\mk

  \item [c)] $2$-Equivalences in $\catwg{2}$ have the $2$-out-$3$ property.
\end{itemize}

\end{remark}

\begin{definition}\label{D2def}
  Given $X\in\catwg{2}$ let $D_2 X\in\funcat{}{\Cat}$ be
  \begin{equation*}
    (D_2 X)_n=
    \left\{
      \begin{array}{ll}
        X_0^d, & n=0 \\
        X_n, & n>0\;.
      \end{array}
    \right.
  \end{equation*}%
The face operators $\pt'_0,\pt'_1:X_1\rightrightarrows X_0^d$ are given by $\pt'_i=\zg\pt_i$ where  $\pt_i:X_1\rightrightarrows X_0,\;\, i=0,1$ while the degeneracy operator $\zs'_0:X_0^d\rw X_1$ is $\zs'_0=\zs_0\zg'$ where $\zg': X_0^d\rw X_0$ is a pseudo-inverse of $\zg$,\, $\zg\zg'=\Id$. All the other face and degeneracy operators of $D_2 X$ are as in $X$.
\end{definition}

\begin{remark}\label{D2def-Rem}
  $D_2X$ can be obtained by transport of structure along the equivalences of categories $(D_2X)_k\simeq X_k$ given by $\gamma'$ for $k=0$ and $id$ for $k>0$. Thus by Lemma \ref{lem-PP} there is a pseudo-natural transformation $D_2 X\rw X$ in $\Ps\funcat{}{\Cat}$ which is a levelwise equivalence of categories.
\end{remark}

\subsection{Weakly globular Tamsamani $2$-categories and Segalic pseudo-functors} We recall from \cite{PBook2019} the definitions of the categories $\tawg{2}$ of weakly globular Tamsamani $2$-categories and $\segpsc{}\Cat$ of Segalic pseudo-functors, as well as the construction of the functor
\begin{equation*}
 \tr{2}:\tawg{2}\rw\segpsc{ }{\Cat}.
\end{equation*}
These play an important role in Section \ref{weak2fair} in building the functor $F_2:\catwg{2}\rw\fair{2}$.
\begin{definition}\label{segpseudo-def-1}\cite{PBook2019}
  The category $\tawg{2}$ of weakly globular Tamsamani $2$-categories is the full subcategory of $\funcat{} {\Cat}$ whose objects $X$ are such that
  \begin{itemize}
    \item [i)] $X_0\in\cathd{}$.\mk

    \item [ii)] The induced Segal maps $\hmu_k:X_k\rw \pro{X_k}{X_0^d}{k}$ are equivalences of categories for all $k\geq 2$.
  \end{itemize}
\end{definition}
\begin{remark}\label{segpseudo-rem-1}\

\begin{itemize}
  \item [a)] From the definitions, $\catwg{2}$ is the full subcategory of $\tawg{2}$ whose objects $X$ are such that all the Segal maps are isomorphisms.\mk

  \item [b)] There is a functor $\p{1}:\tawg{2}\rw\Cat$ given by $(\p{1}X)_k=p X_k$, $k\geq 0$. The proof that the essential image of $\p{1}:\tawg{2}\rw\funcat{}{\Set}$ consists of nerves of categories is as in the case of $\catwg{2}$.
\end{itemize}
\end{remark}
\nid The category $\ta{2}$ of Tamsamani $2$-categories was originally introduced in \cite{Ta} but can now be seen as a subcategory of $\tawg{2}$ as follows:

\begin{definition}\label{segpseudo-def-2}
  The full subcategory of $\tawg{2}$ whose objects $X$ are such that $X_0$ is discrete is the category $\ta{2}$ of Tamsamani $2$-categories.
\end{definition}
  Note that for Tamsamani $2$-categories the induced Segal maps and the Segal maps coincide.\bk

Let $H\in\Ps\funcat{}{\Cat}$ be such that $H_0$ is discrete. There is a commuting diagram in $\Cat$ for each $k\geq 2$,
\small
\begin{equation*}
\xymatrix@C=18pt{
&&&& H\sb{k} \ar[llld]_{\nu\sb{1}} \ar[ld]_{\nu\sb{2}} \ar[rrd]^{\nu\sb{k}} &&&& \\
& H\sb{1} \ar[ld]_{d\sb{1}} \ar[rd]^{d\sb{0}} &&
H\sb{1} \ar[ld]_{d\sb{1}} \ar[rd]^{d\sb{0}} && \dotsc &
H\sb{1} \ar[ld]_{d\sb{1}} \ar[rd]^{d\sb{0}} & \\
H\sb{0} && H\sb{0} && H\sb{0} &\dotsc H\sb{0} && H\sb{0}
}
\end{equation*}
\normalsize
where $\nu_j$ is induced by the map $[1]\rw [k]$ sending $0$ to $j-1$ $1$ to $j$. Hence there is a Segal map

\begin{equation*}
  H_k\rw \pro{H_1}{H_0}{k}\;.
\end{equation*}
\begin{definition}\label{segpseudo-def-3}
  The category $\segpsc{}{\Cat}$ is the full subcategory of $\Ps[\Dop,\Cat]$ whose objects $H$ are such that
  \begin{itemize}
    \item [i)] $H_0$ is discrete.\mk

    \item [ii)] All Segal maps are isomorphisms for all $k\geq 2$
    \begin{equation*}
      H_k\cong\pro{H_1}{H_0}{k}\;.
    \end{equation*}

  \end{itemize}
\end{definition}
In \cite{PBook2019} we constructed a functor
\begin{equation}\label{functor_tr2}
  \tr{2}:\tawg{2}\rw\segpsc{}{\Cat}
\end{equation}
by applying transport of structure to $X\in\tawg{2}\subset\funcat{}{\Cat}$ along the equivalence of categories $\zg:X_0\rw X_0^d$, $\Id:X_1\rw X_1$, $\hmu_k:X_k\rw \pro{X_1}{X_0^d}{k}$ for $k\geq 2$. Thus by construction
\begin{equation*}
  (\tr{2}X)_k=
  \left\{
    \begin{array}{ll}
      X_0^d, & k=0 \\
      X_1, & k=1 \\
      \pro{X_1}{X_0^d}{k}, & k>1\;.

    \end{array}
  \right.
\end{equation*}
and there is a pseudo-natural transformation $t_2(X):\tr{2}X\rw X$ which is a levelwise equivalence of categories.

Segalic pseudo-functors and weakly globular double categories are related by the following result, which we will use later.

 \begin{theorem}{\rm{\cite{PBook2019}}}\label{book-st-cat}
 The strictification functor $\St:\Ps\funcat{}{\Cat}\rw\funcat{}{\Cat}$ restricts to a functor
\begin{equation*}
  \St:\segpsc{}{\Cat}\rw\catwg{2}\;.
  \end{equation*}
 \end{theorem}

 \bk

\nid Since $\ta{2}\subset \tawg{2}$, by composition we obtain a functor 'rigidification'
\begin{equation*}
  Q_2:\ta{2}\xrw{\tr{2}}\segpsc{}{\Cat}\xrw{\St}\catwg{2}\;.
\end{equation*}
In \cite{PBook2019} we also built a functor 'discretization' in the opposite direction
\begin{equation*}
  \Disc{2}:\catwg{2}\rw \ta{2}
\end{equation*}
and we showed that $Q_2$ and $\Disc{2}$ induce an equivalence of categories after localization with respect to the $2$-equivalences. Combining this with the result of \cite{LackPaoli2008} relating $\ta{2}$ to bicategories, we obtained in \cite{PP} a $2$-categorical equivalence between weakly globular double categories and bicategories, showing that $\catwg{2}$ is a suitable model of weak $2$-categories.

\section{The fat delta category}\label{fatdelta_first} In this section we recall from \cite{Kock2006} the  category 'fat delta', denoted $\zuD$ and we discuss the notion of Segal maps for functors from $\zuDop$ to a category with pullbacks. These notions will be used in section \ref{fair2cats} to define fair $2$-categories. The content of this section is essentially contained in \cite{Kock2006}: however, we adopt here a different definition of the fat delta from the one used in \cite{Kock2006}, and we adopt a new notation for the Segal maps.
\subsection{Definition of the fat delta and elementary properties}\label{fatd}

We recall the definition of the fat delta category $\zuD$. This category was introduced in \cite{Kock2006} in terms of coloured semi-ordinals, and an alternative description was stated in \cite{Kock2006} without proof. Since this alterative description is needed for several proofs in this work, we adopt it throughout as our definition of $\zuD$.
\begin{definition}\label{fatd-def1}
The category $\epid$ has for objects the epimorphisms in $\zD$ and for morphisms the commuting squares in $\zD$
\begin{equation}\label{fatd-eq-1}
\begin{gathered}
\xymatrix{[n']\ar^{f}[r]\ar@{->>}_{\eta_1}[d] & [m']\ar@{->>}^{\eta_2}[d]\\
[n]\ar_{g}[r] & [m]
}
\end{gathered}
\end{equation}
where the vertical arrows $\eta_1,\eta_2$ are epimorphisms.
\end{definition}

We now introduce the fat delta category $\zuD$ as a subcategory of $\epid$.
\begin{definition}\label{fatd-def3}
The fat delta category $\zuD$ has for objects the epimorphisms in $\zD$ and for morphisms the commuting squares in $\zD$
\begin{equation}\label{fatd-eq-2}
\begin{gathered}
\xymatrix{[n']\ar@{^(->}^{f}[r]\ar@{->>}_{\eta_1}[d] & [m']\ar@{->>}^{\eta_2}[d]\\
[n]\ar_{g}[r] & [m]
}
\end{gathered}
\end{equation}
where the vertical arrows $\eta_1,\eta_2$ are epimorphisms and the top arrow $f$ is a monomorphism.
\end{definition}

An important role in the theory is played by the projection functor $\pi:\zuD\rw\zD$. This takes the epimorphism $\eta:[n']\rw[n]$ to the target $[n]$ and the morphism in $\zuD$ given by \eqref{fatd-eq-2} to the target morphism $g$.

Denote by $\zD_{mono}$ the wide subcategory of $\zD$ whose morphisms are injective maps, so they are uniquely a composition of face maps $\zve_i:[n-1]\rw[n]$ ($0\leq i\leq n$) where $\zve_i$ is the unique injective map whose image misses $i$.

There is a vertical inclusion
\begin{equation*}
v:\zDm \hookrightarrow \zuD
\end{equation*}
as follows. Given $[n]\in\zDm$, $v([n])$ is the surjection $[n]\twoheadrightarrow[0]$. Given a map $\zve:[n]\hookrightarrow[m]$ in $\zDm$, $v(\zve)$ is the map in $\zuD$
\begin{equation*}
\xymatrix{[n]\ar@{^(->}^{\zve}[r]\ar@{->>}[d] & [m]\ar@{->>}[d]\\
[0]\ar@{=}[r] & [0]
}
\end{equation*}
There is also a horizontal inclusion
\begin{equation*}
h:\zDm \hookrightarrow \zuD\;.
\end{equation*}
Given $[n]\in\zDm$, $h([n])$ is the surjection $\Id:[n]\rw [n]$. Given $\zve:[n]\hookrightarrow[m]$ in $\zDm$, $h(\zve)$ is the map in $\zD$
\begin{equation*}
\xymatrix{[n]\ar@{^(->}^{\zve}[r]\ar_{\Id}[d] & [m]\ar@{->>}^{\Id}[d]\\
[n]\ar_{\zve}[r] & [m]
}
\end{equation*}
We will often identify $h[n]$ with $[n]$ and $h(\zve)$ with $\zve$.

The composite functor $\zDm\overset{h}{\hookrightarrow}\zuD\overset{\pi}\twoheadrightarrow\zD$ is the standard inclusion of $\zDm$ in $\zD$. Thus $\zuD$ can be interpreted as intermediate between $\zDm$ and $\zD$.
\subsection{A different description of the fat delta}\label{diffat}
We discuss a different description of the fat delta category, which was adopted by \cite{Kock2006} as its definition. We will however not use this alternative description of $\zuD$ in the rest of this work.

We can describe the category $\epid$ in terms of coloured ordinals; first recall the following definition, where we adopt the terminology 'coloured category' instead of 'relative category' for consistency with \cite{Kock2006}.
\begin{definition}{\rm{\cite{BarwiKan2012}}}\label{fatd-def2}
A coloured category consists of a pair $(\clC,we\clC)$ where $\clC$ is a category and $we\clC$ is a wide subcategory (that is, a subcategory containing all the objects of $\clC$). Arrows of $we\clC$ are called coloured arrows. A coloured functor $(\clC,we\clC)\rw (\clD,we\clD)$ is a functor $f:\clC\rw\clD$ that preserves coloured arrows.
\end{definition}
Recall that each $[n]\in\zD$ can be considered a category (a pre-order), which is the ordinal $[n]$. An object $\eta:[n']\rw [n]$ of $\epid$ identifies a wide subcategory of the ordinal $[n']$ with non-identity arrows $i\rw j$ (for $0\leq i<j\leq n'$) if $\eta(i)=\eta(j)$. A morphism in $\epid$ as in \eqref{fatd-eq-1} corresponds to a coloured functor since if $\eta_1(i)=\eta_1(j)$ (with $0\leq i<j\leq n'$) then $\eta_2 f(i)=g \eta_1(i)=g\eta_1(j)=\eta_2 f(j)$.

We call the coloured category corresponding to the epimorphism $\eta:[n']\rw [n]$ a coloured ordinal; the coloured arrows are pictured as links, with the dots labelled from bottom to top, as in the following example.
\begin{example}\label{fatd-ex1}
Let $\eta:[2]\rw [1]$ be the epimorphism $\eta(0)=\eta(1)=0$, $\eta(2)=1$. The coloured category corresponding to this epimorphism is the ordinal $[2]$ with the subcategory with objects $0,1,2$ and a unique non-identity arrow from $0$ to $1$ (since $\eta(0)=\eta(1)$). We call this coloured category 'coloured ordinal' and represent it pictorially as follows:
\begin{equation*}
\begin{tikzpicture}
\node (A) at (0,0) {$2$};
\node (B) at (0,-1) {$1$};
\node (C) at (0,-2) {$0$};
\filldraw(1,0) \pcir (1,-1) \pcir -- (1,-2) \pcir;
\end{tikzpicture}
\end{equation*}
\end{example}
In this picture we do not explicitly draw the non-coloured non-identity arrows (that is, the arrows from $0$ to $2$ and from $1$ to $2$, as these arrows are implicitly represented by the ordering of the labelling of the dots.\bk

Thus $\epid$ can be described as the category of finite non-empty coloured ordinals and colour preserving maps. The graphical representation of morphisms of coloured ordinals is as morphisms of usual ordinals for the dots, but a link can be set and not be broken.

We are going to give a different description of $\zuD$ that builds upon the description of $\epid$ as category of non-empty coloured ordinals. We first recall the notion of semi-category.
\begin{definition}\label{fatd-def4}
A semi-category in a category $\clC$ with pullbacks is a diagram in $\clC$
  \begin{equation*}
  \xymatrix{C_1\tiund{C_0}C_1\ar^(0.6){m}[r]& C_1 \ar@<1ex>^{d_0}[r]\ar@<-1ex>_{d_1}[r] & C_0}
  \end{equation*}
satisfying $d_1p_2=d_1 m$, $d_0p_1=d_0 m$, $m(\Id\tiund{C_0}m)=m(m\,\tiund{C_0}\Id)$.
A semi-functor in $\clC$ is a map of diagrams commuting in the obvious way.
\end{definition}
A semi-category is thus like a category without identities. In particular, a finite semi-ordinal is the semi-category associated to a finite total strict order relation $<$. Since $<$ is not reflexive, there are no identity arrows, and all morphisms between semi-ordinals are injective. We can therefore identify the category of finite non-empty semi-ordinals and order preserving maps with $\zDm$, which is the wide subcategory of $\zD$ containing only the monomorphisms. We also have a notion of coloured semi-category:
\begin{definition}\label{fatd-def5}
A coloured semi-category is a pair $(\clC,we\clC)$ where $\clC$ is a semi-category and $we\clC$ is a semi-subcategory  containing all the objects of $\clC$.  A coloured semi-functor $(\clC,we\clC)\rw (\clD,we\clD)$ is a semi-functor $f:\clC\rw\clD$  preserving coloured arrows.
\end{definition}
Given the morphism \eqref{fatd-eq-2} in $\zuD$, since $f$ is a map in $\zDm$, we can think of $[n']$ and $[m']$ as semi-ordinals. The epimorphisms $\eta_1$ and $\eta_2$ then define coloured semi-categories where $i<j$ (for $0\leq i < j\leq n'$) if $\eta_1(i)=\eta_1(j)$ and similarly for $\eta_2$. We call these coloured semi-categories coloured semi-ordinals. The morphism \eqref{fatd-eq-2} corresponds to a coloured semi-functor: that is, a colour-preserving morphism of coloured semi-ordinals.

In summary we can think of $\zuD$ as the category of non-empty coloured semi-ordinals and colour preserving maps. This description helps the intuition, though we will use Definition \ref{fatd-def3} in this work.

\subsection{Segal maps}\label{segal}
In this section we recall Segal maps for objects of $[\zuDop,\clC]$ where $\clC$ is a category with pullbacks. This notion was already in \cite{Kock2006} but is presented differently from \cite{Kock2006} because we adopted a different definition of fat delta. In particular, we introduce the notation of Definition \ref{segal-def1} for objects of $\zuD$.

\begin{definition}\label{segal-def1}
Given an epimorphism $\eta:[n']\rw [n]$ in $\zD$ and $0\leq j\leq n$ we denote by $\eta^{-1}(j)$ the pre-image of $j$, that is $\eta^{-1}(j)=\{0\leq i\leq n'\,|\,\eta(i)=j\}$. Let $0\leq j_1<j_2<\cdots<j_t\leq n$ be such that $|\eta^{-1}(j_i)|>1$ (for $i=1,\ldots,t$), where $|\eta^{-1}(j_i)|$ denotes the size of the set $\eta^{-1}(j_i)$ and let $n_i=|\eta^{-1}(j_i)|-1$.\smallskip

Given $[n],\,[m]$ in $\zD$ we denote by $[m]\dotp[n]$ the pushout in $\zD$
\begin{equation*}
\xymatrix{
[0]\ar[r]^{0} \ar[d]^{m} & [n]\ar[d] &\\
[m]\ar[r] & [m]\dotp[n]=& \!\!\!\!\!\!\!\!\!\!\!\!\![m+n]
}
\end{equation*}
where the map $[0]\xrw{m} [m]$ sends $0$ to $m$ and the map $[0]\xrw{0} [n]$ sends $0$ to $0$.

Given morphisms $\eta:[n']\rw [n]$, $\mu:[m']\rw [m]$ in $\zD$
\begin{equation*}
\eta\dotp\mu:[n']\dotp[m']=[n'+m']\rw[n]\dotp[m]=[n+m]
\end{equation*}
has components $\eta$ and $\mu$.

\begin{remark}\label{fatdels-rem1}
Given the epimorphism $\eta:[n']\rw[n]$ in $\zD$, let $j_i,n_i$ for $i=1,\ldots,t$ be as in Definition \ref{segal-def1}. Then
\begin{equation*}
[n]\cong[j_1]\dotp[j_2-j_1]\dotp[j_3-j_2]\dotp\cdots\dotp[j_t-j_{t-1}]\dotp[n-j_t]\cong [j_1]\dotp[0]\dotp[j_2-j_1]\dotp[0]\dotp\cdots\dotp[0]\dotp[n-j_t]
\end{equation*}
\begin{equation*}
  [n']\cong[j_1]\dotp[n_1]\dotp[j_2-j_1]\dotp[n_2]\dotp\cdots\dotp[n_t]\dotp[n-j_t]\;.
\end{equation*}
Thus $\eta$ can also be written as
\begin{equation}\label{fatdels-eq1}
\eta=\Id_{[j_1]}\dotp v[n_1]\dotp\Id_{[j_2-j_1]}\dotp v[n_2]\dotp\cdots\dotp v[n_t]\dotp\Id_{[n-j_t]}
\end{equation}
where $v[n_i]:[n_i]\rw [0]$ is as in Section \ref{fatd}.
\end{remark}

Thinking of $\eta$ as a coloured semi-ordinal, $j_i$ identifies the position of the links (as the pre-image of $j_i$ under $\eta$ has more than one element) while $n_i$ is the length of the links.

\bk

 To introduce Segal maps for objects of $[\zuDop,\clC]$ we first recall a preliminary notion, which is well known, about simplicial objects $Y\in[\Dop,\clC]$.
\begin{remark}\label{segal-rem1}
Let $[k]\in\zD$ and $k=k_1+\cdots +k_s$ with $0<k_i<k$. The following diagram commutes in $\zD$:
\begin{equation*}
\xymatrix@C=20pt{
&&& [k] \\
& [k_1] \ar^{\nu_1}[urr] && [k_2] \ar_{\nu_2}[u] && [k_3]\ar_{\nu_3}[ull] && \ & \cdots& [k_s]\ar_{\nu_s}[ullllll]\\
[0]\ar^{\zs_{k_1}}[ur] && [0]\ar_{\zs_0}[ul]\ar^{\zs_{k_2}}[ur] && [0]\ar_{\zs_0}[ul]\ar^{\zs_{k_3}}[ur] && [0]\ar[ul]\ar[ur] & \cdots & [0]\ar[ur] && [0]\ar[ul]^{\zs_0}
}
\end{equation*}
where $\zs_{k_i}(0)=0$, $\zs_0(0)=k_i$, $\nu_1(j)=j$ for $j=0,\ldots,k_1$ and for $1\leq i\leq s$ $\nu_i(j)=k_1+\cdots+k_{i-1}+j$ for $j=0,\ldots,k_i$.
\end{remark}
Thus given a simplicial object $Y\in[\Dop,\clC]$ there is a commuting diagram in $\clC$:
\begin{equation*}
\xymatrix{
&&&& Y_k \ar[dlll]\ar[drrr]\ar[dl]\\
& Y_{k_1}\ar[dl]\ar[dr] && Y_{k_2}\ar[dl]\ar[dr] && \cdots && Y_{k_s}\ar[dl]\ar[dr]\\
Y_0 && Y_0 && Y_0 & \cdots & Y_0 && Y_0
}
\end{equation*}
and a corresponding generalized Segal map
\begin{equation*}
  Y_k \rw Y_{k_1}\tiund{Y_0}Y_{k_2}\tiund{Y_0}\cdots \tiund{Y_0}Y_{k_s}\;.
\end{equation*}
When $k_i=1, \; i=1,\ldots,s$ this coincides with the Segal map of definition \ref{wg-doubcat-def-1}.

We will use this to define Segal maps for $X\in[\zuDop,\clC]$. Such maps have source $X_{\eta}$ where $\eta:[n']\rw [n]$ is an object of $\zuDop$, that is an epimorphism in $\zD$.

Given $X\in[\zuDop,\clC]$ and $\eta\in\zuDop$ as above, we define the Segal map with source $X_{\eta}$ in three steps:
\begin{itemize}
  \item [a)] Let $\eta=h[k]$ with $k\geq 2$ where $h:\zDm \rw \zD$ as in Section \ref{fatd}. Noting that the maps $[1]\rw [k]$ and $[0]\rw [1]$ in Definition \ref{wg-doubcat-def-1} are all maps in $\zDm$, we obtain a unique Segal map
\begin{equation}\label{segal-eq2}
\mu_{h[k]}: X_{h[k]}\rw\pro{X_{h[1]}}{X_{h[0]}}{k}\;.
\end{equation}
Using our convention of identifying $h[k]$ with $[k]$, we write the Segal map simply as
\begin{equation}\label{segal-eq3}
\mu_{k}:X_k\rw \pro{X_1}{X_0}{k}\;.
\end{equation}

  \item [b)] Let $\eta=v[k]$ with $k\geq 2$ where $v:\zDm\rw \zuD$ as in Section \ref{fatd}. There is a commuting diagram in $\clC$
\begin{equation*}
\resizebox{0.85\textwidth}{!}{
\xymatrix{
&&&& X_{v[k]} \ar[dlll]_{\tilde{\nu}_1}\ar[dl]^{\tilde{\nu}_2}\ar[drrr]^{\tilde{\nu}_k} \\
& X_{v[1]}\ar[dl]_{\tilde{d}_1}\ar[dr]^{\tilde{d}_0} && X_{v[1]}\ar[dl]_{\tilde{d}_1}\ar[dr]^{\tilde{d}_0} && \cdots && X_{v[1]}\ar[dl]_{\tilde{d}_1}\ar[dr]^{\tilde{d}_0}\\
X_0 && X_0 && X_0 & \cdots & X_0 && X_0
}}
\end{equation*}
where $\tilde{\nu}_i$, $\tilde{d}_i$, $i=0,1$ correspond to the maps in $\zuD$ given by
\begin{equation*}
\xymatrix{[1]\ar@{^(->}^{\nu_i}[r]\ar@{->>}[d] & [k]\ar@{->>}[d]\\
[0]\ar@{=}[r] & [0]
}
\qquad\qquad
\xymatrix{[0]\ar@{^(->}^{\zs_i}[r]\ar_{\Id}[d] & [1]\ar@{->>}[d]\\
[0]\ar@{=}[r] & [0]
}
\end{equation*}
(with $\nu_i,\;\zs_i$ as in Remark \ref{segal-rem1}) we therefore obtain a Segal map
\begin{equation}\label{segal-eq4}
\mu_{v[k]}:X_{v[k]}\rw \pro{X_{v[1]}}{X_0}{k}\;.
\end{equation}

  \item [c)] We use a) and b) above to tackle the general case of $\eta:[n']\rw [n]$ in $\zuDop$. Let $j_i,n_i$ for $i=1,\ldots,t$ be as in Definition \ref{segal-def1}. Since
\begin{equation*}
\begin{split}
  n' & = j_1+n_1+(j_2-j_1)+n_2+(j_3-j_2)+n_3+\cdots +n_t+(n'-j_t) \\
  n  & = j_1+(j_2-j_1)+(j_3-j_2)+\cdots +(n-j_t)
\end{split}
\end{equation*}
by Remark \ref{segal-rem1}, there are commuting diagrams in $\zD$
\begin{equation*}
\resizebox{0.85\textwidth}{!}{
\xymatrix@C=15pt{
&&&&& [n']\\
& [j_1]\ar[urrrr]^{\nu_{j_1}} && [n_1]\ar[urr]_{\nu_{n_1}} && [j_2-j_1]\ar[u]_{\nu_{j_2-j_1}} && \  & [n_t]\ar[ulll]^{\nu_{n_t}} && [n-j_t]\ar[ulllll]_{\nu_{n-j_t}}\\
[0]\ar[ur]^{\zs_{j_1}} && [0]\ar[ur]^{\zs_{n_1}}\ar[ul]_{\zs_0} && [0]\ar[ur]^{\zs_{j_2-j_1}}\ar[ul]_{\zs_0} && [0]\ar[ur]\ar[ul]^{\zs_0} && \cdots & [0]\ar[ur]\ar[ul] & & [0]\ar[ul]
}}
\end{equation*}
\begin{equation*}
\resizebox{0.85\textwidth}{!}{
\xymatrix@C=15pt{
&&&&& [n]\\
& [j_1]\ar[urrrr]^{\tilde{\nu}_{j_1}} && [j_2-j_1]\ar[urr]_{\tilde{\nu}_{j_2-j_1}} && [j_3-j_2]\ar[u]_{\tilde{\nu}_{j_3-j_2}} && \  & && [n-j_t]\ar[ulllll]_{\tilde{\nu}_{n-j_t}}\\
[0]\ar[ur]^{\zs_{j_1}} && [0]\ar[ur]^{\zs_{j_2-j_1}}\ar[ul]_{\zs_0} && [0]\ar[ur]\ar[ul]_{\zs_0} && [0]\ar[ul] && \cdots & [0]\ar[ur] & & [0]\ar[ul]
}}
\end{equation*}

\nid and a corresponding commuting diagram in $\zuD$:
\begin{equation*}
\resizebox{0.85\textwidth}{!}{
\xymatrix@C=15pt{
&&&&& [\eta]\\
& [j_1]\ar[urrrr] && v[n_1]\ar[urr] && [j_2-j_1]\ar[u] && v[n_2]\ar[ull] & \  && v[n_t]\ar[ulllll] && [n-j_t]\ar[ulllllll]\\
[0]\ar[ur] && [0]\ar[ur]\ar[ul] && [0]\ar[ur]\ar[ul] && [0]\ar[ur]\ar[ul] &&  [0]\ar[ul] & &\cdots & [0]\ar[ul]\ar[ur]
}}
\end{equation*}
where the maps $[j_1]\rw \eta$, $[j_i-ji-1]\rw \eta$ and $v[n_i]\rw \eta$ are given by
\begin{equation*}
\xymatrix{
[j_1]\ar[r]^{\nu_{j_1}}\ar[d]_{\Id} & [n']\ar[d]^{\eta}\\
[j_1] \ar[r]_{\tilde{\nu}_{j_1}} & [n]
}
\qquad\quad
\xymatrix@C=45pt{
[j_i-j_{i-1}]\ar[r]^(0.6){\nu_{j_i-j_{i-1}}}\ar[d]_{\Id} & [n']\ar[d]^{\eta}\\
[j_i-j_{i-1}] \ar[r]_(0.6){\nu'_{j_i-j_{i-1}}} & [n]
}
\qquad\quad
\xymatrix{
[n_i]\ar[r]^{\nu_{n_i}}\ar[d] & [n']\ar[d]^{\eta}\\
[0] \ar[r]_{d_{j_i}} & [n]
}
\end{equation*}
Given $X\in[\zuDop,\clC]$ we therefore obtain a commuting diagram in $\clC$:
\begin{equation*}
\resizebox{0.85\textwidth}{!}{
\xymatrix@R=40pt@C=15pt{
&&&&& X_{\eta}\ar[dllll]\ar[dll]\ar[d]\ar[drr]\ar[drrrrr]\ar[drrrrrrr]\\
& X_{j_1}\ar[dl]\ar[dr] && X_{v[n_1]}\ar[dl]\ar[dr] && X_{j_2-j_1}\ar[dl]\ar[dr] && X_{v[n_2]}\ar[dl]\ar[dr] && \  & X_{v[n_t]}\ar[dl]\ar[dr] && X_{n-j_t}\ar[dl]\\
 X_0 && X_0 && X_0 && X_0 && X_0  &\cdots X_0  && X_0
}}
\end{equation*}
and therefore a Segal map
\begin{equation}\label{segal-eq5}
X_{\eta}\rw X_{j_1}\tiund{X_0}X_{v[n_1]}\tiund{X_0}X_{j_2-j_1}\tiund{X_0}\cdots\tiund{X_0}X_{v[n_t]} \tiund{X_0} X_{n-j_t}
\end{equation}
Combining \eqref{segal-eq5} with \eqref{segal-eq3} and \eqref{segal-eq4} we finally obtain the Segal map
\begin{equation}\label{segal-eq5.1}
\begin{gathered}
\begin{split}
    X_{\eta}\rw &(\pro{X_1}{X_0}{j_1})\tiund{X_0}(\pro{X_{v[1]}}{X_0}{n_1})\tiund{X_0}\\
    &(\pro{X_1}{X_0}{j_2-j_1}) \tiund{X_0} \cdots\tiund{X_0}(\pro{X_{v[1]}}{X_0}{n_t})\\
    &\tiund{X_0}(\pro{X_1}{X_0}{n-j_t})\;.
\end{split}
\end{gathered}
\end{equation}

\end{itemize}

\end{definition}

\section{Fair $2$-categories}\label{fa2cat}

\subsection{Fair $2$-categories}\label{fair2cats}
In defining fair 2-categories we consider $\zuD$ as a coloured category in which the coloured arrows are the ones sent to identities by $\pi:\zuD\rw\zD$, while $\Cat$ is a coloured category in which the coloured arrows are the equivalences of categories.
\begin{definition}\cite{Kock2006}\label{fatdeltaDef-1}
A fair 2-category is a colour-preserving functor $X:\zuDop\rw\Cat$ such that $X_0$ is a discrete category and all the Segal maps \eqref{segal-eq5.1} are isomorphisms. We denote by $\fair{2}$ the category of fair 2-categories.
\end{definition}
Given $X\in\fair{2}$ we sometimes denote $X_0=\clO$, $X_1=\clA$, $X_{v[1]}=\clU$ and call them categories of objects, arrows and weak units respectively. It is shown in \cite{Kock2006} that the two functors $\clU\rightrightarrows\clO$ coincide.
\begin{remark}\label{fair2catRem-1}
  As observed in \cite{Kock2006}*{\S 3.1} to give a fair $2$-category $X$ it is enough to give the following data:
\begin{itemize}
  \item [a)] A discrete category of objects $\clO=X_{0}$, a category of arrows $\clA=X_{1}$ and a category of weak units $\clU=X_{v[1]}$ together with a serially commuting diagram
\begin{equation*}
 \xymatrix@R=40pt @C=40pt{
\clO  &  \clA \ar@<-0.5ex>_{t}[l] \ar@<0.5ex>^{s}[l] \\
\clU \ar@<-0.5ex>[u]_{t'} \ar@<0.5ex>[u]^{s'} \ar_{u}[ur]
}
\end{equation*}
$su=s',\; tu=t'$.\mk

  \item [b)] Semi-category structures (internal to $\Cat$)

  \nid $\xymatrix{\clU\tiund{\clO}\clU\ar[r]&\clU\ar@<-0.5ex>[r] \ar@<0.5ex>[r] & \clO}$ and $\xymatrix{\clA\tiund{\clO}\clA\ar[r]&\clA\ar@<-0.5ex>[r] \ar@<0.5ex>[r] & \clO}$ such that
\begin{equation*}
\xymatrix@R=40pt @C=40pt{
\clU\tiund{\clO}\clU\ar[r]\ar[d] &\clU \ar@<-0.5ex>[r] \ar@<0.5ex>[r] \ar_{u}[d] & \clO \ar@{=}[d] \\
\clA\tiund{\clO}\clA\ar[r]&\clA \ar@<-0.5ex>[r] \ar@<0.5ex>[r] & \clO
}
\end{equation*}
is a semi-functor internal to $\Cat$.\mk

  \item [c)] The maps $\xymatrix{\clU \ar@<-0.5ex>[r] \ar@<0.5ex>[r] & \clO}$ as well as the composition maps
\begin{equation*}
  \clU\tiund{\clO}\clA\rw\clA \lw \clA \tiund{\clO}\clU,\qquad \clU\tiund{\clO}\clU\rw\clU
\end{equation*}
are equivalences of categories. These maps are induced by the maps in $\zuD$ given by
\begin{equation}\label{fatdelta-eq1}
\begin{gathered}
\xymatrix{
[0]\ar@{^(->}^{\zs_i}[r]\ar[d]_{\Id} & [1]\ar@{->>}[d]\\
[0]\ar@{=}[r] & [0]
}
\quad
\xymatrix{
[1]\ar@{^(->}^{\zve}[r]\ar[d]_{\Id} & [2]\ar@{->>}^{\eta_1}[d]\\
[1]\ar@{=}_{\Id}[r] & [1]
}
\quad
\xymatrix{
[1]\ar@{^(->}^{\zve}[r]\ar[d]_{Id} & [2]\ar@{->>}^{\eta_2}[d]\\
[1]\ar@{=}_{\Id}[r] & [1]
}
\quad
\xymatrix{
[1]\ar@{^(->}^{\zve}[r]\ar@{->>}[d] & [2]\ar@{->>}[d]\\
[0]\ar@{=}[r] & [0]
}
\end{gathered}
\end{equation}
where $\zve(0)=0,\:\zve(1)=2,\:\eta_1(0)=\eta_1(1)=0, \:\eta_1(2)=1,\:\eta_2(0)=0,\:\eta_2(1)=\eta_2(2)=1,\: \zs_0(0)=0,\: \zs_{1}(0)=1$.

\end{itemize}
\mk

The rest of the diagram can be constructed from a) and b). Further, as observed in \cite{Kock2006}, the maps \eqref{fatdelta-eq1} in $\zuD$ generate all coloured arrows in $\zuD$.
By the Segal condition and the fact that equivalences of categories are stable under pullbacks over discrete objects, requiring the five maps in c) are equivalences of categories is equivalent to requiring that every coloured map in $\zuD$ is sent to an equivalence of categories.

We also observe that to give a morphism $f:X\rw Y$ in $\fair{2}$ is equivalent to give semi-functors
\begin{equation*}
\xymatrix@R=40pt @C=40pt{
\clA\tiund{\clO}\clA\ar[r]\ar[d] & \clA \ar@<-0.5ex>[r] \ar@<0.5ex>[r] \ar[d] & \clO \ar[d] \\
\clA'\tiund{\clO}\clA'\ar[r] &\clA' \ar@<-0.5ex>[r] \ar@<0.5ex>[r] & \clO'
}
\qquad \text{and}\qquad
\xymatrix@R=40pt @C=40pt{
\clU\tiund{\clO}\clU\ar[r]\ar[d] &\clU \ar@<-0.5ex>[r] \ar@<0.5ex>[r] \ar[d] & \clO \ar[d] \\
\clU'\tiund{\clO'}\clU'\ar[r] & \clU' \ar@<-0.5ex>[r] \ar@<0.5ex>[r] & \clO'
}
\end{equation*}
making the following diagrams commute:
\begin{equation*}
\xymatrix@R=30pt @C=40pt{
{\dblar{\clA}{\clO}} \ar[r] & {\dblar{\clA'}{\clO'}}\\
{\dblar{\clU}{\clO}} \ar[r] \ar[u] & {\dblar{\clU'}{\clO'}}\ar[u]
}
\qquad
\xymatrix@R=30pt @C=40pt{
(\clA\tiund{\clO}\clA\rw\clA) \ar[r]\ar[d] & (\clA'\tiund{\clO'}\clA'\rw\clA')\ar[d]\\
(\clU\tiund{\clO}\clU\rw\clU) \ar[r] & (\clU'\tiund{\clO'}\clU'\rw\clU')\\
}
\end{equation*}
\end{remark}
\begin{lemma}\label{level_fwg_0}
There is a truncation functor
\begin{equation*}
\p{1}:\fair{2}\rw \Cat
\end{equation*}
with $(\p{1}X)_0=pX_0$, $(\p{1}X)_1=pX_1$ where $p:\Cat\rw \Set$ is the isomorphism classes of object functor and we identified $[0]=h[0]$, $[1]=h[1]$.
\begin{proof}
Denote $\clO=X_{0}$,\, $\clA=X_{1}$,\, $\clU=X_{v[1]}$ and let $t,s:\clA\rw\clO$, $u:\clU\rw\clA$ be the structure maps. Since $p$ commutes with pullbacks over discrete objects, $p(\clA\tiund{\clO}\clA)=p\clA\tiund{p\clO}p\clA=p\clA\tiund{\clO}p\clA$. By Remark \ref{fair2catRem-1} there is a semi-category structure internal to $\Cat$
\begin{equation*}
\xymatrix{
 \clA\tiund{\clO}\clA \ar[r]^(0.6)c & \clA \ar@<0.5ex>[r]^t \ar@<-0.5ex>[r]_s & \clO\;.
}
\end{equation*}
Therefore we obtain a semi-category
\begin{equation*}
\xymatrix{
p(\clA\tiund{\clO}\clA)=p\clA\tiund{\clO}p\clA \ar[r]^(0.8){pc} & p\clA \ar@<0.5ex>[r]^{pt} \ar@<-0.5ex>[r]_{ps} & p\clO =\clO\;.
}
\end{equation*}
By Remark \ref{fair2catRem-1} there is a serially commuting diagram in $\Cat$
\begin{equation*}
\xymatrix{
\clO & \clA \ar@<-0.5ex>[l]_t \ar@<0.5ex>[l]^s\\
\clU \ar@<-0.5ex>[u] \ar@<0.5ex>[u] \ar[ru]_u
}
\end{equation*}
where the two (equal) maps $\clU\rightrightarrows \clO$ are equivalences of categories. Therefore we have a map
\begin{equation*}
  pu:\clO\cong p\clU\rw p\clA
\end{equation*}
with $(ps)(pu)=\Id=(pt)(pu)$. By Remark \ref{fair2catRem-1} the composition maps
\begin{equation*}
  \clU\tiund{\clO}\clA \rw \clA \lw \clA\tiund{\clO}\clU
\end{equation*}
are equivalences of categories. Thus, since $p$ commutes with pullbacks over discrete objects, we obtain isomorphisms
\begin{equation*}
  p\clU\tiund{\clO}p\clA=p(\clU\tiund{\clO}\clA)\cong p\clA\cong p(\clA\tiund{\clO}\clU)=p\clA\tiund{\clO}p\clU\;.
\end{equation*}
That is
\begin{equation*}
  pc(\Id_{\clA},(pu)(ps))=\Id_{p\clA}=pc((pu)(pt),\Id_{\clA})\;.
\end{equation*}
In conclusion the maps
\begin{equation*}
\labelmargin-{2pt}
\xymatrix{
(p\clA\tiund{\clO} p\clA) \ar[r]^(0.7){pc} & p\clA \ar@<2ex>[r]^{ps} \ar[r]_{pt} & \clO \ar@<2.5ex>[l]^{pu}
}
\end{equation*}
satisfy the axioms of a category.
\end{proof}

\end{lemma}
\begin{definition}\label{def2eqfr}
Given $a,b\in X_0$, let $X(a,b)$ be the fiber at $(a,b)$ of the map
\begin{equation*}
  X_1\xrw{(\pt_0,\pt_1)}X_0\times X_0\;.
\end{equation*}
A morphism $f:X\rw Y$ in $\fair{2}$ is a $2$-equivalence if
\begin{itemize}
  \item [(i)] For all $a,b\in X_0$, $f_{(a,b)}:X(a,b)\rw Y(fa,fb)$ is an equivalence of categories.

  \item [(ii)] $\p{1}f$ is an equivalence of categories.
\end{itemize}
\end{definition}
\begin{lemma}\label{level_fwg}
Let $F:X\rw Y$ be a morphism in $\fair{2}$ which is levelwise equivalence of categories (i.e. $F_{\eta}$ is an equivalence of categories for all $\eta\in\zuDop$). Then $F$ is a $2$-equivalence.
\end{lemma}
\begin{proof}
Since $X_0,Y_0$ are discrete, $F_0$ is an isomorphism. Thus
\begin{equation*}
  Y_1=\uset{a',b'\in Y_0}{\cop} Y(a',b')\cong \uset{\parbox{12mm}{\centering $\Sc{Fa,Fb}$\\$\Sc{a,b\in X_0}$}}{\cop} Y(Fa,Fb)\;.
\end{equation*}
Since $\ds X_1=\uset{a,b\in X_0}{\cop}X(a,b)$ and $F_1$ is an equivalence of categories, it follows that $X(a,b)\rw Y(Fa,Fb)$ is an equivalence of categories for all $a,b\in X_0$. Also, since $F_n$ is an equivalence of categories for all $[n]\in\zDopm$, $pF_n=(\p{1}F)_n$ is a bijection. Therefore $\p{1}F$ is an isomorphism. By definition, it follows that $F$ is a $2$-equivalence.

\end{proof}


\section{Further properties of the fat delta}\label{fatdels}
In this section we establish some new properties of the fat delta category $\zuD$ that will be used in Section \ref{fairtowg2}. The main results of this section, Proposition \ref{fatdels-pro00} and Proposition \ref{fatdels-pro1}, will be used in the proof of Theorem \ref{fairtowg2-the1}.

\begin{lemma}\label{fatdels-lem1}
Let $\eta:[n_i]\rw [n]$ be an epimorphism in $\zD$ and let $j_t,t_i$ $i=1,\ldots,t$ be as in Definition \ref{segal-def1}. Denote, for each $j\in\{j_1,j_2-j_1,j_3-j_2,\ldots,j_t-j_{t-1},n-j_t\}$, $n\in\{n_1,n_2,\ldots,n_t\}$ and $\mu\in\zuD$:
\begin{equation*}
\begin{split}
    & \zuD([1],\mu)^j=\pro{\zuD([1],\mu)}{\zuD([0],\mu)}{j} \\
    & \zuD(v[1],\mu)^n=\pro{\zuD(v[1],\mu)}{\zuD([0],\mu)}{n}
\end{split}
\end{equation*}
Then there is a bijection
\begin{equation}\label{fatdels-eq2}
\begin{gathered}
\begin{split}
& \zuD(\eta,\mu)\cong \\ &\zuD([1],\mu)^{j_1}\tiund{\zuD([0],\mu)}\zuD(v[1],\mu)^{n_1}\tiund{\zuD([0],\mu)} \zuD([1],\mu)^{j_2-j_1}\tiund{\zuD([0],\mu)}\cdots \tiund{\zuD([0],\mu)} \zuD([1],\mu)^{n-j_t}\;.
\end{split}
\end{gathered}
\end{equation}
\end{lemma}
\begin{proof}
For each $[n], [m]$ in $\zD$ there is a pushout in $\zuD$
\begin{equation*}
\xymatrix{
[0]\ar[r]^{\ovl{0}} \ar[d]_{\ovl{m}} & v[n]\ar[d] &\\
v[m]\ar[r] & v[m]\dotp v[n]=& \!\!\!\!\!\!\!\!\!\!\!\!\!v[m+n]
}
\end{equation*}
where the maps $\ovl{m}$ and $\ovl{0}$ are
\begin{equation*}
\xymatrix{
[0]\ar[r]^{m}\ar@{=}[d] & [m]\ar^{v[m]}[d]\\
[0]\ar@{=}[r] & [0]
}
\qquad\qquad
\xymatrix{
[0]\ar[r]^{0}\ar@{=}[d] & [n]\ar^{v[n]}[d]\\
[0]\ar@{=}[r] & [0]
}
\end{equation*}
where $[0]\xrw{m}[m]$ sends $0$ to $m$ and  $[0]\xrw{0}[n]$ sends $0$ to $0$.
Since, for each $[j]\in\zDm$, $[j]=[1]\dotp\oset{j}{\cdots}\dotp[1]$, it follows by \eqref{fatdels-eq1} that
\begin{equation*}
\begin{split}
  \eta & \cong [j_1]\dotp v[n_1]\dotp[j_2-j_1]\dotp v[n_2]\dotp\cdots\dotp[n-j_t]\cong \\
    & \cong([1]\dotp\oset{j_1}{\cdots}\dotp[1])\dotp(v[1]\dotp\oset{n_1}{\cdots}\dotp v[1])\dotp\cdots([1]\dotp\oset{n-j_t}{\cdots}\dotp[1])\:.
\end{split}
\end{equation*}
Since $\zuD(\mi,\mu)$ sends pushout to pullbacks, \eqref{fatdels-eq2} follows.
\end{proof}
\begin{lemma}\label{fatdels-lem2}
Let $\uf:\mu_1\rw\mu_2$ be the following map in $\zuD$
\begin{equation*}
\xymatrix{[m']\ar@{^(->}^{f'}[r]\ar@{->>}_{\mu_1}[d] & [n']\ar@{->>}^{\mu_2}[d]\\
[m]\ar_{f}[r] & [n]
}
\end{equation*}
Let $\pi\uf=f=\zve\eta$ be the epi-mono factorization of $f$ in $\zD$ . Then there are maps $\ueta$ and $\uzve$ in $\zuD$ with $\pi(\uzve)=\zve$, $\pi(\ueta)=\eta$ and $\uf=\uzve\ueta$.
\end{lemma}
\begin{proof}
If $\eta:[m]\rw[r]$ and $\zve:[r]\rw[n]$ let $[r']$ be the pullback in $\zD$
\begin{equation*}
\xymatrix{[r']\ar@{^(->}^{i_2}[r]\ar@{->>}_{\mu}[d] & [n']\ar@{->>}^{\mu_2}[d]\\
[r]\ar_{\zve}[r] & [n]
}
\end{equation*}
That is, $[r']$ is the full subcategory of the ordinal $[n']$ with objects $i$ such that $\mu_2(i)\in \zve[r]$. Since $\mu_2 f'=f\mu_1$ there is $i_1:[m']\rw[r']$ making the following diagram commute:
\begin{equation*}
\xymatrix{
[m'] \ar@/_/[ddr]_{\eta\mu_1} \ar@/^/[drr]^{f'}
\ar[dr]^{i_1} \\
& [r'] \ar[d]_{\mu} \ar@{^{(}->}[r]^{i_2}
& [n'] \ar[d]^{\mu_2} \\
& [r] \ar@{^{(}->}[r]_{\zve} & [n]
}
\end{equation*}
Since $i_2$ and $f'$ are monos, such is $i_1$. In conclusion, we have maps $\ueta$ and $\uzve$ in $\zuD$
\begin{equation*}
\xymatrix{
[m']\ar@{^{(}->}[r]^{i_1}\ar@{->>}_{\mu_1}[d] & [r']\ar@{^{(}->}[r]^{i_2}\ar^{\mu}[d]& [n'] \ar@{->>}^{\mu_2}[d]\\
[m]\ar_{\eta}[r] & [r]\ar@{^{(}->}_{\zve}[r]& [n]
}
\end{equation*}
with the required properties.
\end{proof}
\begin{remark}\label{fatdels-rem2}
Given an epimorphism $\eta:[n']\rw[n]$ in $\zD$, there is a map $\ueta:[n']\rw\eta$ in $\zuD$ given by
\begin{equation*}
\xymatrix{
[n']\ar^{\Id}[r]\ar_{\Id}[d] & [n']\ar@{->>}^{\eta}[d]\\
[n']\ar_{\eta}[r] & [n]
}
\end{equation*}
with $\pi(\ueta)=\eta$.
\end{remark}
The following lemma will be used in the proof of Proposition \ref{fatdels-pro00}.

\begin{lemma}\label{fatdels-lem3}
Given an epimorphism $\eta:[n']\rw[n]$ and a monomorphism $\zve:[n]\rw [m]$ in $\zD$, there is an epimorphism $\eta':[m']\rw[m]$ and a map in $\zuD$
\begin{equation*}
\xymatrix{
[n']\ar^{\zve'}[r]\ar_{\eta}[d] & [m']\ar^{\eta'}[d]\\
[n]\ar_{\zve}[r] & [m]
}
\end{equation*}
with $m'=n'+m-n$.
\end{lemma}
\begin{proof}
By Remark \ref{fatdels-rem1}, we have
\begin{equation*}
\begin{split}
  [n] & = [j_1]\dotp[j_2-j_1]\dotp[j_3-j_2]\dotp\cdots\dotp[n-j_t] \\
  [n'] & =[j_1]\dotp[n_1]\dotp[j_2-j_1]\dotp[n_2]\dotp\cdots\dotp[n_t]\dotp[n-j_t]\\
  [m] & = [\zve(j_1)]\dotp[\zve(j_2)-\zve(j_1)]\dotp[\zve(j_3)-\zve(j_2)]\dotp\cdots \dotp[m-\zve(j_t)]\;.
\end{split}
\end{equation*}
Given $0\leq i\leq n$ the monomorphism $\zve$ restricts to monomorphisms \mk

$\zve_{0 j_1}:[j_1]\rw[\zve(j_1)]$,\; $\zve_{j_rj_{r+1}}:[j_{r+1}-j_r]\rw[\zve(j_{r+1})-\zve(j_r)]$,\, $r=1,\ldots,t-1$,

$\zve_{j_t n}:[n-j_t]\rw[m-\zve(j_t)]$.\mk

Thus
\begin{equation*}
\zve\cong\zve_{0j_1}\dotp\zve_{j_1j_2}\dotp\zve_{j_2j_3}\dotp\cdots\dotp\zve_{j_t n}\;.
\end{equation*}
Let
\begin{equation*}
\eta'=\Id_{[\zve(j_1)]}\dotp v[n_1]\dotp\Id_{[\zve(j_2)-\zve(j_1)]}\dotp v[n_2]\dotp\cdots\dotp v[n_t]\dotp\Id_{m-\zve(j_t)}
\end{equation*}
and let $\zve'$ be the monomorphism
\begin{equation*}
\zve'=\zve_{0j_1}\dotp\Id_{[n_1]}\dotp\zve_{j_1j_2}\dotp\Id_{[n_2]}\dotp\cdots\dotp\zve_{j_tn}\;.
\end{equation*}
It follows by \eqref{fatdels-eq1} that
\begin{equation*}
\zve \eta=\zve_{0j_1}\dotp v[n_1]\dotp\zve_{j_1j_2}\dotp v[n_2]\dotp\cdots\dotp v[n_t]\dotp\zve_{j_t n}=\zve'\eta'\;.
\end{equation*}
\end{proof}
\begin{proposition}\label{fatdels-pro00}
The map of simplicial sets $N\pi:N\zuD\rw N\zD$ is levelwise surjective.
\end{proposition}
\begin{proof}
$(N\pi)_0$ is surjective since $\Id_{[n]}$ is an object of $\zuD$ for every $[n]\in \zD$. We now show that $(N\pi)_k$ is surjective for each $k\geq 1$ by induction on $k$. First consider the case $k=1$. Let $f\in(N\zD)_1$, so $f:[n_1]\rw[n_2]$ is a map in $\zD$. Let $f:\zve_1\eta_1$ be its epi-mono factorization in $\zD$, where $\eta_1:[n_1]\rw [s_1]$ and $\zve_1:[s_1]\rw [n_2]$. From Remark \ref{fatdels-rem2} and Lemma \ref{fatdels-lem3} there are maps in $\zuD$
\begin{equation*}
\xymatrix{
[n_1]\ar^{\Id}[r]\ar_{\Id}[d] & [n_1]\ar^{\eta_1}[d]\ar@{^{(}->}^{\zve'_1}[r] & [n'_2]\ar^{\eta'_1}[d]\\
[n_1]\ar_{\eta_{1}}[r] & [s_1]\ar_{\zve_1}[r] & [n_2]
}
\end{equation*}
where $n'_2=n_1+n_2-s_1$. The composite is the map $\wt f$ in $\zuD$
\begin{equation*}
\xymatrix{
[n_1]\ar^{\zve'_1}[r]\ar_{\Id}[d] & [n'_2]\ar^{\eta'_1}[d]\\
[n_{1}]\ar_{f}[r] & [n_2]
}
\end{equation*}
with $\pi\wt f=f$.

Suppose, inductively, that the statement holds for $k-1$ and let
\begin{equation}\label{fatdels-eq001}
[k_1]\xrw{f_1}[n_2]\xrw{f_2}\cdots\xrw{f_k}[n_{k+1}]
\end{equation}
be composable arrows in $\zD$, that is an element of $(N\zD)_k$. By inductive hypothesis there is a string of $k-1$ composable arrows in $\zuD$
\begin{equation}\label{fatdels-eq002}
\begin{gathered}
\xymatrix{
[n'_{1}]\ar@{^{(}->}[r]\ar_{\eta'_1}[d] & [n'_2]\ar^{\eta'_2}[d]\ar@{^{(}->}[r] & \cdots\ar@{}|{\textstyle\cdots}[r] & [n'_{k-1}]\ar^{\eta'_{k-1}}[d]\\
[n_1]\ar_{f_1}[r] & [n_2]\ar_{f_2}[r] & \cdots \ar_{f_{k-2}}[r] & [n_{k-1}]
}
\end{gathered}
\end{equation}
Let $f_k=\zve_k\eta_k$ be the epi-mono factorization in $\zD$. Applying Lemma \ref{fatdels-lem3} we obtain maps in $\zuD$
\begin{equation*}
\xymatrix{
[n'_{k-1}]\ar@{}|=[r]\ar_{\eta'_{k-1}}[d] & [n'_{k-1}]\ar^{\eta_k\eta'_{k-1}}[d] \ar@{^{(}->}^{\zve'_k}[r] & [n'_{k}]\ar^{\eta'_k}[d]\\
[n_{k-1}]\ar_{\eta_k}[r] & [s_{k-1}]\ar_{\zve_k}[r] & [n_k]
}
\end{equation*}
where $n'_k=n'_{k-1}+n_k-s_{k-1}$ and $\eta'_k,\zve'_k$ are as in Lemma \ref{fatdels-lem3}. The composite is the map in $\zuD$
\begin{equation}\label{fatdels-eq003}
\begin{gathered}
\xymatrix{
[n'_{k-1}]\ar@{^{(}->}^{\zve'_k}[r] \ar_{\eta'_{k-1}}[d] & [n'_{k}]\ar^{\eta'_k}[d]\\
[n_{k-1}]\ar_{f_k}[r] & n_k
}
\end{gathered}
\end{equation}
In conclusion \eqref{fatdels-eq002} and \eqref{fatdels-eq003} give a string of $k$ composable maps in $\zuD$ (that is an element of $(N\zuD)_k$) which is sent by $\pi$ to the string \eqref{fatdels-eq001}. This proves the inductive step.
\end{proof}
\begin{definition}\label{fatdels-def1}
Let $\eta:[n']\rw[n]$ be an epimorphism in $\zD$ and let $j_i,n_i$ $(i=1,\ldots,t)$ be as in Definition \ref{segal-def1} so that, by Remark \ref{fatdels-rem1},
\begin{equation*}
\eta=\Id_{[j_1]}\dotp v[n_1]\dotp\Id_{[j_2-j_1]}\dotp v[n_2]\dotp\cdots\dotp\Id_{[n-j_t]}\;.
\end{equation*}
Let $\za_{n_i}:[0]\rw[n_i]$ be the map in $\zDm$ which  sends $0$ to $n_i$. Let $\oeta:[n]\rw[n']$ be given by
\begin{equation*}
\oeta=\Id_{[j_1]}\dotp\za_{n_1}\dotp\Id_{[j_2-j_1]}\dotp\za_{n_2}\dotp\cdots\dotp\Id_{[n-j_t]}\;.
\end{equation*}
\end{definition}
\begin{remark}\label{fatdels.rem3}
Since $v[n_i]\za_{n_i}=\Id_{[0]}$ it follows that $\eta\oeta=\Id_{[n]}$.
\end{remark}
\begin{definition}\label{fatdels-def2}
Let $\eta:[n']\rw[n]$ and $\oeta:[n]\rw[n']$ be as in Definition \ref{fatdels-def1}. Denote by $\nu_{\eta}:[n]\rw \eta$ the map in $\zuD$
\begin{equation*}
\xymatrix{
[n]\ar^{\oeta}[r]\ar_{\Id}[d] & [n']\ar^{\eta}[d]\\
[n]\ar_{\Id}[r] & [n]
}
\end{equation*}
\end{definition}
The following two lemmas will be used in the proof of Proposition \ref{fatdels-pro1}.
\begin{lemma}\label{fatdels-lem4}
Let $\eta,\zve,\eta',\zve'$ be as in Lemma \ref{fatdels-lem3} so that the following diagram commutes:
\begin{equation*}
\xymatrix{
[n']\ar^(0.4){\zve'}[r]\ar_{\eta}[d] & [n'+m-n]\ar^{\eta'}[d]\\
[n]\ar_{\zve}[r] & [m]
}
\end{equation*}
Let $\oeta:[n]\rw[n']$ and $\oeta':[m]\rw[n'+m-n]$ be as in Definition \ref{fatdels-def1}. Then $\zve'\oeta=\oeta'\zve.$
\end{lemma}
\begin{proof}
Since, by the proof of Lemma \ref{fatdels-lem3},
\begin{equation*}
\eta'=\Id_{[\zve(j_1)]}\dotp v[n_1]\dotp\Id_{[\zve(j_2)-\zve(j_1)]}\dotp v[n_2]\dotp\cdots\dotp v[n_t]\dotp\Id_{[m-\zve(j_t)]}
\end{equation*}
we have, by Definition \ref{fatdels-def1}
\begin{equation*}
\oeta'=\Id_{[\zve(j_1)]}\dotp\za_{n_1}\dotp\Id_{[\zve(j_2)-\zve(j_1)]}\dotp\za_{n_2} \dotp\cdots\dotp\za_{n_t}\dotp\Id_{[m-\zve(j_t)]}\;.
\end{equation*}
Also, from the proof of Lemma \ref{fatdels-lem3},
\begin{equation*}
\zve=\zve_{0j_1}\dotp\zve_{j_1j_2}\dotp\cdots\dotp\zve_{j_t n}\cong \zve_{0j_1}\dotp\Id_{[0]}\dotp\zve_{j_1j_2}\dotp\Id_{[0]}\dotp\cdots\dotp\zve_{j_t n}\;.
\end{equation*}
Therefore
\begin{equation}\label{fatdels-eq3}
\oeta'\zve=\zve_{0j_1}\dotp\za_{n_1}\dotp\zve_{j_1j_2}\dotp\za_{n_2}\dotp\cdots \dotp\za_{n_t}\dotp\zve_{j_t n}\;.
\end{equation}
On the other hand, by Definition \ref{fatdels-def1} and by the proof of Lemma \ref{fatdels-lem3},
\begin{equation*}
\begin{split}
  \oeta & = \Id_{[j_1]}\dotp\za_{n_1}\dotp\Id_{[j_2-j_1]}\dotp\za_{n_2}\dotp\cdots \dotp\za_{n_t}\dotp\Id_{[n-j_t]}\\
  \zve'  & = \zve_{0j_1}\dotp\Id_{[n_1]}\dotp\zve_{j_1j_2}\dotp\Id_{[n_2]}\dotp\cdots\dotp \Id_{n_t}\dotp\zve_{j_t n}\;.
\end{split}
\end{equation*}
Therefore
\begin{equation}\label{fatdels-eq4}
\zve'\oeta=\zve_{0j_1}\dotp\za_{n_1}\dotp\zve_{j_1j_2}\dotp\za_{n_2}\dotp\cdots\dotp \za_{n_t}\dotp\zve_{j_t n}\;.
\end{equation}
In conclusion, \eqref{fatdels-eq3} and \eqref{fatdels-eq4} imply that $\oeta'\zve=\zve'\oeta$.
\end{proof}
\begin{lemma}\label{fatdels-lem5}
Any morphism in $\zuD$
\begin{equation*}
\xymatrix{
[n']\ar^{\za}[r]\ar_{\eta}[d] & [m']\ar^{\mu}[d]\\
[n]\ar_{\zve}[r] & [m]
}
\end{equation*}
in which $\zve$ is a monomorphism  factors as a composite in $\zuD$
\begin{equation*}
\xymatrix{
[n']\ar^{\zve'}[r]\ar_{\eta}[d] & [n'+m-n]\ar^{\zb}[r]\ar_{\eta'}[d] & [m']\ar^{\mu}[d]\\
[n]\ar_{\zve}[r] & [m]\ar[r]_{\Id} & [m]
}
\end{equation*}
where $\zve'$ and $\eta'$ are as in Lemma \ref{fatdels-lem3}.
\end{lemma}
\begin{proof}
Recall that, by Remark \ref{fatdels-rem1},
\begin{equation*}
\begin{split}
    & \eta=\Id_{[j_1]}\dotp v[n_1]\dotp\Id_{[j_2-j_1]}\dotp v[n_2]\dotp\cdots\dotp v[n_t]\dotp\Id_{[n-j_t]}\\
    & [n']=[j_1]\dotp[n_1]\dotp[j_2-j_1]\dotp[n_2]\dotp\cdots\dotp[n_t]\dotp[n-j_t]\;.
\end{split}
\end{equation*}
If $0\leq r_1< r_2 < \cdots < r_s < m$, $m_i=|\mu^{-1}(r_i)|-1$ are as in Definition \ref{segal-def1} for the epimorphism $\mu$, then
\begin{equation*}
[m']=[r_1]\dotp[m_1]\dotp[r_2-r_1]\dotp[m_2]\dotp\cdots\dotp[r_s]\dotp[m-r_s]\;.
\end{equation*}
Note that, for each $i=1,\ldots,t$
\begin{equation}\label{fatdels-eq5}
\za(\eta^{-1}(j_i))\subset \mu^{-1}(\zve(j_i))
\end{equation}
since if $\ell\in\eta^{-1}(j_i)$, $\eta(\ell)=j_i$ so that $\mu\za(\ell)=\zve\eta(\ell)=\zve(j_i)$, hence $\za(\ell)\in\mu^{-1}(\zve(j_i))$.

Since $|\eta^{-1}(j_i)|>1$ and $\za$ is a monomorphism, \eqref{fatdels-eq5} implies that, for each $i=1,\ldots,t$
\begin{equation*}
|\mu^{-1}(\zve(j_i))|> |\za(\eta^{-1}(j_i))|=|\eta^{-1}(j_i)|>1
\end{equation*}
so that $r_{j_i}=\zve(j_i)$. Then \eqref{fatdels-eq5} implies that, for all $i=1,\ldots,t$
\begin{equation*}
n_i=|\eta^{-1}(j_i)|-1<|\mu^{-1}(\zve(j_i))|-1=|\mu^{-1}(r_{j_i})|-1=m_{j_i}\;.
\end{equation*}
It follows  that $\za$ restricts to a monomorphisms
\begin{equation}\label{fatdels-eq6}
\za_{|_{[n_i]}}:[n_i]\rw [m_{j_i}]\;.
\end{equation}
From the proof of Lemma \ref{fatdels-lem3},
\begin{equation*}
\zve':[n']\rw[n'+m-n]=\zve(j_1)\dotp[n_1]\dotp[\zve(j_2)-\zve(j_1)] \dotp\cdots\dotp[n_t]\dotp[m-\zve(j_t)]
\end{equation*}
is given by
\begin{equation}\label{fatdels-eq7}
\zve'=\zve_{0 j_1}\dotp\Id_{[n_1]}\dotp\zve_{j_1j_2}\dotp\Id_{[n_2]}\dotp\cdots\dotp\zve_{j_t n}\;.
\end{equation}
We define $\zb$ by specifying its restriction on each component of $[n'+m-n]$ as follows for $i=1,\ldots,t$
\begin{equation*}
\begin{split}
    & \zb_{|_{[\zve(j_1)]}}=\za_{|_{[j_1]}} \\
    & \zb_{|_{[\zve(j_i)-\zve(j_{i-1})]}}=\za_{|_{[j_i-j_{i-1}]}} \\
    &  \zb_{|_{[m-\zve(j_{t})]}}=\za_{|_{[n-j_{t}]}}\\
    & \zb_{|_{[n_i]}}=\za_{|_{[n_i]}}\;.
\end{split}
\end{equation*}
By construction, $\zb\zve'$ and $\za$ agree on each component of $[n']$, hence we conclude that $\zb\zve'=\za$. By the proof of Lemma \ref{fatdels-lem3},
\begin{equation*}
\eta'=\Id_{\zve(j_1)}\dotp v[n_1]\dotp\Id_{[\zve(j_2)-\zve(j_1)]}\dotp v[n_2]\dotp\cdots\dotp v[n_t]+\Id_{m-\zve(j_t)}
\end{equation*}
therefore
\begin{equation*}
\mu\zb_{|_{[\zve(j_1)]}}=\mu\za_{|_{[j_1]}} =\zve\eta_{|_{[j_1]}}=\zve\Id_{[j_1]}=\Id_{[\zve(j_1)]}= \eta'_{|_{[\zve(j_1)]}}
\end{equation*}
A similar calculation shows that $\mu\zb$ and $\eta'$ agree on all other components of $[n'+m-n]$ so that in conclusion $\mu\zb=\eta'$, as required.
\end{proof}
\begin{remark}\label{fatdels-rem4}
From Lemma \ref{fatdels-lem5} it follows that the epimorphism $\eta':[m']\rw[m]$ of Lemma \ref{fatdels-lem3} is uniquely characterized by the given property. In fact, suppose there is another map in $\zuD$
\begin{equation*}
\xymatrix{
[n']\ar@{^{(}->}^{\zve''}[r] \ar[d]_{\eta} & [m']\ar[d]^{\eta''}\\
[n]\ar[r]_{\zve} & [n]
}
\end{equation*}
with $m'=n'+m-n$ and $\zve$ a monomorphism. Then by Lemma \ref{fatdels-lem5} this map factors as a composite of maps in $\zuD$
\begin{equation*}
\xymatrix{
[n_1]\ar^(0.4){\zve'}[r] \ar_{\eta}[d] & [n'+m-n]\ar^{\zb}[r]\ar^{\eta'}[d] & [n'+m-n]\ar^{\eta''}[d]\\
[n]\ar_{\zve}[r] & [m]\ar_{\Id}[r] & [m]
}
\end{equation*}
\end{remark}
Since $\zb$ has the same source and target and it is injective, it must be the identity. It follows that $\eta''=\eta'$.
\begin{proposition}\label{fatdels-pro1}
Let $f_1:\za_1\rw\zg_1$ and $f_2:\za_2\rw\zg_2$ be the following maps in $\zuD$ with $\pi f_1=\pi f_2=f$
\begin{equation*}
\xymatrix{
[m_1]\ar@{^{(}->}^{s_1}[r] \ar_{\za_1}[d] & [n_1]\ar^{\zg_1}[d]\\
[m]\ar_{f}[r] & [n]
}
\qquad \qquad
\xymatrix{
[m_2]\ar@{^{(}->}^{s_2}[r] \ar_{\za_2}[d] & [n_2]\ar^{\zg_2}[d]\\
[m]\ar_{f}[r] & [n]
}
\end{equation*}
Let $\eta:[m]\rw[r]$ and $\zve:[r]\rw[n]$ be the epi-mono factorization of $f$ in $\zD$ and let $f_1=\zve_1\eta_1$, $f_2=\zve_2\eta_2$ be the corresponding factorizations as in Lemma \ref{fatdels-lem2}, given by:
\begin{equation*}
\xymatrix{
[m_1]\ar@{^{(}->}^{i_1}[r] \ar_{\za_1}[d] & [r_1]\ar@{^{(}->}^{j_1}[r]\ar^{\zb_1}[d] & [n_1]\ar^{\zg_1}[d]\\
[m]\ar_{\eta}[r] & [r]\ar_{\zve}[r] & [n]
}
\qquad \qquad
\xymatrix{
[m_2]\ar@{^{(}->}^{i_2}[r] \ar_{\za_2}[d] & [r_2]\ar@{^{(}->}^{j_2}[r]\ar^{\zb_2}[d] & [n_2]\ar^{\zg_2}[d]\\
[m]\ar_{\eta}[r] & [r]\ar_{\zve}[r] & [n]
}
\end{equation*}
Then there are maps in $\zuD$
\begin{equation*}
[m]\xrw{\eta_{min}} \eta \xrw{\zve_{min}}\eta'
\end{equation*}
given by
\begin{equation*}
\xymatrix{
[m]\ar^{\Id}[r] \ar_{\Id}[d] & [m]\ar^{\zve'}[r]\ar^{\eta}[d] & [m+n-r]\ar^{\eta'}[d]\\
[m]\ar_{\eta}[r] & [r]\ar_{\zve}[r] & [n]
}
\end{equation*}
(where $\zve'$ and $\eta'$ are as in Lemma \ref{fatdels-lem3}) such that the following diagrams in $\zuD$ commute:
\begin{itemize}
  \item [a)] \
\begin{equation}\label{fatdels-eq11}
\begin{gathered}
\xymatrix{
[m]\ar^{\eta_{min}}[rr]\ar_{\nu_{\za_{_1}}}[d] && \eta\ar_{w_1}[dl] \ar^{w_2}[dr] && [m]\ar_{\eta_{min}}[ll] \ar^{\nu_{\za_{_2}}}[d]\\
\za_1\ar[r]_{\eta_1} & \zb_1 && \zb_2 & \za_2\ar[l]^{\eta_2}
}
\end{gathered}
\end{equation}
where $\nu_{\za_1},\nu_{\za_2}$ are as in Definition \ref{fatdels-def2},  $w_1=(i_1\ovl{\za}_1,\Id_r)$, $w_2=(i_2\ovl{\za}_2,\Id_r)$, with $\ovl{\za_{i}}:[m]\rw [m_i]$ $i=1,2$ as in Definition \ref{fatdels-def1}.

  \item [b)] \
\begin{equation}\label{fatdels-eq12}
\begin{gathered}
\xymatrix{
[r]\ar^{\zve}[r]\ar_{\nu_{\eta}}[d] & [n]\ar^{\nu_{\eta'}}[d]\\
\eta \ar_{\zve_{min}}[r] \ar_{w_1}[d] & \eta'\ar^{z_1}[d]\\
\zb_1 \ar_{(j_1,\zve)}[r] & \zg_1
}
\qquad
\xymatrix{
[r]\ar^{\zve}[r]\ar_{\nu_{\eta}}[d] & [n]\ar^{\nu_{\eta'}}[d]\\
\eta \ar_{\zve_{min}}[r] \ar_{w_2}[d] & \eta'\ar^{z_2}[d]\\
\zb_2 \ar_{(j_2,\zve)}[r] & \zg_2
}
\end{gathered}
\end{equation}
where the map $z_1:\eta'\rw\zg_1$ is given by applying Lemma \ref{fatdels-lem5} to the map in $\zuD$
\begin{equation*}
\xymatrix{
[m] \ar^{s_1\ovl{\za}_1}[r]\ar_{\eta}[d] & [n_1]\ar^{\zg_1}[d]\\
[r]\ar_{\zve}[r] & [n]
}
\end{equation*}
and similarly for the map $z_2:\eta'\rw\zg_2$
\end{itemize}
\begin{proof} \

\begin{itemize}
  \item [a)] The commutativity of the left hand side of \eqref{fatdels-eq11} corresponds to the equality of the following composite maps in $\zuD$:
\begin{equation*}
\xymatrix{
[m]\ar^{\Id}[r]\ar_{\Id}[d] & [m] \ar^{i_1\ovl{\za}_1}[r] \ar^{\eta}[d] & [r_1]\ar^{\zb_1}[d]\\
[m]\ar_{\eta}[r] & [r]\ar_{\Id}[r] & [r]
}
\quad
\xymatrix{
{\phantom{XX}}\ar@{}[d]|{\equiv}\\
{\phantom{XX}}
}
\quad
\xymatrix{
[m]\ar^{\ovl{\za}_{1}}[r]\ar_{\Id}[d] & [m_{1}] \ar^{i_1}[r] \ar^{\za_{1}}[d] & [r_1]\ar^{\zb_1}[d]\\
[m]\ar_{\Id}[r] & [m]\ar_{\eta}[r] & [r]
}
\end{equation*}
The commutativity of the right hand side of \eqref{fatdels-eq11} is proved similarly.
  \mk

  \item [b)] The commutativity of the top squares in \eqref{fatdels-eq12} corresponds to the equality of the following composite maps:
\begin{equation*}
\xymatrix{
[r]\ar^{\ovl{\eta}}[r]\ar_{\Id}[d] & [m] \ar^(0.4){\zve'}[r] \ar^{\eta}[d] & [m+n-r]\ar^{\eta'}[d]\\
[r]\ar_{\Id}[r] & [r]\ar_{\zve}[r] & [n]
}
\quad
\xymatrix{
{\phantom{XX}}\ar@{}[d]|{\equiv}\\
{\phantom{XX}}
}
\quad
\xymatrix{
[r]\ar^{\zve}[r]\ar_{\Id}[d] & [n] \ar^(0.4){\ovl{\eta}'}[r] \ar^{\Id}[d] & [m+n-r]\ar^{\eta'}[d]\\
[r]\ar_{\zve}[r] & [n]\ar_{\Id}[r] & [n]
}
\end{equation*}
where the fact that $\zve'\ovl{\eta}=\ovl{\eta}' \zve$ holds by Lemma \ref{fatdels-lem4}.

The commutativity of the bottom square in the left diagram \eqref{fatdels-eq12} corresponds to the equality of the composite maps:
\begin{equation*}
\xymatrix{
[m]\ar^(0.4){\zve'}[r]\ar_{\eta}[d] & [m+n-r] \ar^(0.6){i}[r] \ar^{\eta'}[d] & [n_1]\ar^{\zg_1}[d]\\
[r]\ar_{\zve}[r] & [n]\ar_{\Id}[r] & [n]
}
\quad
\xymatrix{
{\phantom{XX}}\ar@{}[d]|{\equiv}\\
{\phantom{XX}}
}
\quad
\xymatrix{
[m]\ar^{i_1\ovl{\za}_1}[r]\ar_{\eta}[d] & [r_1] \ar^{j_1}[r] \ar^{\zb_1}[d] & [n_1]\ar^{\zg_1}[d]\\
[r]\ar_{\Id}[r] & [r]\ar[r]_{\zve} & [n]
}
\end{equation*}
where we use the fact that, by construction of the map $z_1:\eta'\rw\zg_1$, it is $i\zve'=s_1\ovl{\za}_1=j_1i_1 \ovl{\za}_1$.
\end{itemize}
\end{proof}
\end{proposition}


\section{From weakly globular double categories to fair $2$-categories}\label{weak2fair} \

In this Section we construct the first half of the comparison between weakly globular double categories and fair $2$-categories, namely we build in Theorem \ref{StrongpseThe-1} a functor $F_2:\catwg{2}\rw \fair{2}.$
We first prove in Proposition \ref{StrongpsePro-1} that the essential image of the functor $\tr{2}:\catwg{2}\rw\segpsc{}{\Cat}$ which is the restriction of $\tr{2}:\tawg{2}\rw\segpsc{}{\Cat}$  (see \eqref{functor_tr2} in Section \ref{sbs-wg-doubcat}) consists of Segalic pseudo-functors such that their restriction to $\Delmono$ is a functor. We call these pseudo-functors strong Segalic pseudo-functors (Definition \ref{StrongpseDef-1}). Given $X\in \catwg{2}$ we can build a semi-category internal to $\Cat$ with object of objects $X_0^d$ and object of arrows $X_1$. The fair $2$-category $F_2X$ has $X_0^d$ as category of objects, $X_1$ as category of arrows and $X_0$ as category of weak units. The rest of the axioms of fair $2$-category for $F_2 X$ are checked using the properties of weakly globular double categories for $X$.

\subsection{Strong Segalic pseudo-functors}\label{Strongpse}
The inclusion functor
\begin{equation*}
  i:\Delmono\rw\Dop
\end{equation*}
induces a functor
\begin{equation*}
  i^*:\psc{}{\Cat}\rw\pscmon{}{\Cat}\;.
\end{equation*}
Since $\segpsc{}{\Cat}\subset \psc{}{\Cat}$ there is also a functor
\begin{equation*}
  i^*:\segpsc{}{\Cat}\rw\pscmon{}{\Cat}\;.
\end{equation*}
\mk

\begin{definition}\label{StrongpseDef-1}
  A Segalic pseudo-functor $X\in \segpsc{}{\Cat}$ is called strong if $i^*X$ is a functor from $\Delmono$ to $\Cat$. A morphism of strong Segalic pseudo-functors is a pseudo-natural transformation $F$ in $\segpsc{}{\Cat}$ such that $i^*F$ is a natural transformation in $[\Delmono,\Cat]$.  We denote by $\Ssegpsc{}{\Cat}$ the category of strong Segalic pseudo-functors, so that
  \begin{equation*}
    i^*:\Ssegpsc{}{\Cat}\rw[\Delmono,\Cat]\;.
  \end{equation*}
\end{definition}
\begin{remark}\label{StrongpseRem-1}\
\begin{itemize}
  \item [a)] We recall that an object $Z$ of $\funcatmon{}{\Cat}$ is a semi-simplicial object in $\Cat$; that is, a sequence of objects $Z_i\in\Cat$ ($i\geq 0$) together with face operators $\pt_i: Z_n\rw Z_{n-1}$ ($i=0,\ldots,n$) satisfying the semi-simplicial identities $\pt_i\pt_j=\pt_{j-1}\pt_{i}$ if $i<j$.\mk

  \item [b)] Recall that a semi-category internal to $\Cat$ consists of a semi-simplicial object $Z\in[\Delmono,\Cat]$ such that the Segal maps $Z_k\rw\pro{Z_1}{Z_0}{k}$ are isomorphisms for all $k\geq 2$.
\end{itemize}

\end{remark}

The following property of weakly globular double categories is crucial for building the functor $F_2$ of Theorem \ref{StrongpseThe-1}.

\begin{proposition}\label{StrongpsePro-1}
   The restriction
  to $\catwg{2}\subset \tawg{2}$ of the functor $\tr{2}:\tawg{2}\rw\segpsc{}{\Cat}$ in \eqref{functor_tr2} is a functor
   \begin{equation*}
   \tr{2}:\catwg{2}\rw\Ssegpsc{}{\Cat}.
   \end{equation*}
\end{proposition}
\begin{proof}
  By definition of strong Segalic pseudo-functor, given $X\in\catwg{2}$ and a morphism $F$ in $\catwg{2}$ we need to show that $i^*\tr{2}X\in\funcatmon{}{\Cat}$ and that  $i^*\tr{2}F$ is a natural transformation in $\funcatmon{}{\Cat}$.

  Let $\pt_i:X_n\rw X_{n-1}$ be the face operators of $X$. By Remark \ref{StrongpseRem-1} a) we need to show that $\pt'_i=\tr{2}\pt_i:(\tr{2}X)_n\rw(\tr{2}X)_{n-1}$ satisfy the semi-simplicial identities $\pt'_i\pt'_j=\pt'_{j-1}\pt'_{i}$ if $i<j$. By construction of $\tr{2}$ \cite{PBook2019}*{Theorem 10.1.1}
  \begin{equation*}
(\tr{2}X)_n =
    \left\{
      \begin{array}{ll}
        X_0^d, & {n=0;} \\
        X_1, & {n=1;} \\
        \pro{X_1}{X_0^d}{n}, & {n\geq 2.}
      \end{array}
    \right.
  \end{equation*}
and $\tr{2}X$ is built from $X$ by transport of structure from the equivalences of categories
\begin{eqnarray*}
 & \zg:X_0\rw X_0^d = (\tr{2}X)_0& \\
 & \Id:X_1\rw X_1 = (\tr{2}X)_1 &\\
 & \hmu_k:X_k=\pro{X_1}{X_0}{k}\rw\pro{X_1}{X_0^d}{k} = (\tr{2}X)_k\quad k\geq 2&
\end{eqnarray*}
where $\hmu_k$ is the $k^{th}$ induced Segal map of $X$. Since $\hmu_k$ is injective on objects, its pseudo-inverse $\nu_k$ satisfies
\begin{equation*}
  \nu_k\hmu_k=id\quad \text{for}\quad k\geq 2\;.
\end{equation*}
By \cite{PBook2019}*{Lemma 4.3.2} the face maps $\pt'_i:(\tr{2}X)_k\rw(\tr{2}X)_{k-1}$ are given as follows:

\begingroup
\setlength\abovedisplayskip{0pt}
\begin{itemize}
  \item [i)] For $k=1$, $i=0,1$
\begin{equation*}
  \pt'_i=\zg\pt_i:X_1\rw X_0^d\;.
\end{equation*}
  \item [ii)]  For $k=2$, $i=0,1,2$
\begin{equation*}
  \pt'_i=\pt_i\nu_2:\tenx{X_1}{X_0^d}\rw X_1\;.
\end{equation*}
  \item [iii)] For $k>2$, $i=0,\ldots,k$
\begin{equation*}
  \pt'_i=\hmu_{k-1}\pt_i\nu_k:\pro{X_1}{X_0^d}{k}\rw\pro{X_1}{X_0^d}{k-1}\;.
\end{equation*}
\end{itemize}
\endgroup

We now verify the semi-simplicial identities for $i^*\tr{2}X$ when $i<j$:
\begingroup
\setlength\abovedisplayskip{0pt}
\begin{itemize}
  \item [a)]
\begin{equation*}
\begin{split}
    & (\tr{2}X)_2 \xrw{\pt'_j} (\tr{2}X)_1 \xrw{\pt'_i} (\tr{2}X)_0\\
    & \pt'_i\pt'_j=\zg\pt_i\pt_j\nu_2=\zg\pt_{j-1}\pt_i\nu_2=\pt'_{j-1}\pt'_i\;.
\end{split}
\end{equation*}
  \item [b)]
\begin{equation*}
\begin{split}
    & (\tr{2}X)_3\xrw{\pt'_j}(\tr{2}X)_2 \xrw{\pt'_i}(\tr{2}X)_1\\
    &  \pt'_i\pt'_j=\pt_i\nu_2\hmu_2\pt_j\nu_3=\pt_i\pt_j\nu_3=\pt_{j-1}\pt_i\nu_3=    \pt_{j-1}\nu_2\hmu_2\pt_i\nu_3=\pt'_{j-1}\pt'_i\;.
\end{split}
\end{equation*}
  \item [c)] For $k>2$
\begin{equation*}
 \begin{split}
     & (\tr{2}X)_{k+1}\xrw{\pt'_j}(\tr{2}X)_k \xrw{\pt'_i}(\tr{2}X)_{k-1} \\
     & \pt'_i\pt'_j=\hmu_{k-1}\pt_i\nu_k\hmu_k\pt_j\nu_{k+1}=\hmu_{k-1}\pt_{i}\pt_j\nu_{k+1}=\\
     & =\hmu_{k-1}\pt_{j-1}\pt_{i}\nu_{k+1}=\hmu_{k-1}\pt_{j-1}\nu_{k}\hmu_k\pt_i\nu_{k+1}= \pt'_{j-1}\pt'_i\;.
 \end{split}
\end{equation*}
\end{itemize}
\endgroup
Thus $\tr{2}X$ satisfies the semi-simplicial identities, hence $i^*\tr{2}X\in\funcatmon{}{\Cat}$.

Let $F:X\rw Y$ be a morphism in $\catwg{2}$. By \cite{PBook2019}*{Lemma 4.3.2} $\tr{2}F$ is given by
\begin{equation*}
  (\tr{2}F)_k=
\left\{
  \begin{array}{ll}
    F_0^d, & k=0 \\
    \Id, & k=1 \\
    \hmu_k(\pro{F_1}{F_0}{k})\nu_k, & k\geq 2\;.
  \end{array}
\right.
\end{equation*}
Using the functoriality of $F$, the definition of $\pt'_i$, the fact that $F_0^d\zg=\zg F_0$ and $\nu_j\hmu_k=\Id$ we see that the following diagrams commute for all $k\geq 2$:
\begin{equation*}
\xymatrix{
X_1 \ar^{\pt'_i}[r] \ar_{F_1}[d] & X_0^d \ar^{F_0^d}[d] \\
Y_1 \ar_{\pt'_i}[r] & Y_0^d
}
\qquad
\xymatrix@C=40pt{
\tenx{X_1}{X_0^d} \ar^(0.6){\pt'_i}[r] \ar_{\nu_2}[d] & X_1 \ar^{\Id}[d] \\
\tenx{X_1}{X_0} \ar^(0.6){\pt_i}[r] \ar_{(F_1\tiund{F_0}F_1)}[d] & X_1 \ar^{F_1}[d] \\
\tenx{Y_1}{Y_0} \ar_(0.6){\pt_i}[r] \ar_{\hmu_2}[d] & Y_1 \ar^{\Id}[d] \\
\tenx{Y_1}{Y_0^d} \ar_(0.6){\pt'_i=\pt_i\nu_2}[r] & Y_1
}
\end{equation*}
\begin{equation*}
\xymatrix{
\pro{X_1}{X_0^d}{k}\ar^{\hmu_{k-1}\pt_i\nu_k}[rr] \ar_{\nu_k}[d] && \pro{X_1}{X_0^d}{k-1} \ar^{\nu_{k-1}}[d]\\
\pro{X_1}{X_0}{k}\ar^{\pt_i}[rr] \ar_{(\pro{F_1}{F_0}{k})}[d] && \pro{X_1}{X_0}{k-1} \ar^{(\pro{F_1}{F_0}{k-1})}[d]\\
\pro{Y_1}{Y_0}{k}\ar^{\pt_i}[rr] \ar_{\hmu_k}[d] && \pro{Y_1}{Y_0}{k-1} \ar^{\hmu_{k-1}}[d]\\
\pro{Y_1}{Y_0^d}{k} \ar_{\hmu_{k-1}\pt_i\nu_k}[rr] && \pro{Y_1}{Y_0^d}{k-1}
}
\end{equation*}
In conclusion, for each $k\geq 0$ and $i=0,\ldots,k+1$ the following diagram commutes
\begin{equation*}
\xymatrix@R=40pt{
(\tr{2}X)_{k+1} \ar^{\pt'_{i}}[rr] \ar_{(\tr{2}F)_{k+1}}[d] && (\tr{2}X)_{k} \ar^{(\tr{2}F)_{k}}[d]\\
(\tr{2}Y)_{k+1} \ar_{\pt'_{i}}[rr]  && (\tr{2}Y)_{k}
}
\end{equation*}
This shows that $i^*\tr{2}F$ is a natural transformation of functors in $\funcatmon{}{\Cat}$.
\end{proof}

\subsection{The functor $F_2$}\label{functor f2}
In this section we prove the existence of a functor $F_2$ from weakly globular double categories to fair $2$-categories that preserves $2$-equivalences. In Proposition \ref{fairtowg2Cor-2} we also compare $X\in\catwg{2}$ and $F_2 X\in \fair{2}$ by suitably modifying $X$ to an object $\tilde\pi^* X\in [\zuDop,\Cat]$ and by constructing an equivalence $S_2(X):F_2(X)\rw\tilde\pi^* X$. These results will be used in Theorem \ref{compar-the1} to establish the equivalence after localization of $\catwg{2}$ and $\fair{2}$.

\begin{theorem}\label{StrongpseThe-1}
There is a functor
\begin{equation*}
  F_2:\catwg{2}\rw \fair{2}
\end{equation*}
with $(F_2 X)_0=X_0^d$, $\p{1}X=\p{1}F_2 X$ and, for each $a,b\in X_0^d$, $X(a,b)\cong (F_2 X)(a,b)$. $F_2$ sends $2$-equivalences in $\catwg{2}$ to $2$-equivalences in $\fair{2}$.
\end{theorem}
\begin{proof}
Let $X\in \catwg{2}$. We use Remark \ref{fair2catRem-1} to build  a fair 2-category $F_2 X$. Define
\begin{equation*}
(F_2 X)_{0}=X_0^d,\qquad (F_1 X)_{1}=X_1,\qquad (F_2 X)_{v[1]}=
X_0\;.
\end{equation*}
There is a commuting diagram
\begin{equation*}
\xymatrix@R=40pt @C=40pt{
X_0^d  &  X_1 \ar@<-0.5ex>_{\zg \pt_0}[l] \ar@<0.5ex>^{\zg \pt_1}[l] \\
X_0 \ar@<-0.5ex>_{\zg}[u] \ar@<0.5ex>^{\zg}[u] \ar_{\zs_0}[ur]
}
\end{equation*}
where $\pt_0, \pt_1:X_1\rw X_0$ (resp. $\zs_0: X_0\rw X_1)$ are the face (resp. degeneracy) operators in $X$.
By Proposition \ref{StrongpsePro-1}, $i^*\tr{2}X \in \funcatmon{}{\Cat}$ with
\begin{equation*}
(i^*\tr{2}X)_k=
\left\{
  \begin{array}{ll}
    X_0^d, & k=0 \\
    X_1, & k=1 \\
    \pro{X_1}{X_0^d}{k}, & k>1\;.
  \end{array}
\right.
\end{equation*}
Thus by Remark \ref{StrongpseRem-1} b), $i^*\tr{2}X$ is a semi-category object internal to $\Cat$. When restricted to  $\xymatrix{X_0 \ar@{^{(}->}^{\zs_0}[r] & X_1}$, this becomes a semi-category object internal to $\Cat$
\begin{equation*}
\xymatrix{
X_0 \tiund{X_0^d}X_0 \ar[r] & X_0 \ar@<0.5ex>^{\zg}[r] \ar@<-0.5ex>_{\zg}[r] & X_0^d\;.
}
\end{equation*}
Recall that the face maps $\tenx{X_1}{X_0^d}\rw X_1$ are given by $\pt_i\nu_2$, where $\pt_i:\tenx{X_1}{X_0}\rw X_1$. Since, when restricted to $X_0=\tenx{X_0}{X_0}\hookrightarrow \tenx{X_1}{X_0}$, $\pt_i=\Id$, we conclude that the face maps $\tenx{X_0}{X_0^d}\rw X_0$ are all equal to $\nu_2$.

Similarly when $k\geq 2$ the face maps $\pro{X_1}{X_0^d}{k+1}\rw \pro{X_1}{X_0^d}{k}$ are given by $\hmu_k\pt_i\nu_{k+1}$. Since, where restricted to
\begin{equation*}
X_0=\pro{X_0}{X_0}{k+1}\hookrightarrow \pro{X_1}{X_0}{k}\;,
\end{equation*}
$\pt_i=\Id$, it follows that all the face maps $\pro{X_0}{X_0^d}{k+1}\rw \pro{X_0}{X_0^d}{k}$ are equal to $\hmu_k \nu_{k+1}$.

We also have a semi-functor
\begin{equation*}
\xymatrix{
\cdots X_0\tiund{X_0^d} X_0 \tiund{X_0^d} X_0 \ar@<1.0ex>[r] \ar@<0.33ex>[r] \ar@<-0.33ex>[r] \ar@<-1.0ex>[r] \ar^{(\zs_0\tiund{\Id}\zs_0\tiund{\Id}\zs_0)}[d] & X_0\tiund{X_0^d} X_0 \ar@<0.66ex>[r] \ar@<0.0ex>[r] \ar@<-0.66ex>[r] \ar^{(\zs_0\tiund{\Id}\zs_0)}[d]   & X_0 \ar@<0.33ex>[r] \ar@<-0.33ex>[r] \ar^{\zs_0}[d] & X^d_0 \ar^{\Id}[d] \\
\cdots X_1\tiund{X_0^d} X_1 \tiund{X_0^d} X_1 \ar@<1.0ex>[r] \ar@<0.33ex>[r] \ar@<-0.33ex>[r] \ar@<-1.0ex>[r] & X_1\tiund{X_0^d} X_1 \ar@<0.66ex>[r] \ar@<0.0ex>[r] \ar@<-0.66ex>[r] & X_1 \ar@<0.33ex>[r] \ar@<-0.33ex>[r]  & X^d_0
}
\end{equation*}
Since $X_0\in\cathd{}$, $\zg:X_0\rw X_0^d$ is an equivalence of categories. By Remark \ref{fair2catRem-1} to prove that $F_2 X\in \fair{2}$ it remains to show that the composition maps
\begin{equation*}
  X_0\tiund{X_0^d}X_0 \rw X_0,\qquad  X_0\tiund{X_0^d}X_1 \rw X_1,\qquad X_1\tiund{X_0^d}X_0 \rw X_1
\end{equation*}
are equivalence of categories. As noted above, all face maps $X_0\tiund{X_0^d}X_0 \rw X_0$ are equal to $\nu_2$, so in particular the composition map is an equivalence of categories.

Consider the commuting diagram
\begin{equation*}
\xymatrix@C=40pt{
X_0\tiund{X_0^d}X_1 \ar^{(\zs_0,\Id)} [r] & X_1\tiund{X_0^d}X_1 \ar^(0.6){c\nu_2}[r] & X_1 \\
X_1=X_0\tiund{X_0}X_1 \ar_{(\zs_0,\Id)} [r] \ar^{\hmu_2}[u] & X_1\tiund{X_0}X_1 \ar_(0.6){c}[r] \ar_{\hmu_2}[u] & X_1 \ar^{\Id}[u]
}
\end{equation*}
Note that the bottom morphism is $c(\zs_0,\Id)=\Id:X_1\rw X_1$. From the commutativity of the above diagram, since $\hmu_2$ is an equivalence of categories, it follows that such is $c\nu_2(\zs_0,\Id): X_0\tiund{X_0^d}X_1 \rw X_1$. The case for the map $c\nu_2(\Id,\zs_0): X_1\tiund{X_0^d}X_0 \rw X_1$ is completely similar. In conclusion, $F_2 X\in\fair{2}$.

If $f:X\rw Y$ is a morphism in $\catwg{2}$, by Proposition \ref{StrongpsePro-1} $i^* \tr{2}f$ is a natural transformation in $\funcatmon{}{\Cat}$. Thus there is a semi-functor internal to $\Cat$ $\clA(f) :\clA(X) \rw \clA(Y)$:
\begin{equation*}
\xymatrix{
\cdots X_1\tiund{X_0^d} X_1 \tiund{X_0^d} X_1 \ar@<1.0ex>[r] \ar@<0.33ex>[r] \ar@<-0.33ex>[r] \ar@<-1.0ex>[r] \ar^{(f_1\tiund{f_0^d}f_1\tiund{f_0^d}f_1)}[d] & X_1\tiund{X_0^d} X_1 \ar@<0.66ex>[r] \ar@<0.0ex>[r] \ar@<-0.66ex>[r] \ar^{(f_1\tiund{f_0^d}f_1)}[d]   & X_1 \ar@<0.33ex>[r] \ar@<-0.33ex>[r] \ar^{f_1}[d] & X^d_0 \ar^{f_0^d}[d] \\
\cdots Y_1\tiund{Y_0^d} Y_1 \tiund{Y_0^d} Y_1 \ar@<1.0ex>[r] \ar@<0.33ex>[r] \ar@<-0.33ex>[r] \ar@<-1.0ex>[r] & Y_1\tiund{Y_0^d} Y_1 \ar@<0.66ex>[r] \ar@<0.0ex>[r] \ar@<-0.66ex>[r] & Y_1 \ar@<0.33ex>[r] \ar@<-0.33ex>[r]  & Y^d_0
}
\end{equation*}
which restricts to a semi-functor $\clU(X)\rw \clU(Y)$ internal to $\Cat$:
\begin{equation*}
\xymatrix{
\cdots X_0\tiund{X_0^d} X_0 \tiund{X_0^d} X_0 \ar@<1.0ex>[r] \ar@<0.33ex>[r] \ar@<-0.33ex>[r] \ar@<-1.0ex>[r] \ar^{}[d] & X_0\tiund{X_0^d} X_0 \ar@<0.66ex>[r] \ar@<0.0ex>[r] \ar@<-0.66ex>[r] \ar^{}[d]   & X_0 \ar@<0.33ex>[r] \ar@<-0.33ex>[r] \ar^{}[d] & X^d_0 \ar^{}[d] \\
\cdots Y_0\tiund{Y_0^d} Y_0 \tiund{Y_0^d} Y_0 \ar@<1.0ex>[r] \ar@<0.33ex>[r] \ar@<-0.33ex>[r] \ar@<-1.0ex>[r] & Y_0\tiund{Y_0^d} Y_0 \ar@<0.66ex>[r] \ar@<0.0ex>[r] \ar@<-0.66ex>[r] & Y_0 \ar@<0.33ex>[r] \ar@<-0.33ex>[r]  & Y^d_0
}
\end{equation*}
making the following diagram in $\funcatmon{}{\Cat}$ commute:
\begin{equation*}
\xymatrix{
\clA(X) \ar[r]^{\clA(f)} & \clA(Y) \\
\clU(X) \ar[u] \ar[r]^{\clU(f)} & \clU(Y) \ar[u]
}
\end{equation*}
By Remark \ref{fair2catRem-1} it follows that $F_2 f$ is a morphism in $\fair{2}$.

By construction, $(\p{1}F_2 X)_k=(\p{1}X)_k$ for all $k\in \zDm^{op}$, so that $\p{1}F_2 X=\p{1}X$. Also, for all $a,b\in X_0^d$, $(F_2 X)(a,b)=X(a,b)$. It follows that a $2$-equivalence in $\catwg{2}$ is sent by $F_2$ to a $2$-equivalence in $\fair{2}$.
\end{proof}
\begin{remark}\label{fairtowg2Rem-01}

Let $\eta:[n']\rw[n]$ be an epimorphism in $\zD$ (hence an object of $\zuD)$ and let $j_i,n_i\; (i=1,\cdots,t)$ be as in Definition \ref{segal-def1}; then there are Segal maps \eqref{segal-eq5.1}. By definition of $F_2 X$ we have
\begin{equation}\label{fairtowg2eq-01}
\begin{gathered}
\begin{split}
 &( F_2 X)_{\eta}  \cong (\pro{X_1}{X_0^d}{j_1})\tiund{X_0^d}(\pro{X_0}{X_0^d}{n_1})\tiund{X_0^d}\\
 & (\pro{X_1}{X_0^d}{j_2-j_1})\tiund{X_0^d}\cdots\tiund{X_0^d} (\pro{X_0}{X_0^d}{n_t})\tiund{X_0^d} (\pro{X_1}{X_0^d}{n-j_t})
\end{split}
\end{gathered}
\end{equation}
We next want to relate $F_2 X$ and $X$. For this purpose, we first note that $X\in \catwg{2}\subset \funcat{}{\Cat}$ gives rise to an object of $[\zuDop,\Cat]$ closely related to $X$, as illustrated in the following definitions and lemma.
\end{remark}

\begin{definition}\label{fairtowg2Def-1}
  Let ${\pi^*}:\funcat{}{\Cat}\rw[\zuDop,\Cat]$ be induced by the map $\pi:\zuDop\rw\Dop$ of Section \ref{fatd}. That is, for each $X\in\funcat{}{\Cat}$ and $\eta:[n']\rw[n]$ in $\Dop$
  \begin{equation*}
    (\pi^* X)_{\eta}=X_{\pi(\eta)}=X_n
  \end{equation*}
\end{definition}
Since $\catwg{2}\hookrightarrow\funcat{}{\Cat}$, given $X\in\catwg{2}$, $\pi^*X\in[\zuDop,\Cat]$. In the next definition, we introduce a modification $\tilde{\pi}^*X$ of $\pi^* X$ that will lead in Proposition \ref{fairtowg2Cor-2} to a natural transformation $F_2 X\rw \tilde{\pi}^* X$.
\begin{definition}\label{fairtowg2Def-2}
  Let ${\tilde\pi^*}:\catwg{2}\rw[\zuDop,\Cat]$ be given by
\begin{equation*}
(\tilde\pi^* X)_{\eta}=
\left\{
  \begin{array}{ll}
    (\pi^*X)_{\eta}, & \text{if }\eta\neq 0 \\
    X_0^d, &\text{if }\eta = 0\;.
  \end{array}
\right.
\end{equation*}
where the maps $(\tilde\pi^* X)_{1}=X_1\rightrightarrows (\tilde\pi^* X)_{0}=X_0^d$ are $\zg\pt_i\;\, i=0,1$ and the maps $(\tilde\pi^* X)_{v[1]}=X_0 \rightrightarrows (\tilde\pi^* X)_{0}=X_0^d$ are both equal to $\zg$. All other maps $(\tilde\pi^* X)_{\eta}\rw(\tilde\pi^* X)_{\mu}$ corresponding to maps $\eta\rw\mu$ in $\zuDop$ are equal to the maps $(\pi^* X)_{\eta}\rw(\pi^* X)_{\mu}$.
\end{definition}
\begin{remark}\label{fairtowg2Rem-2}\
We note that $\tilde\pi^* X$ can be obtained from $\pi^* X$ by transport of structure along the equivalences of categories $(\tilde\pi^* X)_{\eta} \simeq (\pi^* X)_{\eta}$ given by $\gamma':X_0^d \rw X_0$ for $\eta=\Id_{0}$ (where $\gamma'$ is the pseudo-inverse to $\gamma$) and $\Id$ for $\eta\neq \Id_{[0]}$. Therefore by Lemma \ref{lem-PP} there is a pseudo-natural transformation $\tilde\pi^*X\rw \pi^* X$ in $\Ps\fatcat{\zuDop}{\Cat}$ which is a levelwise equivalence of categories.
\end{remark}
\begin{lemma}\label{fairtowg2Rem-1}\

\begin{itemize}
  \item [a)] Let $\eta:[n']\rw[n]$ be an epimorphism $\zD$, different from $\Id_{[0]}$ and let $ X\in\catwg{2}$. There is an injective equivalence of categories $X_n\rw (F_2 X)_{\eta}$.

  \item [b)] Let $\eta:[n']\rw[n]$ be an epimorphism $\zD$ and let $X\in\catwg{2}$. There is an equivalence of categories $z_{\eta}(X):(\tilde{\pi}^* X)_{\eta}\rw(F_2X)_{\eta}$.

\end{itemize}
\end{lemma}
\begin{proof} Let $j_i,n_i$ $(i=1,\ldots,t)$ be as in Definition \ref{segal-def1} for $\eta$.
\begin{itemize}
  \item [a)] The induced Segal map condition for $X$ (see Definition \ref{wg-doubcat-def-4} c)) gives injective equivalence of categories
\begin{equation*}
\begin{split}
   X_n & \cong (\pro{X_1}{X_0}{j_1})\tiund{X_0}(\pro{X_0}{X_0}{n_1})\tiund{X_0} (\pro{X_1}{X_0}{j_2-j_1})\tiund{X_0}\cdots  \\
    &\cdots \tiund{X_0}(\pro{X_1}{X_0}{n-j_t})\rw\\
   & \rw (\pro{X_1}{X_0^d}{j_1})\tiund{X_0^d}(\pro{X_0}{X_0^d}{n_1})\tiund{X_0^d}\cdots\\
    &\cdots\tiund{X_0^d}(\pro{X_1}{X_0^d}{n-j_t})=F_2 X
\end{split}
\end{equation*}
where the isomorphisms on the left hand side holds since $\pro{X_0}{X_0}{n_i}\cong X_0$ for $i=1,\ldots,t$ and the isomorphism on the right hand side holds by \eqref{fairtowg2eq-01}.\mk

  \item [b)] By construction $(F_2 X)_{[0]}=X_0^d=(\tilde{\pi}^* X)_{0}$ so we can take $z_0=\Id$ and when $\eta$ is not $\Id_{[0]}$, $z_{\eta}(X)$ is as in a) (as $(\wt\pi^* X)_{\eta}=(\pi^* X)_n=X_n$ when $\eta\neq 0$.
\end{itemize}
\end{proof}

The following proposition, together with Theorem \ref{StrongpseThe-1}, will be used in the proof of Lemma \ref{fairtowg2-lem1}, leading to the main result Theorem \ref{compar-the1}.

\begin{proposition}\label{fairtowg2Cor-2}
  Let $F_2:\catwg{2}\rw\fair{2}$ be as in Theorem \ref{StrongpseThe-1} and $\tilde\pi^*$ as in Definition \ref{fairtowg2Def-2}. There is a natural transformation $S_2(X):F_2(X)\rw\tilde\pi^* X$ in $[\zuDop,\Cat]$ which is a levelwise equivalence of categories.
\end{proposition}
\begin{proof}
For each $\eta:[n']\rw [n]$ in $\zuDop$ let $(S_2(X))_{\eta}$ be the pseudo-inverse to the equivalence of categories $z_{\eta}(X)$ in Lemma \ref{fairtowg2Rem-1}, so that $(S_2(X))_{\eta}$ is itself an equivalence of categories.

Let $\uf:\eta\rw \mu$ be the map in $\zuDop$
\begin{equation*}
\xymatrix{
[n']\ar@{^{(}->}^{i}[r]\ar_{\eta}[d] & [m']\ar^{\mu}[d]\\
[n]\ar_{f}[r]& [m]
}
\end{equation*}
We claim that $F_2\uf$ is  the composite
\begin{equation}\label{fairtowg2Eq-3A}
(F_2 X)_{\eta}\xrw{(S_2 X)_{\eta}}(\tilde\pi_* X)_{\eta}\xrw{\tilde\pi_* \uf} (\tilde\pi_* X)_{\mu} \xrw{z_{\mu}(X)}(F_2 X)_{\mu}\;.
\end{equation}
In fact, by the construction of $F_2$ in the proof of Theorem \ref{StrongpseThe-1}, $F_2 X$ is determined by the following maps:
\begin{itemize}
  \item [a)]  The maps giving the semi-category structures
  \begin{equation*}
    (F_2 X)_{2}=\tenx{X_1}{X_0^d} \rw X_1=(F_2 X)_{1}\rightrightarrows X_0^d=(F_2 X)_{0}
  \end{equation*}
\begin{equation*}
  (F_2 X)_{v[2]}=\tenx{X_0}{X_0^d} \rw X_0=(F_2 X)_{v[1]}\rightrightarrows X_0^d=(F_2 X)_{0}
\end{equation*}
which are given by the composites
\begin{equation*}
  \tenx{X_1}{X_0^d}\xrw{\nu_2} \tenx{X_1}{X_0}\xrw{c}X_1
\end{equation*}
\begin{equation*}
  \tenx{X_0}{X_0^d}\xrw{\nu_2} \tenx{X_0}{X_0}=X_0
\end{equation*}

  \item [b)] The maps
\begin{equation*}
  (F_2 X)_{v[1]\dotp [1]}=X_0\tiund{X_0^d}X_1\rw X_1=(F_2 X)_{1}
\end{equation*}
\begin{equation*}
  (F_2 X)_{[1]\dotp v[1]}=X_1\tiund{X_0^d}X_0\rw X_1=(F_2 X)_{1}
\end{equation*}
which are given as composites
\begin{equation*}
X_0\tiund{X_0^d}X_1\xrw{\nu_2}X_0\tiund{X_0}X_1=X_1
\end{equation*}
\begin{equation*}
X_1\tiund{X_0^d}X_0\xrw{\nu_2}X_1\tiund{X_0}X_0=X_1\;.
\end{equation*}
We see that the maps in a) and b) are of the form \eqref{fairtowg2Eq-3A}. Since all other maps $F_2 f:(F_2 X)_{\eta}\rw(F_2 X)_{\mu}$ are determined by these, they also are of the form \eqref{fairtowg2Eq-3A}. This proves the claim, so that
\begin{equation*}
  (F_2 X)(\uf)=z_{\mu}(X)(\tilde\pi_* \uf)(S_2 X)_{\eta}\;.
\end{equation*}
Since by Lemma \ref{fairtowg2Rem-1} $z_{\eta}(X)$ is an injective equivalence, $(S_2 X)_{\eta} z_{\eta}(X)=\Id$ for all $\eta\in\zuDop$ so that
\begin{equation*}
  (S_2 X)_{\mu}(F_2 X)(\uf)=(S_2 X)_{\mu}z_{\mu}(X)(\tilde\pi_*\uf)(S_2 X)_{\eta}=(\tilde\pi_*\uf)(S_2 X)_{\eta}\;.
\end{equation*}
 Therefore the following diagram commutes
\begin{equation*}
\xymatrix@C=40pt@R=35pt{
(F_2 X)_{\eta} \ar[r]^{F_2 X(\uf)} \ar[d]_{(S_2 X)_{\eta}} & (F_2 X)_{\mu} \ar[d]^{(S_2 X)_{\mu}}\\
(\tilde\pi_* X)_{\eta} \ar[r]_{\tilde\pi_* \uf} & (\tilde\pi_* X)_{\mu}
}
\end{equation*}
This shows that $S_2 X: F_2 X\rw \tilde\pi_* X$ is a natural transformation in $[\zuDop,\Cat]$.
\end{itemize}

\end{proof}

\section{Weakly globular fair $2$-categories}\label{wgfair2}

In this section we introduce a new player, the category $\fairwg{2}$ of weakly globular fair 2-categories. This structure will be needed in the proof of our main comparison result Theorem \ref{compar-the1}. We show in Theorem \ref{fairtowg2The-1} that weakly globular fair 2-categories arise as strictification of Segalic pseudo-functors from $\zuDop$ to $\Cat$, which we introduce in Definition \ref{segpseu-def1}. We also show in Lemma \ref{fairtowg2Lem-6} the existence of a functor $D:\fairwg{2}\rw \fair{2}$ which preserves 2-equivalences. Since the constructions and results of this section will not be used until the proof of Theorem \ref{compar-the1}, this section may be skipped at first reading.

\subsection{Segalic pseudo-functors from $\zuDop$ to $\Cat$}\label{segpseu}
Let $\eta:[n']\rw [n]$ be an object of $\zuD$ and $j_i,n_i\;(i=1,\ldots,t)$ be as in Definition \ref{segal-def1}. Let $H\in\Ps[\Dop,\Cat]$ be such that $H_0$ is discrete. This discreteness condition implies that there are commuting diagrams as in Section \ref{segal} (even if $H$ is a pseudo-functor rather than a functor). Therefore there is a Segal map (similar to \eqref{segal-eq5.1} in Section \ref{segal})
\begin{equation}\label{segpseu-eq1}
    H_{\eta}\rw H_1^{j_1} \tiund{H_0}H^{n_1}_{v[1]}\tiund{H_0} H_1^{j_2-j_1}
   \tiund{H_0}\cdots \tiund{H_0}H_{v[1]}^{n_t}\tiund{H_0}H_1^{n-j_t}\;,
\end{equation}
where we denoted $H_1^k=\pro{H_1}{H_0}{k}$ and $H^k_{v[1]}=\pro{H_{v[1]}}{H_0}{k}$.
\begin{definition}\label{segpseu-def1}
The category $\segps[\zuDop,\Cat]$ of Segalic pseudo-functors from $\zuDop$ to $\Cat$ is the full subcategory of $\Ps[\zuDop,\Cat]$ whose objects $H$ are such that
\begin{itemize}
  \item [i)] $H_0$ is discrete.

  \item [ii)] For each $\eta\in\zuDop$ the Segal map \eqref{segpseu-eq1} is an isomorphism.

  \item [iii)] The maps
$$H_{v[1]}\rightrightarrows H_0,\,\; H_{v[1]}\tiund{H_0}H_1\rw H_{v[1]}\lw H_1\tiund{H_0}H_{v[1]},\,\; H_{v[1]}\tiund{H_0}H_{v[1]}\rw H_{v[1]}$$

\nid which are images of the maps \eqref{fatdelta-eq1} (see Remark \ref{fair2catRem-1} c)) are equivalences of categories.

\end{itemize}
\end{definition}
\begin{remark}\label{segpseu-rem1}
From the definitions, if $X\in\fair{2}$, then $X\in \segps\fatcat{\zuDop}{\Cat}$. In fact the inclusion $\fatcat{\zuDop}{\Cat}\subset \Ps\fatcat{\zuDop}{\Cat}$ restricts to the inclusion $\fair{2}\subset \segps\fatcat{\zuDop}{\Cat}$. This is analogous to the fact that the inclusion $\funcat{}{\Cat}\subset \Ps\funcat{}{\Cat}$ restricts to the inclusion $2-\Cat\subset\segps\funcat{}{\Cat}$.

Recall \cite{PW} that the functor $2$-category $[\zuDop,\Cat]$ is 2-monadic over $[\ob(\zuDop),\Cat]$. Let $U:[\zuDop,\Cat]\rw[ob(\zuDop),\Cat]$ be the forgetful functor; then its left adjoint is
\begin{equation*}
  (FY)_{\eta}=\uset{\mu\in ob(\zuDop)}{\cop}\zuDop (\mu,\eta)\times Y_{\mu}
\end{equation*}
for $Y\in[ob(\zuDop),\Cat]$, $\eta\in\zuDop$.

Let $T$ the monad corresponding to the adjunction $F\dashv U$. Then the pseudo $T$-algebra corresponding to $H\in\Ps[\zuDop,\Cat]$ has structure map $h:TUH\rw UH$ as follows. Denoting
\begin{equation*}
  (TUH)_{\eta}=\uset{\mu\in\zuD}{\cop}\zuD(\eta,\mu)\times H_{\mu}=\uset{\mu\in\zuD}{\cop}\;\uset{\zuD(\eta,\mu)}{\cop}H_{\mu}\;.
\end{equation*}
and, if $f\in\zuD(\eta,\mu)$, denoting
\begin{equation*}
  i_{\mu}:\uset{\zuD(\eta,\mu)}{\cop} H_{\mu} \rw \uset{\mu\in\zuD}{\cop}\; \uset{\zuD(\eta,\mu)}{\cop} H_{\mu} = (TUH)_{\mu},\qquad j_f:H_{\mu}\rw \uset{\zuD(\eta,\mu)}{\cop} H_{\mu}
\end{equation*}
the coproduct inclusions, then the map $h$ is the unique map satisfying
\begin{equation*}
  h_{\eta}i_{\mu}j_f=H(f)\;.
\end{equation*}
\end{remark}
\begin{lemma}\label{segpseu-lem1}
  Let $U,T,F$ be as above. Let $H\in\segps\fatcat{\zuDop}{\Cat}$. Then
\begin{itemize}
\item [a)] There are functors
\begin{equation*}
\pt'_0,\pt'_1:(TUH)_{1}\rightrightarrows (TUH)_0, \qquad \wt{\pt}_0,\wt{\pt}_1:(TUH)_{v[1]}\rightrightarrows (TUH)_0
\end{equation*}
making the following diagrams commute
\begin{equation}\label{segpseu-eq1a}
\begin{gathered}
\xymatrix@C=50pt@R=35pt{
(TUH)_{1}\ar[r]^{h_{1}}\ar@<-0.5ex>[d]_{\pt'_{0}} \ar@<0.5ex>[d]^{\pt'_{1}} & (UH)_{1} \ar@<-0.5ex>[d]_{\pt_{0}} \ar@<0.5ex>[d]^{\pt_{1}}\\
(TUH)_0 \ar[r]_{h_0} & (UH)_0
}
\qquad
\xymatrix@C=50pt@R=35pt{
(TUH)_{v[1]}\ar[r]^{h_{v[1]}}\ar@<-0.5ex>[d]_{\wt{\pt}_{0}} \ar@<0.5ex>[d]^{\wt{\pt}_{1}} & (UH)_{v[1]} \ar@<-0.5ex>[d]_{\pt_{0}} \ar@<0.5ex>[d]^{\pt_{1}}\\
(TUH)_0 \ar[r]_{h_0} & (UH)_0
}
\end{gathered}
\end{equation}

\item [b)]

Let $\eta:[n']\rw[n]$ be in $\zuD$ and let $j_i,n_i\;(i=1,\ldots,t)$ be as in Definition \ref{segal-def1}. Then
\begin{equation*}
(TUH)_{\eta}\cong(TUH)_1^{j_1}\tiund{(TUH)_0}(TUH)_{v[1]}^{n_1}\tiund{(TUH)_0} \cdots\tiund{(TUH)_0}(TUH)_1^{n-j_t}\;,
\end{equation*}
where we denoted
\begin{equation*}
\begin{split}
  (TUH)^k_1 & = \pro{(TUH)_1}{(TUH)_0}{k} \\
   (TUH)^k_{v[1]} & = \pro{(TUH)_{v[1]}}{(TUH)_{0}}{k}\;.
\end{split}
\end{equation*}
\item [c)] Let $\eta\in\zuDop$ be as in b). Then the morphism $h_{\eta}:(TUH)_{\eta}\rw(UH)_{\eta}$ is given by:
\begin{equation*}
h_1^{j_1}\times h^{n_1}_{v[1]}\times  h_1^{j_2-j_1}\times\cdots\times  h_1^{n-j_t}
\end{equation*}
where we denoted:
\begin{equation*}
h_1^k=\pro{h_1}{h_0}{k},\qquad h^k_{v[1]}=\pro{h_{v[1]}}{h_0}{k}\;.
\end{equation*}
\end{itemize}
\end{lemma}
\begin{proof}\

\begin{itemize}
  \item [a)] Let $\zs_i:[0]\rw[1]$, \; $\zst_i:[0]\rw v[1]$ ($i=0,1$) be as in Section \ref{segal}. Let
\begin{equation*}
\pt'_i:(TUH)_1\rw(TUH)_0,\qquad \wt\pt_i:(TUH)_{v[1]}\rw(TUH)_0\quad (i=0,1)
\end{equation*}
\nid be the functors determined by
\begin{equation*}
  \pt'_i i_{\mu}j_f=i_{\mu}j_{f\zs_i},\quad \ptt_i i_{\mu}j_f=i_{\mu}j_{f\zst_i}\quad\text{when }f\in\zuD(\eta,\mu)\;.
\end{equation*}
Then
\begin{equation*}
\begin{split}
    & h_0\pt'_i i_{\mu}j_f=h_0 i_{\mu}j_{f\zs_{i}}=H(f\zs_i),\qquad h_0\ptt_i i_{\mu}j_f=h_0 i_{\mu}j_{f\zst_{i}}=H(f\zst_i)\\
    & \pt_i h_1 i_{\mu}j_f=H(\zs_i)H(f),\qquad\qquad\quad \ptt_i h_1 i_{\mu}j_f=H(\zst_i)H(f)\;.
\end{split}
\end{equation*}
Since $H\in\Ps[\zuDop,\Cat]$ and $H_0$ is discrete, it follows that
\begin{equation*}
H(f\zs_i)=H(\zs_i)H(f),\qquad H(f\zst_i)=H(\zst_i)H(f)
\end{equation*}
so that, from above
\begin{equation*}
h_0\pt'_i i_{\mu}j_f=\pt_i h_1 i_{\mu}j_f, \qquad h_0\ptt_i i_{\mu}j_f=\ptt_i h_{v[1]} i_{\mu}j_f\;.
\end{equation*}
We conclude that
\begin{equation*}
h_0\pt'_i=\pt_i h_1,\qquad h_0\ptt_i=\ptt_i h_{v[1]}\;.
\end{equation*}

  \item [b)] From the proof of a), the functors
\begin{equation*}
\pt'_i:(TUH)_1\rw (TUH)_0,\qquad \ptt_i:(TUH)_{v[1]}\rw (TUH)_0
\end{equation*}
for $i=0,1$ are determined by the functors
\begin{equation*}
\begin{split}
    & (\ovl{\zs}_i,\Id):\zuD([1],\mu)\times H_{\mu}\rw \zuD([0],\mu)\times H_{\mu} \\
    & (\ovl{\zst}_i,\Id):\zuD(v[1],\mu)\times H_{\mu}\rw \zuD([0],\mu)\times H_{\mu}
\end{split}
\end{equation*}
where $\ovl{\zs}_i(g)=g\zs_i$ for $g\in\zuD([1],\mu)$ and $\ovl{\zst}_i(\tilde{g})=\tilde{g}\zst_i$ for $\tilde{g}\in\zuD(v[1],\mu)$. Further.
\begin{equation}\label{segpseu-eq2}
(TUH)_1=\uset{\mu\in\zuD}{\cop}\zuD([1],\mu)\times H_{\mu}\;,
\end{equation}
\begin{equation}\label{segpseu-eq3}
(TUH)_{v[1]}=\uset{\mu\in\zuD}{\cop}\zuD(v[1],\mu)\times H_{\mu}\;,
\end{equation}
\begin{equation}\label{segpseu-eq4}
(TUH)_0=\uset{\mu\in\zuD}{\cop}\zuD([0],\mu)\times H_{\mu}\;.
\end{equation}
This implies
\begin{equation}\label{segpseu-eq5}
\uset{\mu\in\zuD}{\cop}\{\pro{\zuD([1],\mu)}{\zuD([0],\mu)}{j}\}\times H_{\mu}=(TUH)^j_1
\end{equation}
and similarly
\begin{equation}\label{segpseu-eq6}
\uset{\mu\in\zuD}{\cop}\{\pro{\zuD(v[1],\mu)}{\zuD([0],\mu)}{s}\}\times H_{\mu}=(TUH)^s_{v[1]}\;.
\end{equation}
Using \eqref{segpseu-eq5}, \eqref{segpseu-eq6} and the bijection \eqref{fatdels-eq2} in Lemma \ref{fatdels-lem1} we obtain
\begin{equation*}
\begin{split}
& (TUH)_{\eta}  = \uset{\mu\in\zuD}{\cop}\zuD(\eta,\mu)\times H_{\mu} \cong  \\
    & \cong (TUH)^{j_1}_{1}\tiund{(TUH)_0}(TUH)^{n_1}_{v[1]}\tiund{(TUH)_0} (TUH)^{j_2-j_1}\tiund{(TUH)_0}\cdots\\
   & \cdots\tiund{(TUH)_0} (TUH)^{n-j_t}\;.
\end{split}
\end{equation*}

  \item [c)] From above, for each $f\in\zuD(\eta,\mu)$, $h_{\eta}i_{\mu}j_f=H(f)$. Let $f$ correspond to
      \begin{equation*}
        \zd_1^{j_1}\tiund{\zd_0}\tilde{\zd}_1^{n_1}\tiund{\zd_0}\zd_{1}^{j_2-j_1}\tiund{\zd_0} \cdots \tiund{\zd_0}\zd_{1}^{n-j_t}
      \end{equation*}
      in the bijection \eqref{fatdels-eq1} of Lemma \ref{fatdels-lem1}, where $\zd_1\in\zuD([1],\mu)$, $\tilde{\zd}_1\in\zuD(v[1],\mu)$, $\zd_0\in\zuD([0],\mu)$, with $\zd_1^j=\pro{\zd_1}{\zd_0}{j}$, $\tilde{\zd}_1^j=\pro{\tilde{\zd}_1}{\tilde{\zd}_0}{j}$.

      Then $j_f=j_{\zd_1^{j_1}}\tiund{j_{\zd_0}} j_{\tilde{\zd}_1^{n_1}} \tiund{j_{\zd_0}}\cdots \tiund{j_{\zd_0}} j_{\zd_1^{n-j_t}}$. Since by hypothesis the Segal maps \eqref{fairtowg2Eq-3A} are isomorphisms, $H(f)$ corresponds to
       \begin{equation*}
       H(\zd_1)^{j_1}\tiund{H(\zd_0)}H(\tilde{\zd}_1)^{n_1}\tiund{H(\zd_0)} \cdots\tiund{H(\zd_0)}H(\zd_1)^{n-j_t}\;.
       \end{equation*}
Therefore
\begin{equation*}
\begin{split}
 &\ \;\;\;\;h_{\eta}i_{\mu}j_f=H(f)= \\
 & = (h_{1}i_{\mu}j_{\zd_1})^{j_1}\tiund{h_{0}i_{\mu}j_{\zd_0}} (h_{v[1]}i_{\mu}j_{\tilde{\zd}_1})^{n_1} \tiund{h_{0}i_{\mu}j_{\zd_0}} \cdots \tiund{h_{0}i_{\mu}j_{\zd_0}}(h_{1}i_{\mu}j_{\zd_1})^{n-j_t}=\\
& = (h_{1}^{j_1}\tiund{h_0}h_{v[1]}^{n_1}\tiund{h_0}h_{1}^{j_2-j_1} \tiund{h_0} \cdots \tiund{h_0}h_{1}^{n-j_t})i_{\mu}j_{f}\;.
\end{split}
\end{equation*}
It follows that
\begin{equation*}
h_\eta=h_1^{j_1}\tiund{h_0} h^{n_1}_{v[1]}\tiund{h_0}  h_1^{j_2-j_1}\tiund{h_0}\cdots\tiund{h_0}  h_1^{n-j_t}\;.
\end{equation*}
\end{itemize}

\end{proof}

\subsection{Weakly globular fair $2$-categories}\label{wgfair}
We now introduce the category $\fairwg{2}$ of weakly globular fair $2$-categories. This is a weakly globular version of $\fair{2}$. We replace the discrete object $X_0$ with a homotopically discrete object while retaining the strict Segal condition. This gives semi-categories internal to $\Cat$
\begin{equation*}
\tenx{X_1}{X_0}\rw X_1\rightrightarrows X_0\quad \text{and}\quad \tenx{X_{v[1]}}{X_0}\rw X_{v[1]}\rightrightarrows X_0\;.
\end{equation*}
The set underlying the discrete category $X_0^d$ plays the role of 'set of objects'. By analogy with the category $\catwg{2}$, we require induced Segal maps conditions.
In Lemma \ref{fairtowg2Lem-6} we show that there are semi-category structures
\begin{equation*}
\tenx{X_1}{X^d_0}\rw X_1\rightrightarrows X^d_0\quad \text{and}\quad \tenx{X_{v[1]}}{X^d_0}\rw X_{v[1]}\rightrightarrows X^d_0
\end{equation*}
and thus build a functor $D:\fairwg{2}\rw\fair{2}$ which discretizes $X_0$ to $X_0^d$.

The main property of the category $\fairwg{2}$ is that it arises as strictification of Segalic pseudo-functors from $\zuDop$ to $\Cat$ (see Theorem \ref{fairtowg2The-1}) in a way formally analogous to the way $\catwg{2}$ arises as strictification of Segalic pseudo-functors from $\Dop$ to $\Cat$ (see Theorem \ref{book-st-cat}).

Let $X\in[\zuDop,\Cat]$ be such that $X_0\in\cathd{}$, so that there is a map $\zg:X_0\rw X_0^d$. Let $\eta:[n']\rw [n]$ in $\zuDop$ and $j_i,n_i$ for $i=1,\ldots,t$ as in Definition \ref{segal-def1}. The map $\zg$ induces maps
\begin{equation*}
\begin{split}
    & \mu_j: \pro{X_1}{X_0}{j}\rw \pro{X_1}{X^d_0}{j} \\
    & \wt\mu_j: \pro{X_{v[1]}}{X_0}{j}\rw\pro{X_{v[1]}}{X^d_0}{j}
\end{split}
\end{equation*}
therefore composing with the Segal maps \eqref{segal-eq5.1}  we obtain induced Segal maps
\begin{equation}\label{wgfair-eq1}
\begin{gathered}
\begin{split}
    & X_{\eta}\rw (\pro{X_1}{X^d_0}{j_1})\tiund{X^d_0}(\pro{X_{v[1]}}{X^d_0}{n_1}) \tiund{X^d_0}(\pro{X_1}{X^d_0}{j_2-j_1}) \tiund{X^d_0}\cdot   \\
    & \cdots \tiund{X^d_0}(\pro{X_{v[1]}}{X^d_0}{n_t})\tiund{X^d_0}(\pro{X_1}{X^d_0}{n-j_t})\;.
\end{split}
\end{gathered}
\end{equation}
\begin{definition}\label{fairtowg2Def-4}
The category $\fairwg{2}$ of weakly globular fair $2$-categories is the full subcategory of $\fatcat{\zuDop}{\Cat}$ whose objects $X$ are such that
\begin{itemize}
  \item [a)] $X_0\in\cathd{}$.\mk

  \item [b)] The Segal maps \eqref{segal-eq5.1} are isomorphisms.\mk

  \item [c)] The induced Segal maps  \eqref{wgfair-eq1} are equivalences of categories.\mk

  \item [d)] $X$ preserves colours.
\end{itemize}
\end{definition}

It is clear that $\fair{2}\subset \fairwg{2}$. As in the case of $\fair{2}$ the Segal condition for $X\in\fairwg{2}$ means that the restriction to either copy of $\zDm \subset \zuD$ satisfies the Segal condition. Hence there are semi-categories internal to $\Cat$
\begin{equation*}
\begin{split}
    & X_1\tiund{X_0}X_1 \rw X_1\rightrightarrows X_0 \\
    & X_{v[1]}\tiund{X_0}X_{v[1]}\rw X_{v[1]}\rightrightarrows X_0
\end{split}
\end{equation*}
The Segal condition means that the rest of the diagram can be constructed from the serially commuting diagram
\begin{equation}\label{wgfair-eq01}
\begin{gathered}
\xymatrix@R=40pt@C=40pt{
X_0 & X_1 \ar@<-0.5ex>[l] \ar@<0.5ex>[l] \\
X_{v[1]} \ar[ur] \ar@<-0.5ex>[u] \ar@<0.5ex>[u] &
}
\end{gathered}
\end{equation}
provided that $X_{v[1]}\rw X_1$ is a semi-functor internal to $\Cat$, that is the following diagram commutes
\begin{equation}\label{wgfair-eq02}
\begin{gathered}
\begin{tikzcd}
	{X_{v[1]}\times_{X_0}X_{v[1]}} && {X_{v[1]}\rightrightarrows X_0} \\
	{X_{1}\times_{X_0}X_{1}} && {X_{1}\rightrightarrows X_0}
	\arrow[from=1-1, to=2-1]
	\arrow[shift right=5, from=1-3, to=2-3]
	\arrow[from=1-1, to=1-3]
	\arrow[from=2-1, to=2-3]
	\arrow["{\|}", shift left=5, draw=none, from=1-3, to=2-3]
\end{tikzcd}
\end{gathered}
\end{equation}

The preservation of colours is equivalent to requiring that the maps
\begin{equation}\label{wgfair-eq03}
X_{v[1]}\rightrightarrows X_0,\quad X_{v[1]}\tiund{X_0} X_1\rw X_1 \lw X_1\tiund{X_0}X_{v[1]},\quad X_{v[1]}\tiund{X_0}X_{v[1]}\rw X_{v[1]}\;.
\end{equation}
which are in the images of the map \eqref{fatdelta-eq1} are equivalences of categories.
\begin{remark}\label{wgfair-rem01}
As in the case of $\fair{2}$ (see Remark \ref{fair2catRem-1}) to give a weakly globular fair 2-category $X$ it is enough to give commutative diagram \eqref{wgfair-eq01} with $X_0\in\cathd{}$, a semi-functor of semi-categories internal to $\Cat$ \eqref{wgfair-eq02} such that \eqref{wgfair-eq03} are equivalences of categories.
\end{remark}

\begin{lemma}\label{fairtowg2Lem-6}
  There is a functor $D:\fairwg{2} \rw \fair{2}$ with

 \begin{equation*}
  (DX)_k=
\left\{
  \begin{array}{ll}
    X_0^d, & k=0 \\
    X_1, & k=1 \\
    \pro{X_1}{X_0^d}{k},& k>1\;.
  \end{array}
\right.
\end{equation*}

for each $k\in\zuDopm$. If $X\in\fair{2}$, $DX=X$.
\end{lemma}
\begin{proof}

Define the composite maps (for $k\geq 2$)
\begin{equation}\label{fairtowg2Eq-7}
\begin{split}
   & X_1\xrw{\pt_{i}}X_0\xrw{\zg}X_0^d \\
   & X_1\tiund{X_0^d}X_1 \xrw{\nu_2} X_1\tiund{X_0}X_1\xrw{\pt_i}X_1\\
   & \pro{X_1}{X_0^d}{k+1}\xrw{\nu_{k+1}}\pro{X_1}{X_0}{k+1}\xrw{\pt_i}\\
   & \rw\pro{X_1}{X_0}{k}\xrw{\hmu_k}\pro{X_1}{X_0^d}{k}
\end{split}
\end{equation}
where $\nu_{k+1}$ is pseudo-inverse to the induced Segal map $\hat\mu_{k+1}$.

Since $X\in\fairwg{2}$,
\begin{equation*}
 X_1\tiund{X_0}X_1 \rw X_1\rightrightarrows X_0
\end{equation*}
is a semi-category internal to $\Cat$. Also, $\hmu_k$ is an injective equivalence of categories. Reasoning as in the proof of Theorem \ref{StrongpseThe-1}, the maps \eqref{fairtowg2Eq-7} define a semi-category internal to $\Cat$
\begin{equation*}
 X_1\tiund{X^d_0}X_1 \rw X_1\rightrightarrows X^d_0
\end{equation*}
Let $(DX)_{v[1]}=X_{v[1]}$ and define the composite maps (for $k\geq 2$)
\begin{equation}\label{fairtowg2Eq-8}
\begin{split}
   & X_{v[1]}\xrw{\pt_{i}}X_0\xrw{\zg}X_0^d \\
   & X_{v[1]}\tiund{X_0^d}X_{v[1]} \xrw{\nu_2} X_{v[1]}\tiund{X_0}X_{v[1]}\xrw{\pt_i}X_{v[1]}\\
   & \pro{X_{v[1]}}{X_0^d}{k+1}\xrw{\nu_{k+1}}\pro{X_{v[1]}}{X_0}{k+1}\xrw{\pt_i}\\
   & \rw\pro{X_{v[1]}}{X_0}{k}\xrw{\wt\mu_k}\pro{X_{v[1]}}{X_0^d}{k}
\end{split}
\end{equation}
where $\nu_{k+1}$ is pseudo-inverse to the induced Segal map $\wt\mu_{k+1}$. Since $X\in\fairwg{2}$, there is a semi-category internal to $\Cat$
\begin{equation*}
X_{v[1]}\tiund{X_0}X_{v[1]}\rw X_{v[1]}\rightrightarrows X_0\;.
\end{equation*}
Also, $\hmu_k$ is an injective equivalence. Thus, reasoning as in the proof of Theorem \ref{StrongpseThe-1}, the maps \eqref{fairtowg2Eq-8} define a semi-category internal to $\Cat$.
\begin{equation*}
X_{v[1]}\tiund{X^d_0}X_{v[1]}\rw X_{v[1]}\rightrightarrows X^d_0\;.
\end{equation*}
 Since $X\in\fairwg{2}$, there is a semi-functor internal to $\Cat$
\begin{equation*}
\xymatrix@C=40pt@R=35pt{
 X_{v[1]}\tiund{X_0}X_{v[1]} \ar[r] \ar[d] & X_{v[1]}\ar[d] \ar@<0.5ex>[r] \ar@<-0.5ex>[r] & X_0 \ar@{=}[d] \\
X_1\tiund{X_0}X_1 \ar[r] & X_1 \ar@<0.5ex>[r] \ar@<-0.5ex>[r] & X_0
}
\end{equation*}
Hence the following diagram commutes:
\begin{equation*}
\xymatrix@C=40pt@R=25pt{
X_{v[1]}\tiund{X^d_0}X_{v[1]} \ar[r] \ar[d]_{\nu_2} & X_{v[1]}\ar@{=}[d] \ar@<0.5ex>[r] \ar@<-0.5ex>[r] & X_0 \ar@{=}[d] \ar[r] & X_0^d \ar@{=}[d]\\
X_{v[1]}\tiund{X_0}X_{v[1]} \ar[r] \ar[d]_{} & X_{v[1]}\ar[d] \ar@<0.5ex>[r] \ar@<-0.5ex>[r] & X_0 \ar@{=}[d] \ar[r] & X_0^d \ar@{=}[d]\\
X_1\tiund{X_0}X_1 \ar[r] \ar[d]_{\hmu_2} & X_1 \ar@<0.5ex>[r] \ar@<-0.5ex>[r]\ar@{=}[d] & X_0 \ar[r]\ar@{=}[d] & X_0^d \ar@{=}[d]\\
X_1\tiund{X^d_0}X_1 \ar[r] & X_1 \ar@<0.5ex>[r] \ar@<-0.5ex>[r] & X_0 \ar[r] & X_0^d
}
\end{equation*}
That is, there is a semi-functor internal to $\Cat$
\begin{equation*}
\xymatrix@C=40pt@R=35pt{
 X_{v[1]}\tiund{X^d_0}X_{v[1]} \ar[r] \ar[d] & X_{v[1]}\ar[d] \ar@<0.5ex>[r] \ar@<-0.5ex>[r] & X_0^d \ar@{=}[d] \\
X_1\tiund{X^d_0}X_1 \ar[r] & X_1 \ar@<0.5ex>[r] \ar@<-0.5ex>[r] & X_0^d
}
\end{equation*}
By Remark \ref{fair2catRem-1} to show that $DX\in\fair{2}$ it remains to show that the maps $\xymatrix{X_{v[1]} \ar@<0.5ex>[r] \ar@<-0.5ex>[r]  & X_0^d }$ as well as the composition maps
\begin{equation*}
\xymatrix{
X_{v[1]}\tiund{X^d_0}X_{v[1]} \ar[r] & X_{v[1]}
}
\end{equation*}
\begin{equation*}
\xymatrix{
X_{v[1]}\tiund{X^d_0}X_{1} \ar[r] & X_1 & X_1\tiund{X^d_0}X_{v[1]} \ar[l]
}
\end{equation*}
are equivalences of categories. This follows from the fact that $X\in\fairwg{2}$ and from the commutativity of the diagram
\begin{equation*}
\xymatrix@C=40pt@R=30pt{
X_{v[1]}\tiund{X^d_0}X_{v[1]} \ar[r] \ar[d] & X_{v[1]}\\
X_{v[1]}\tiund{X_0}X_{v[1]} \ar[ur]
}
\end{equation*}
\begin{equation*}
\xymatrix@C=40pt@R=30pt{
X_{v[1]}\tiund{X^d_0}X_{1} \ar[r] \ar[d] & X_{1} & X_1\tiund{X^d_0}X_{v[1]} \ar[d] \ar[l]\\
X_{v[1]}\tiund{X_0}X_{1} \ar[ur] && X_1\tiund{X_0}X_{v[1]} \ar[ul]
}
\end{equation*}
It is straightforward from the definition that if $X\in\fair{2}$, $DX=X$.

\end{proof}

\begin{definition}\label{fairtowg2-def2a1}
There is a truncation functor given as composite
\begin{equation*}
\fairwg{2}\xrw{D}\fair{2}\xrw{\p{1}}\Cat
\end{equation*}
where $\p{1}$ is as in Lemma \ref{level_fwg_0}. Thus if $X\in\fairwg{2}$,
\begin{equation*}
(D\p{1}X)_s=
\left\{
  \begin{array}{ll}
    X_0^d, & s=0 \\
    p X_1, & s=1 \\
    \pro{pX_1}{X_0^d}{s}, & s\geq 2\;.
  \end{array}
\right.
\end{equation*}
\end{definition}

\begin{definition}\label{fairtowg2-def2a}
Given $X\in\fairwg{2}$ and $a,b\in X_0^d$, we denote by $X(a,b)\subset X_1$ the fiber at $(a,b)$ of the map
\begin{equation*}
  X_1\xrw{(\pt_0,\pt_1)} X_0\times X_0 \xrw{\zg\times\zg} X_0^d\times X_0^d\;.
\end{equation*}
A morphism $F:X\rw Y$ in $\fairwg{2}$ is a $2$-equivalence if
\begin{itemize}
  \item [(i)] For all $a,b\in X_0^d$, $F(a,b):X(a,b)\rw Y(Fa,Fb)$ is an equivalence of categories.\mk

  \item [(ii)] $\p{1} D F$ is an equivalence of categories.
\end{itemize}
\end{definition}
\begin{remark}\label{fairtowg2Rem-4a}
Given $X\in\fairwg{2}$ and $a,b\in X_0^d$ we have $X(a,b)=DX(a,b)$. It follows that $D$ sends 2-equivalences to 2-equivalences.
\end{remark}

We have the following analogue for $\fairwg{2}$ of Lemma \ref{level_fwg} for $\fair{2}$.
\begin{lemma}\label{level_fwg_equi}
  Let $F:X\rw Y$ be a morphism in $\fairwg{2}$ which is a levelwise equivalence of categories. Then $F$ is a 2-equivalence.
\end{lemma}
\begin{proof}
Since $F_0$ is an equivalence of categories, $F_0^d$ is an isomorphism. Thus
\begin{equation*}
  Y_1=\uset{a',b'\in Y_0^d}{\cop} Y(a',b')\cong \uset{\parbox{12mm}{\centering $\Sc{Fa,Fb}$\\$\Sc{a,b\in X_0^d}$}}{\cop} Y(Fa,Fb)\;.
\end{equation*}
Since $\ds X_1=\uset{a,b\in X_0^d}{\cop}X(a,b)$ and $F_1$ is an equivalence of categories, it follows that $X(a,b)\rw Y(Fa,Fb)$ is an equivalence of categories for all $a,b\in X_0^d$. Also, since $F_n$ is an equivalence of categories for all $n\in\zDopm$, $pF_n=(D\p{1}F)_n$ is a bijection. Therefore $D\p{1}F$ is an isomorphism. By definition, it follows that $F$ is a $2$-equivalence.
\end{proof}

\begin{remark}\label{fairtowg2Rem-4}
Let ${\pi^*}:\funcat{}{\Cat}\rw[\zuDop,\Cat]$ be as in Definition \ref{fairtowg2Def-1}. We note that this functor restricts to a functor
 \begin{equation*}
   \pi^*: \catwg{2}\rw \fairwg{2}
 \end{equation*}
that is, if $X\in \catwg{2}$, then $\pi^* X\in\fairwg{2}$.
In fact, since $X\in\catwg{2}$ there are semi-categories internal to $\Cat$
\begin{equation*}
\begin{tikzcd}[row sep=8pt]
	{X_{1}\times_{X_0}X_{1}} & {(\pi^* X)_1=X_1} & {X_0=(\pi^* X)_0} \\
	{X_{0}\times_{X_0}X_{0}} & {(\pi^* X)_0=X_0} & {X_0=(\pi^* X)_0} \\
	\arrow["\Id", from=2-1, to=2-2]
	\arrow[from=1-1, to=1-2]
	\arrow[shift right, from=1-2, to=1-3]
	\arrow[shift left, from=1-2, to=1-3]
	\arrow["\Id", shift right=2, from=2-2, to=2-3]
	\arrow["\Id", shift left=2, from=2-2, to=2-3]
\end{tikzcd}
\end{equation*}
By the weak globularity condition, $X_0\in\cathd{}$. The remaining conditions in the definition of $\fairwg{2}$ follow from the induced Segal map condition for $X\in\catwg{2}$.

\end{remark}

\subsection{Weakly globular fair $2$-categories from Segalic pseudofunctors}
In this section we show that weakly globular fair 2-categories arise as strictification of Segalic pseudo-functors from $\zuDop$ to $\Cat$. The proof of this result is formally analogous to the one of Theorem \ref{book-st-cat}.
We will need the following technical observation.

\begin{remark}\label{fairtowg2Rem-5}
It is well known that equivalences of categories have the 2-out-of-3 property. That is, given the commutative diagram in $\Cat$
\begin{equation}\label{fairtowg2Eq-9a}
\begin{gathered}
\xymatrix@R=40pt@C=40pt{
& B \ar^{g}[dr] & \\
A\ar^{f}[ur] \ar_{h}[rr] && C
}
\end{gathered}
\end{equation}

\nid if two of the maps $f,g,h$ are equivalences of categories, so is the third. We note that this property holds even when \eqref{fairtowg2Eq-9a} only pseudo-commutes. In fact, suppose $f$ and $h$ are equivalences of categories then, for each $b,b'\in B$, there are isomorphisms $b\cong fa$ and $b'\cong fa'$; since $f$ and $h$ are fully faithful, we obtain
\begin{equation}\label{fairtowg2Eq-10a}
  B(b,b')=B(fa,fa')=A(a,a')=C(ha,ha')\;.
\end{equation}
On the other hand, since \eqref{fairtowg2Eq-9a} pseudo-commutes, there are isomorphisms $gfa\cong ha$ and $gfa'\cong ha'$. Since $b\cong fa$ and $b'\cong fa'$ we also have isomorphisms $gb\cong gfa$, $gb'\cong gfa'$; in conclusion there are isomorphisms $gb\cong ha$ and $gb'\cong ha'$. It follows that
\begin{equation}\label{fairtowg2Eq-11a}
  C(ha,ha')=C(gb,gb')\;.
\end{equation}
We conclude from \eqref{fairtowg2Eq-10a} and \eqref{fairtowg2Eq-11a} that $B(b,b')=C(gb,gb')$. That is, $g$ is fully faithful.

Let $c\in C$. Since $h$ is essentially subjective on objects, there is an isomorphism $c\cong ha$ for some $a\in A$. Since, from the above, $gb\cong ha$, it follows that $c\cong gb$. Thus $g$ is also essentially surjective on objects. In conclusion, $g$ is an equivalence of categories. The proof in the other cases is similar.
\end{remark}

\begin{theorem}\label{fairtowg2The-1}
The strictification functor
\begin{equation*}
  \St:\Ps\fatcat{\zuDop}{\Cat}\rw \fatcat{\zuDop}{\Cat}
\end{equation*}
restricts to a functor
\begin{equation*}
  \St:\segps\fatcat{\zuDop}{\Cat}\rw \fairwg{2}\;.
\end{equation*}
Further, for each $H\in\segps\fatcat{\zuDop}{\Cat}$ there is a pseudo-natural transformation $\St H\rw H$ whose components are equivalences of categories.
\end{theorem}
\begin{proof}
From \cite{PW}, to construct the strictification $L=\St H$ of the pseudo-functor $H$ we need to factorize $h:TUH\rw UH$ as $h=gv$ in such a way that, for each $\eta\in\zuDop$, $h_{\eta}$ factorizes as
\begin{equation*}
  (TUH)_{\eta}\xrw{v_{\eta}} L_{\eta}\xrw{g_{\eta}}(UH)_{\eta}=H_{\eta}
\end{equation*}
with $v_{\eta}$ bijective on objects and $g_{\eta}$ fully and faithful. As explained in \cite{PW}, $g_{\eta}$ is in fact an equivalence of categories.

Since the bijective on objects and the fully faithful functors form a factorization system in $\Cat$, the commutativity of \eqref{segpseu-eq1a} in Lemma \ref{segpseu-lem1} implies that there are functors
\begin{equation*}
\ovl{d_i}:L_1\rightrightarrows L_0, \qquad \ovl{d'_i}:L_{v[1]}\rightrightarrows L_0\quad i=0,1
\end{equation*}
such that the following diagrams commute:
\begin{equation*}
\xymatrix{
(TUH)_1 \ar[r] \ar@<-1ex>_{\pt'_0}[d]\ar@<1ex>^{\pt'_1}[d] & L_1 \ar[r]\ar@<-1ex>_{\ovl{d}_0}[d]\ar@<1ex>^{\ovl{d}_1}[d] & H_1 \ar@<-1ex>_{\pt_0}[d]\ar@<1ex>^{\pt_1}[d]\\
(TUH)_0\ar[r] & L_0 \ar[r] & H_0
}
\qquad
\xymatrix{
(TUH)_{v[1]} \ar[r] \ar@<-1ex>_{\ptt_0}[d]\ar@<1ex>^{\ptt_1}[d] & L_{v[1]} \ar[r]\ar@<-1ex>_{\ovl{d}'_0}[d]\ar@<1ex>^{\ovl{d}'
_1}[d] & H_{v[1]} \ar@<-1ex>_{\ptt_0}[d]\ar@<1ex>^{\ptt_1}[d]\\
(TUH)_0\ar[r] & L_0 \ar[r] & H_0
}
\end{equation*}
By Lemma \ref{segpseu-lem1}, $h_{\eta}$ factorizes as
\begin{equation*}
\begin{split}
    & (TUH)_{\eta}\cong (TUH)^{j_1}_{1}\tiund{(TUH)_0} (TUH)^{n_1}_{1}\tiund{(TUH)_0} (TUH)^{j_2-j_1}\tiund{(TUH)_0}\cdots\\
    &\cdots \tiund{(TUH)_0}(TUH)^{n-j_t}\xrw{\pro{v_1}{v_0}{}} L_1^{j_1}\tiund{L_0} L_{v[1]}^{n_1}\tiund{L_0} L_1^{j_2-j_1}\tiund{L_0}\cdots \tiund{L_0}L_1^{n-j_t}\\
    & \xrw{\pro{g_1}{g_0}{}} H_1^{j_1}\tiund{H_0} H_{v[1]}^{n_1}\tiund{H_0} H_1^{j_2-j_1}\tiund{H_0}\cdots \tiund{H_0}H_1^{n-j_t}
\end{split}
\end{equation*}
where we denoted%
\begin{equation*}
\begin{split}
    &  L_1^k=\pro{L_1}{L_0}{k},\qquad L_{v[1]}^k=\pro{L_{v[1]}}{L_0}{k}\;,\\
    & H_1^k=\pro{H_1}{H_0}{k},\qquad H_{v[1]}^k=\pro{H_{v[1]}}{H_0}{k}\:.
\end{split}
\end{equation*}
Since $v_1$ and $v_0$ are bijective on objects, such is $\pro{v_1}{v_0}{}$. Since $g_1$ and $g_0$ are fully faithful such is $\pro{g_1}{g_0}{}$. So the above is the required factorization of $g_{\eta}$ and we conclude that
\begin{equation*}
L_{\eta}\cong L_1^{j_1}\tiund{L_0}L_{v[1]}^{n_1}\tiund{L_0}L_1^{j_2-j_1}\tiund{L_0}\cdots \tiund{L_0} L_1^{n-j_t}\;.
\end{equation*}
That is all the Segal maps of $L$ are isomorphisms.

By \cite{PW}, $g:L\rw H$ is a pseudo-natural transformation with $g_{\eta}$ an equivalence of categories for all $\eta\in\zuDop$. In particular there are equivalences of categories
\begin{equation*}
  H_1\simeq L_1, \qquad H_{v[1]}\simeq L_{v[1]}\qquad H_0\simeq L_0\;.
\end{equation*}
Since $H_0$ is discrete, this implies
\begin{equation*}
\begin{split}
    & H_{\eta}\cong H_1^{j_1}\tiund{H_0} H_{v[1]}^{n_1}\tiund{H_0} H_1^{j_2-j_1}\tiund{H_0}\cdots \tiund{H_0}H_1^{n-j_t}\simeq \\
    & \simeq L_1^{j_1}(H_0)\tiund{H_0} L_{v[1]}^{n_1}(H_0)\tiund{H_0} L_1^{j_2-j_1}(H_0)\tiund{H_0}\cdots \tiund{H_0}L_1^{n-j_t}(H_0)\;,
\end{split}
\end{equation*}
where we denoted
\begin{equation*}
L^k_1(H_0)=\pro{L_1}{H_0}{k}\qquad L^k_{v[1]}(H_0)=\pro{L_{v[1]}}{H_0}{k}\;.
\end{equation*}
On the other hand
\begin{equation*}
H_{\eta}\simeq L_{\eta}\cong L_1^{j_1}\tiund{L_0} L_{v[1]}^{n_1}\tiund{L_0} L_1^{j_2-j_1}\tiund{L_0}\cdots \tiund{L_0}L_1^{n-j_t}\;.
\end{equation*}
In conclusion, since $H_0\cong L_0^d$, we obtain an equivalence of categories
\begin{equation*}
\begin{split}
    & L_1^{j_1}\tiund{L_0} L_{v[1]}^{n_1}\tiund{L_0} L_1^{j_2-j_1}\tiund{L_0}\cdots \tiund{L_0}L_1^{n-j_t}\cong \\
  \cong  & L_{\eta}\simeq H_{\eta} \simeq L_1^{j_1}(H_0)\tiund{H_0} L_{v[1]}^{n_1}(H_0)\tiund{H_0} L_1^{j_2-j_1}(H_0)\tiund{H_0}\cdots \tiund{H_0}L_1^{n-j_t}(H_0)\cong \\
  \cong  & L_1^{j_1}(L_0^d)\tiund{L^d_0} L_{v[1]}^{n_1}(L_0^d)\tiund{L^d_0} \cdots \tiund{L^d_0}L_1^{n-j_t}(L_0^d)\;.
\end{split}
\end{equation*}
This means that all the induced Segal maps of $L$ are equivalences of categories.

Since $L\rw H$ is a pseudo-natural transformation, from the above the following diagrams pseudo-commute:
\begin{equation*}
\xymatrix@C=40pt@R=35pt{
L_{v[1]}\ar@<0.5ex>[r] \ar@<-0.5ex>[r] \ar[d] & L_0 \ar[d]\\
H_{v[1]}\ar@<0.5ex>[r] \ar@<-0.5ex>[r] & H_0
}
\qquad \quad
\xymatrix@C=40pt@R=35pt{
L_{v[1]}\tiund{L_0}L_{v[1]} \ar[r] \ar[d] & L_{v[1]} \ar[d] \\
H_{v[1]}\tiund{H_0}H_{v[1]} \ar[r] & H_{v[1]}
}
\end{equation*}
\begin{equation*}
\xymatrix@C=40pt@R=35pt{
L_{v[1]}\tiund{L_0}L_{1} \ar[r] \ar[d] & L_1 \ar[d] &L_1\tiund{L_0}L_{v[1]} \ar[l] \ar[d]\\
H_{v[1]}\tiund{H_0}H_{1} \ar[r]  & H_1  &H_1\tiund{H_0}L_{v[1]} \ar[l]
}
\end{equation*}
Since $H\in\segps\fatcat{\zuDop}{\Cat}$, the bottom maps are equivalences of categories. The vertical maps are also equivalences of categories since $L\rw H$ is a levelwise equivalence. By Remark \ref{fairtowg2Rem-5} it follows that the top maps are also equivalences of categories. By Remark \ref{wgfair-rem01} this completes the proof that $L\in\fairwg{2}$.
\end{proof}
%

\section{From fair $2$-categories to weakly globular double categories}\label{fairtowg2}
In this Section we construct the functor $R_2:\fair{2}\rw\catwg{2}.$ This category factors through the category of Segalic pseudo-functors as follows
\begin{equation*}
\fair{2}\xrw{T_2}\segpsc{}{\Cat}\xrw{St}\catwg{2}\;,
\end{equation*}
where $St$ is as in Theorem \ref{book-st-cat}. The main goal of this section is the construction of the functor $T_2$. We first discuss a general set up and method to construct pseudo-functors which we will then apply to our specific case.
\begin{proposition}\label{fairtowg2-pro1}
Let $\clC,\clD$ be categories such that $\ob\clD\subset \ob\clC$ and suppose there are functors
\begin{equation*}
\xymatrix{
\clC \ar^{F}[r]\ar_{\pi}[d] & \Cat\\
\clD
}
\end{equation*}
Suppose that, for each $C\in\clC$, there is a map in $\clC$
\begin{equation*}
  \nu_C:C\rw\pi(C)
\end{equation*}
and suppose we are given specified adjoint equivalences of categories
\begin{equation*}
  \zb_C:F(\pi(C))\leftrightarrows F(C):\za_C=F(\nu_C)
\end{equation*}
with $\za_C\zb_C=\Id$. Then
\begin{itemize}
  \item [a)] There is a pseudo-functor $G:\clC\rw\Cat$ given on objects by $G(C)=F(\pi(C))$ and there are pseudo-natural transformations $\za:F\rw G$ and $\zb:G\rw F$ with components $\za_C,\zb_C$ respectively.

\nid Further, given another functor $F':\clC\rw\Cat$ and a natural transformation $\mu:F\rw F'$, there is a pseudo-natural transformation $\xi:G\rw G'$ between the corresponding pseudo-functors. If $\mu$ is componentwise an equivalence of categories, such is $\xi$.\mk

  \item [b)] Suppose, further, that the following two conditions are satisfied:\mk
  \begin{itemize}
    \item [i)] If $N:\Cat\rw\funcat{}{\Set}$ denotes the nerve functor, the map in $\funcat{}{\Set}$ \;$N\pi:N\clC\rw N\clD$ is levelwise surjective.\mk

    \item [ii)] Given maps  $f_1:C_1\rw C'_1$ and $f_2:C_2\rw C'_2$ in $\clC$ such that $\pi f_1=\pi f_2$, then
   \begin{equation*}
    \za_{C'_1}F(f_1)\zb_{C_{1}}=\za_{C'_2}F(f_2)\zb_{C_2}\;.
   \end{equation*}
  \end{itemize}

\nid Then there is a pseudo-functor $\wt F\in\Ps[\clD,\Cat]$ given on objects by $F(\pi(C))$ for each $\pi(C)\in\ob\clD$.

\nid Further, given another functor $F':\clC\rw\Cat$ and a natural transformation $\mu:F\rw F'$, there is a pseudo.natural transformation $\xi:\wt F\rw \wt F'$ between the corresponding pseudo-functors. If $\mu$ is componentwise an equivalence of categories, such is $\xi$.\mk

\item[c)] Under the hypotheses of a) and b), if $\pi^*:\Ps[\clD,\Cat]\rw\Ps[\clC,\Cat]$ is induced by $\pi$, it is $\pi^*\wt F=G$ and there are levelwise equivalence pseudo-natural transformations in $\Ps[\clC,\Cat]$, $\pi^*\wt F\rw F$ and $F\rw \pi^*\wt F$.

\end{itemize}
\end{proposition}
\begin{proof}\
\begin{itemize}
  \item [a)] We apply Lemma \ref{lem-PP} with $G(C)=F(\pi(C))$ and equivalences of categories $\zb_C,\,\za_C$. Given $f:C\rw D$ in $\clC$, $G(f)$ is given by the composite
      \begin{equation*}
        F(\pi(C))\xrw{\zb_C}F(C)\xrw{F(f)}F(D)\xrw{\za_D}F(\pi(D))\;.
      \end{equation*}
      The 2-dimensional structure is as in Lemma \ref{lem-PP}. Note that the 2-cell
    \begin{equation*}
    \xymatrix{
    G(C)  \ar@/_2.7pc/[rr]_{\Id_{G(C)}} \ar^{G(\Id_C)}[rr]  & \ar@{}|{\big\Downarrow}[d] & G(C)\\
    & \ &
    }
    \end{equation*}
    is the identity since by hypothesis $\za_C\zb_C=\Id$.

The existence of pseudo-natural transformations $\za:F\rw G$ and $\zb:G\rw F$ follows from Lemma \ref{lem-PP}. To build the pseudo-natural transformation $\xi:G\rw G'$ we apply Lemma \ref{lem-PP} b). We let $\xi_C:G(C)\rw G'(C)$ be
\begin{equation*}
\xi_C=\za_{C'}\mu_C\zb_C\;.
\end{equation*}
Then $\xi_C\za_C=(\za_{C'}\mu_C\zb_C)\za_C\cong \za_{C'}\mu_C$. Therefore hypothesis (3) of Lemma \ref{lem-PP} is satisfied and there is a pseudo-natural transformation $\xi:G\rw G'$ with $\xi(C)=\xi_{C}$.

Since $\za_{C'}$ and $\zb_C$ are equivalences of categories, if $\mu_C$ is an equivalence of categories, such is $\xi_C$.\mk

  \item [b)] Since $\pi$ is surjective on objects, every object of $\clD$ has the form $\pi(C)$ for some $C\in\clC$ and we define $\wt F(\pi(C))=F(\pi(C))=G(C)$ with $G$ as in a).

      Given a map $f:\pi(C)\rw\pi(D)$ in $\clD$, by hypothesis b) i) the map of sets $(N\pi)_1:(N\clC)_1\rw (N\clD)_1$ is a surjection, thus there is a map $f':C'\rw D'$ in $\clC$ with $\pi f'=f$. We define $\wt F(f)$ to be the composite
      \begin{equation*}
        F(\pi(C))=F(\pi(C'))\xrw{\zb_{C'}}F(C')\xrw{F(f')}F(D')\xrw{\za_{D'}}F(\pi(D'))=F(\pi(D))\;.
      \end{equation*}
      Thus $\wt F(f)=G(f')$ with $G$ as in a).
      By hypothesis b) ii) this is well-defined, that is it is independent on the map $f'$ in $\clC$ with $\pi(f')=f$.

      Given composable morphisms in $\clD$
      \begin{equation*}
        \pi(C)\xrw{f}\pi(D)\xrw{g}\pi(E)
      \end{equation*}
      by hypothesis b) i) the map of sets $(N\pi)_2:(N\clC)_2\rw (N\clD)_2$ is a surjection, thus there exist composable morphisms in $\clC$
      \begin{equation*}
        C''\xrw{f''}D''\xrw{g''}E''
      \end{equation*}
      such that $\pi(f'')=f$ and $\pi(g'')=g$. Then $\wt F(gf)$ is given by the composite
      \begin{equation*}
        F(\pi(C))=F(\pi(C''))\xrw{\zb_{C''}}F(C'')\xrw{g''f''}F(E'')\xrw{\za_{E''}}F(\pi(E''))=F(\pi(E))\;.
      \end{equation*}
The 2-cell
\begin{equation*}
\xymatrix{
\wt F(\pi(C)) \ar@/_2.7pc/[rrrr]_{\wt F(gf)} \ar^{\wt F(f)}[rr] && \wt F(\pi(D)) \ar@{}[d]|{\big\Downarrow} \ar^{\wt F(g)}[rr] && \wt F(\pi(E))\\
&& \
}
\end{equation*}
is the same as the 2-cell for the pseudo-functor $G$
\begin{equation*}
\xymatrix@C=18pt{
G(C)\!=\!\wt F(\pi(C)) \ar@/_3pc/[rrrr]_{G(g'' f'')} \ar^{G(f'')}[rr] && G(D)\!=\! \wt F(\pi(D)) \ar@{}[d]|{\big\Downarrow}  \ar^{G(g'')}[rr] && G(E)\!=\!\wt F(\pi(E))\\
&& \
}
\end{equation*}
(note that, from above, $\wt F(f)=G(f')=G(f'')$ since $\pi(f')=\pi(f'')=f$).

The 2-cell
\begin{equation}\label{fairtowg2-eq1}
\begin{gathered}
\xymatrix{
 \wt F(\pi(C))  \ar@/_2.7pc/[rr]_{\Id_{\wt F(\pi(C))}} \ar^{\wt F(\Id_{\pi(C)})}[rr]  & \ar@{}|{\big\Downarrow}[d] & \wt F(\pi(C))\\
 & \ &
 }
\end{gathered}
\end{equation}
is the identity.

Given maps in $\clD$, $\pi(C)\xrw{f}\pi(D)\xrw{g}\pi(E)\xrw{h}\pi(F)$, by hypothesis b) i) the map of sets $(N\pi)_3:(N\clC)_3\rw (N\clD)_3$ is a surjection, thus there are maps in $\clC$,
\begin{equation*}
C'''\xrw{f'''}D'''\xrw{g'''}E'''\xrw{h'''}F'''
\end{equation*}
with $\pi(f''')=f$, $\pi(g''')=g$, $\pi(h''')=h$. By construction,
\begin{equation*}
\begin{split}
    & \wt F(hgf)=G(h'''g'''f''') \\
    & \wt F(hg)\wt F(f) = G(h'''g''')G(f''') \\
    & \wt F(h)\wt F(gf)=G(h''')G(g'''f''')\;.
\end{split}
\end{equation*}
Also by construction, the 2-cells
\begin{equation*}
\begin{split}
    & \wt F(h)\wt F(gf)\Rw \wt F(hgf) \Lw \wt F(hg)\wt F(f)\\
    & \wt F(g)\wt F(f)\Rw \wt F(gf) \\
    & \wt F(h)\wt F(g)\Rw \wt F(hg)
\end{split}
\end{equation*}
are the same as the ones for the pseudo-functor $G$. Therefore, since $G$ is a pseudo-functor, we have the coherence diagram:
\begin{equation*}
\begin{tikzcd}[column sep=large, row sep=large]
\wt F(A) \arrow[rr,"\wt F(hgf)", ""{name=UP}] \arrow[ddrr,"\wt F(gf)", ""{name=DR}] \arrow[phantom, "{\Uparrow}", from=DR, to=UP, shift right=1cm, yshift=-2mm] \arrow[dd,"\wt F(f)"'] && \wt F(D)\\
\\
\wt F(B) \arrow[rr,"\wt F(g)"',""name=BT] \arrow[phantom, "{\Uparrow}", from=BT, to=DR, shift left=1cm, yshift=2mm] && \wt F(C) \arrow[uu,"\wt F(h)"']
\end{tikzcd}
\begin{tikzcd}
  \arrow[d, phantom,"\equiv"] \\
 \
\end{tikzcd}
\begin{tikzcd}[column sep=large, row sep=large]
\wt F(A) \arrow[rr,"\wt F(hgf)", ""{name=UP}]  \arrow[phantom, "{\Uparrow}", from=DR, to=UP, shift right=1cm, yshift=-12mm] \arrow[dd,"\wt F(f)"'] &&  \wt F(D)\\
\\
\wt F(B) \arrow[uurr,"", ""{name=UR}] \arrow[rr,"\wt F(g)"',""name=BT] \arrow[phantom, "{\Uparrow}", from=UR, to=UP, shift left=1.cm, yshift=-4mm] && \wt F(C) \arrow[uu,"\wt F(h)"']
\end{tikzcd}
\end{equation*}
Since the 2-cell \eqref{fairtowg2-eq1} is the identity, there are no further coherence axioms to check. In conclusion,
\begin{equation*}
  \wt F\in\Ps[\clD,\Cat]\;.
\end{equation*}
Since by construction $\wt F(\pi(C))=G(C)$, $\wt F(f)=G(f')$, $\wt F'(\pi(C))=G'(C)$, $\wt F'(f)=G'(f')$ and the 2-dimensional structure of the pseudo-functors $\wt F$ and $\wt F'$ is as the one of $G$ and $G'$ respectively, there is a pseudo-natural transformation $\xi:\wt F\rw\wt F'$ given by $\xi:G\rw G'$. As in part a), if $\mu$ is a componentwise equivalence of categories, such is $\xi$.\mk

\item[c)] By construction, for each $C\in\clC$, $f:C\rw C'$ in $\clC$, $(\pi^*\wt F)(C)=\wt F(\pi(C))= F(\pi(C))=G(C)$ and $(\pi^*\wt F)(f)=\wt F(\pi(f))=\za_{C'}F(f)\zb_C=G(f)$.

\nid    Similarly, the 2-dimensional structures of $\pi^*\wt F$ and $G$ coincide. In conclusion, $\pi^*\wt F=G$.

\nid    The natural transformations $\pi^*\wt F=G\rw F$ and $F\rw \pi^*\wt F=G$ are as in part a) and are levelwise equivalences of categories.

\end{itemize}

\end{proof}
\begin{remark}\label{fairtowg2-rem1}
By Remark \ref{sbs-trans-rem1}, a different choice of adjoint equivalences of categories $\za_{C},\zb_{C}$ in Proposition \ref{fairtowg2-pro1} would yield an equivalent pseudo-functor $G$ in the 2-category $[\clC,\Cat]$ and an equivalent pseudo-functor $\wt F$ in the 2-category $[\clD,\Cat]$.
\end{remark}
To build the functor $T_2$ in the next theorem, we are going to apply the previous proposition to the case where $F=X\in\fair{2}$ and $\pi:\zuD\rw\zD$ is as in Section \ref{fatd}. We treat some preliminaries in the following remarks and lemma.
\begin{remark}\label{fairtowg2-rem2}
Let $X\in\fair{2}$ and denote, as in Section \ref{fair2cats}, $X_0=\clO$, $X_1=\clA$, $X_{v[1]}=\clU$. Recall that the two maps $\clU\rightrightarrows\clO$ coincide and are equivalence of categories. Let denote these maps by $\zg:\clU\rw\clO$. Throughout this section we fix a choice of pseudo-inverse $\zg':\clO\rw\clU$ so that $\zg\zg'=\Id$ (since $\clO$ is a discrete category).
\end{remark}
\begin{lemma}\label{fairtowg2-lem01}
Let $f:\eta_1\rw\eta_2$ be a coloured arrow in $\zuD$ as follows
\begin{equation}\label{fairtowg2-eq1-1}
\begin{gathered}
\xymatrix{
[n'] \ar@{^{(}->}^{\zve}[r]\ar_{\eta_1}[d] & [n'']\ar^{\eta_2}[d]\\
[n]\ar_{\Id_{[n]}}[r] & [n]
}
\end{gathered}
\end{equation}
Let $j_i,t_i$ ($i=1,\ldots,t$) be as in Definition \ref{segal-def1} for $\eta_1$ so that
\begin{equation*}
\eta_1=\Id_{[j_1]}\dotp v[n_1]\dotp\Id_{[j_2-j_1]}\dotp v[n_2]\dotp\cdots\dotp\Id_{[n-j_t]}\;.
\end{equation*}
Since $\pi(f)=\Id_n$ it is
\begin{equation*}
\eta_2=\Id_{[j_1]}\dotp v[m_1]\dotp\Id_{[j_2-j_1]}\dotp v[m_2]\dotp\cdots\dotp\Id_{[n-j_t]}\;.
\end{equation*}
where $n_i\leq m_i$ and $m_i=|\eta_2^{-1}(j_i)-1|$ for $i=1,\ldots,t$.

Denote $\zg\up{j}=\pro{\zg}{\Id_{\clO}}{j}$ and ${\zg'}\up{j}=\pro{\zg'}{\Id_{\clO}}{j}$. Let $X\in\fair{2}$ so that $X(f^{op}):X_{\eta_2}\rw X_{\eta_1}$ is an equivalence of categories. Define
\begin{equation}\label{fairtowg2-eq5}
\zb_{f^{op}}=\Id_{\clA\up{j_1}}\tiund{\clO}{\zg'}\up{m_1}\zg\up{n_1}\tiund{\clO} \Id_{\clA\up{j_2-j_1}}\tiund{\clO}{\zg'}\up{m_2}\zg\up{n_2}\tiund{\clO} \cdots\tiund{\clO}\Id_{\clA\up{n-j_t}}\;,
\end{equation}
where we denoted $\clA^{(j)}=\pro{\clA}{\clO}{j}$.
Then $\zb_{f^{op}}$ is a pseudo-inverse for $X(f^{op})$.
\end{lemma}
\begin{proof}
From the expressions of $\eta_1$ and $\eta_2$, it is
\begin{equation*}
\begin{split}
    & X_{\eta_1}=\clA\up{j_1}\tiund{\clO}\clU\up{n_1}\tiund{\clO} \clA\up{j_2-j_1}\tiund{\clO}\clU\up{n_2}\tiund{\clO}\cdots \tiund{\clO}\clA\up{n-j_t}   \\
    & X_{\eta_2}=\clA\up{j_1}\tiund{\clO}\clU\up{m_1}\tiund{\clO} \clA\up{j_2-j_1}\tiund{\clO}\clU\up{m_2}\tiund{\clO}\cdots \tiund{\clO}\clA\up{n-j_t}
\end{split}
\end{equation*}
where we denoted
\begin{equation*}
\clA\up{j}=\pro{\clA}{\clO}{j},\qquad \clU\up{j}=\pro{\clU}{\clO}{j}\;.
\end{equation*}
For each $i=1,\ldots,t$ the map $X(f^{op})$ restricts to a map
\begin{equation*}
X(f^{op})_{|m_i}:\clU\up{m_i}\rw\clU\up{n_i}
\end{equation*}
making the following diagram commute
\begin{equation*}
\xymatrix@C=40pt@R=30pt{
\clU\up{m_i}\ar^{X(f^{op})_{|m_i}}[r]\ar_{\zg\up{m_i}}[d] & \clU\up{n_i}\ar^{\zg\up{n_i}}[d] \\
\clO\up{m_i}=\clO \ar_{\Id}[r] & \clO=\clO\up{n_i}
}
\end{equation*}
where we denoted $\zg\up{j}=\pro{\zg}{\Id_{\clO}}{j}$. It follows that
\begin{equation}\label{fairtowg2-eq2}
{\zg'}\up{m_i}\zg\up{n_i}X(f^{op})_{|m_i}\cong {\zg'}\up{m_i}\zg\up{m_i}\cong \Id\;.
\end{equation}
Using the fact that $X(f^{op})_{|m_i}$, when restricted to $\pro{\zg'(\clO)}{\clO}{m_i}\cong \zg'(\clO)$ becomes $\Id_{\zg'(\clO)}$, we have that
\begin{equation}\label{fairtowg2-eq3}
X(f^{op})_{|m_i}{\zg'}\up{m_i}={\zg'}\up{n_i}
\end{equation}
and therefore
\begin{equation}\label{fairtowg2-eq4}
X(f^{op})_{|m_i}{\zg'}\up{m_i}\zg\up{n_i}={\zg'}\up{n_i}\zg\up{n_i}\cong\Id\;.
\end{equation}
Hence \eqref{fairtowg2-eq2} and \eqref{fairtowg2-eq4} show that ${\zg'}\up{m_i}\zg\up{n_i}$ is pseudo-inverse for $X(f^{op})_{|m_i}$.
Since
\begin{equation*}
X(f^{op})=\Id_{\clA^{(j_1)}}\tiund{\clO}X(f^{op})_{|_{m_1}}\tiund{\clO}\Id_{\clA^{(j_2-j_1)}} \tiund{\clO}\cdots \tiund{\clO}\Id_{\clA^{(n-j_t)}}
\end{equation*}
from the definition of $\zb_{f^{op}}$ in \eqref{fairtowg2-eq5}, this implies that $\zb_{f^{op}}$ is pseudo-inverse for $X(f^{op})$.
\end{proof}
\begin{remark}\label{fairtowg2-rem3}
Let $f:\eta_1\rw\eta_2$ be a coloured arrow in $\zuD$ as in Lemma \ref{fairtowg2-lem01} and suppose that $\eta_1=\Id_{[n]}$. Let $X\in\fair{2}$ then by \eqref{fairtowg2-eq5}
\begin{equation*}
\zb_{f^{op}}=\Id_{\clA\up{j_1}}\tiund{\clO}{\zg'}\up{m_1}\tiund{\clO}\Id_{\clA\up{j_2-j_1}} \tiund{\clO}\cdots\tiund{\clO}\Id_{\clA\up{n-j_t}}\;.
\end{equation*}
Since $X(f^{op})_{|m_i}{\zg'}\up{m_i}=\Id_{\clO}$ for $i=1,\ldots,t$ it follows that
\begin{equation}\label{fairtowg2-eq6}
X(f^{op})\zb_{f^{op}}=\Id\;.
\end{equation}
\end{remark}
\begin{lemma}\label{fairtowg2-lem02}
Let $f:\eta_1\rw\eta_2$ be the coloured arrow in $\zuD$ as in \eqref{fairtowg2-eq1-1} and let $X\in\fair{2}$. Let $\zb_{f^{op}}:X_{\eta_1}\rw X_{\eta_2}$ be as in \eqref{fairtowg2-eq5}. Let $\nu_{\eta_1}:\Id_{[n]}\rw\eta_1$ be the coloured arrow in $\zuD$ as in Definition \ref{fatdels-def2} and let $\zb_{\nu_{\eta_1}^{op}}:X_n\rw X_{\eta_1}$ be the pseudo-inverse to $X(\nu_{\eta_1}^{op})$ constructed as in Remark \ref{fairtowg2-rem2}. Then the following diagram commutes
\begin{equation*}
\xymatrix@C=40pt{
X_n \ar@/_2.7pc/[rrr]_{\zb_{\nu_{\eta_1}^{op}}} \ar^{\zb_{\nu_{\eta_1}^{op}}}[r] & X_{\eta_1}\ar^{\zb_{f^{op}}}[r] & X_{\eta_2}\ar^{X({f^{op}})}[r] & X_{\eta_1}
}
\end{equation*}
\end{lemma}
\begin{proof}
By Remark \ref{fairtowg2-rem2}
\begin{equation*}
\zb_{f^{op}}=\Id_{\clA\up{j_1}}\tiund{\clO}{\zg'}\up{m_1}{\zg}\up{n_1}\tiund{\clO} \Id_{\clA\up{j_2-j_1}} \tiund{\clO}\cdots\tiund{\clO}\Id_{\clA\up{n-j_t}}\;.
\end{equation*}
By Remark \ref{fairtowg2-rem3}
\begin{equation*}
\zb_{\nu_{\eta_1}^{op}}=\Id_{\clA\up{j_1}}\tiund{\clO}{\zg'}\up{n_1}\tiund{\clO}\Id_{\clA\up{j_2-j_1}} \tiund{\clO}\cdots\tiund{\clO}\Id_{\clA\up{n-j_t}}\;.
\end{equation*}
Since $\zg\up{n_i}{\zg'}\up{n_i}=\Id$ for all $i=1,\ldots,t$ it follows that
\begin{equation}\label{fairtowg2-eq7}
\zb_{f^{op}}\zb_{\nu_{\eta_1}^{op}}=\Id_{\clA\up{j_1}}\tiund{\clO}{\zg'}\up{m_1} \tiund{\clO}\Id_{\clA\up{j_2-j_1}}\tiund{\clO}\cdots\tiund{\clO}\Id_{\clA\up{n-j_t}}\;.
\end{equation}
Then \eqref{fairtowg2-eq3} and \eqref{fairtowg2-eq7} imply
\begin{equation*}
X(f^{op})\zb_{f^{op}}\zb_{\nu_{\eta_1}^{op}}=\Id_{\clA\up{j_1}}\tiund{\clO}{\zg'}\up{n_1} \tiund{\clO}\Id_{\clA\up{j_2-j_1}}\tiund{\clO}\cdots\tiund{\clO}\Id_{\clA\up{n-j_t}} = \zb_{\nu_{\eta_1}^{op}}\;.
\end{equation*}
as required.
\end{proof}
\begin{remark}\label{fairtowg2-rem4}
Note that \eqref{fairtowg2-eq7} also means that $\zb_{f^{op}}\zb_{\nu_{\eta_1}^{op}}=\zb_{\nu_{\eta_2}^{op}}$
\end{remark}
\begin{theorem}\label{fairtowg2-the1}
There is a functor
\begin{equation*}
T_2:\fair{2}\rw \segpsc{}{\Cat}
\end{equation*}
given on objects by
\begin{equation*}
(T_2 X)_n=
\left\{
  \begin{array}{ll}
    X_0, & n=0 \\
    X_1, & n=1 \\
    \pro{X_1}{X_0}{n}, & n>1\;.
  \end{array}
\right.
\end{equation*}
Further $T_2$ preserves levelwise equivalences of categories and there is a levelwise equivalence pseudo-natural transformation in $\Ps[\zuDop,\Cat]$
\begin{equation*}
  \pi^* T_2 X\rw X\;.
\end{equation*}
\end{theorem}
\begin{proof}
We apply Proposition \ref{fairtowg2-pro1} to the case when $\clC=\zuDop$, $\clD=\zD$, $\pi=\zuDop\rw\Dop$ as in Section \ref{fatdelta_first} and $F=X\in\fair{2}\subset\funcat{}{\Cat}$
with the following equivalences of categories between $X_{\eta}$ and $X_{\pi(\eta)}=X_n$ for $\eta:[n']\rw [n]$ in $\zuDop$. The map $\nu_{\eta}^{op}:\eta\rw[n]$ as in Definition \ref{fatdels-def2} is a coloured arrow. Hence the map
\begin{equation*}
  \za_{\eta}=X(\nu_{\eta}^{op}):X_{\eta}\rw X_{n}
\end{equation*}
is an equivalence of categories. Denote by $\zb_{\eta}=\zb_{\nu_{\eta}^{op}}$ the pseudo-inverse to $\za_{\eta}$ as in Lemma \ref{fairtowg2-lem01}. By Remark \ref{fairtowg2-rem3}, $\za_{\eta}\zb_{\eta}=\Id$. It remains to show that the hypotheses of Proposition \ref{fairtowg2-pro1} b) are satisfied.

Hypothesis b) i) holds by Proposition \ref{fatdels-pro00}. We next show that hypothesis b) ii) holds. Let $f_1:\mu_1\rw\zg_1$, $f_2:\mu_2\rw \zg_2$ be the following maps in $\zuD$ with $\pi f_1=\pi f_2=f$
\begin{equation*}
\xymatrix{
[m_1]\ar@{^{(}->}^{s_1}[r] \ar_{\mu_1}[d] & [n_1]\ar^{\zg_1}[d]\\
[m]\ar_{f}[r] & [n]
}
\qquad \qquad
\xymatrix{
[m_2]\ar@{^{(}->}^{s_2}[r] \ar_{\mu_2}[d] & [n_2]\ar^{\zg_2}[d]\\
[m]\ar_{f}[r] & [n]
}
\end{equation*}
We need to show that
\begin{equation}\label{fairtowg2-eq8}
\za_{\mu_1}X(f_1^{op})\zb_{\zg_1}=\za_{\mu_2}X(f_2^{op})\zb_{\zg_2}\;.
\end{equation}
Let $f=\zve\eta$ be the epi-mono factorization of $f$ in $\zD$ and let $f_1=\zve_1\eta_1$, $f_2=\zve_2\eta_2$, $\eta_{min}$, $\zve_{min}$ be as in Proposition \ref{fatdels-pro1}. Then by Proposition \ref{fatdels-pro1} there are commuting diagrams
\begin{equation*}
\xymatrix@R=30pt{
& X_n \ar^{X(\zve^{op})}[rrr]  &&& X_r  &\\
& X_{\eta'} \ar^{X(\zve_{min}^{op})}[rrr] \ar@<0.0ex>[u]^{\za_{\eta'}}  &&& X_{\eta} \ar@<0.0ex>[u]^{\za_{\eta}}    &\\
X_{\zg_1} \ar@/_2pc/[rrr]_{X(\zve_1^{op})}\ar@<-0.6ex>[ur]^{X(z_1^{op})} &&  X_{\zg_2} \ar@/_2pc/[rrr]_{X(\zve_2^{op})}\ar@<-0.6ex>[ul]_{X(z_2^{op})} &  X_{\zb_1} \ar@<0.6ex>[ur]^{X(w_1^{op})} &&  X_{\zb_2} \ar@<-0.6ex>[ul]_{X(w_2^{op})}
}
\end{equation*}
\mk

\begin{equation*}
\xymatrix@R=40pt{
X_m && X_{\eta}\ar_{X(\eta_{min}^{op})}[ll] \ar^{X(\eta_{min}^{op})}[rr] && X_m \\
X_{\mu_1}\ar^{\za_{\mu_1}}[u] & X_{\zb_1}\ar^{X(\eta_1^{op})}[l]\ar^{X(w_1^{op})}[ur] && X_{\zb_2}\ar_{X(w_2^{op})}[ul]\ar_{X_{(\eta_2^{op})}}[r] & X_{\mu_2}\ar_{\za_{\mu_2}}[u]
}
\end{equation*}
Denote by $\zb_{\eta'}=\zb_{\nu_{\eta'}^{op}}$ and $\zb_{\eta}=\zb_{\nu_{\eta}^{op}}$ be pseudo-inverses of $\za_{\eta'}$ and $\za_{\eta}$ as in Lemma \ref{fairtowg2-lem01}. Since $z_1:\eta'\rw\zg_1$ and $z_2:\eta'\rw\zg_2$ are coloured arrows (see Proposition \ref{fatdels-pro1}), $X(z_1^{op})$ and $X(z_2^{op})$ are equivalences of categories. We denote by $\zb_{z_1^{op}}$ and $\zb_{z_2^{op}}$ their pseudo-inverses as in Lemma \ref{fairtowg2-lem01}. By Remark \ref{fairtowg2-rem4}, $\zb_{\zg_1}= \zb_{\nu_{\zg_1}^{op}}=\zb_{z_1^{op}}\zb_{\eta'}$. Thus, since $X(f_1^{op})=X(\eta_1^{op})X(\zve_1^{op})$ we calculate
\begin{equation}\label{fairtowg2-eq9}
\begin{gathered}
\begin{split}
    & \za_{\mu_1}X(f_1^{op})\zb_{\zg_1}=\za_{\mu_1}X(\eta_1^{op})X(\zve_1^{op})\zb_{z_1^{op}} \zb_{\eta'}= X(\eta_{min}^{op}) X(w_1^{op})X(\zve_1^{op})\zb_{z_1^{op}} \zb_{\eta'}=\\
    & =X(\eta_{min}^{op})X(\zve_{min}^{op}) X(z_1^{op})\zb_{z_1^{op}} \zb_{\eta'}\;.
\end{split}
\end{gathered}
\end{equation}
Similarly, since $\zb_{\zg_2}=\zb_{z_2^{op}}\zb_{\eta'}$ and $X(f_2^{op})=X(\eta_2^{op})X(\zve_2^{op})$ we calculate
\begin{equation}\label{fairtowg2-eq10}
\begin{gathered}
\begin{split}
    & \za_{\mu_2}X(f_2^{op})\zb_{\zg_2}=\za_{\mu_2}X(\eta_2^{op})X(\zve_2^{op})\zb_{z_2^{op}} \zb_{\eta'}= X(\eta_{min}^{op}) X(w_2^{op})X(\zve_2^{op})\zb_{z_2^{op}} \zb_{\eta'}=\\
    & =X(\eta_{min}^{op})X(\zve_{min}^{op}) X(z_2^{op})\zb_{z_2^{op}} \zb_{\eta'}\;.
\end{split}
\end{gathered}
\end{equation}
By Lemma \ref{fairtowg2-lem02}, $X({z_1^{op}})\zb_{z_1^{op}} \zb_{\eta'}=\zb_{\eta'}=X({z_2^{op}})\zb_{z_2^{op}} \zb_{\eta'}$. Therefore \eqref{fairtowg2-eq9} and \eqref{fairtowg2-eq10} imply \eqref{fairtowg2-eq8} as required.

In conclusion all the hypotheses of Proposition \ref{fairtowg2-pro1} are satisfied and thus there is a pseudo-functor $T_2 X\in\Ps\funcat{}{\Cat}$ with
\begin{equation*}
(T_2 X)_n=X_n=
\left\{
  \begin{array}{ll}
    X_0, & n=0 \\
    X_1, & n=1 \\
    \pro{X_1}{X_0}{n}, & n>1\;.
  \end{array}
\right.
\end{equation*}
By Proposition \ref{fairtowg2-pro1} a morphism $f:X\rw Y$ in $\fair{2}$ induces a pseudo-natural transformation $T_2 f:T_2 X\rw T_2 Y$. Finally, the existence of the levelwise equivalence pseudo-natural transformation $\pi^* T_2 X\rw X$ follows from Proposition \ref{fairtowg2-pro1} c).
\end{proof}
\begin{definition}\label{fairtowg2Def-5}
  Let $R_2:\fair{2}\rw\catwg{2}$ be the composite
\begin{equation*}
  \fair{2}\xrw{T_2} \segpsc{}{\Cat}\xrw{\St} \catwg{2}\;,
\end{equation*}
where $T_2$ is as in Theorem \ref{fairtowg2-the1} and $\St$ is as in Theorem \ref{book-st-cat}.
\end{definition}

\section{The comparison result}\label{comparison}
In this section we establish our main result, Theorem \ref{compar-the1}, stating that the functors $F_2:\catwg{2}\rw\fair{2}$ and $R_2:\fair{2}\rw\catwg{2}$ induce an equivalence of categories after localization with respect to the $2$-equivalences. The proof of this result uses the category $\fairwg{2}$ and the results in Section \ref{wgfair2}.
We first note two properties of the functor $T_2$ and $R_2$.
\begin{lemma}\label{fairtowg2-lem1}
Let $X\in\catwg{2}$ . There is a pseudo natural transformation in $\Ps\funcat{}{\Cat}$
\begin{equation*}
T_2 F_2 X\rw X
\end{equation*}
which is a levelwise equivalence of categories.
\end{lemma}
\begin{proof}
Let $D_2 X$ be as in Definition \ref{D2def}. We first show that there is a pseudo-natural transformation $T_2 F_2 X\rw D_2 X$ which is a levelwise equivalence of categories.

We apply Proposition \ref{fairtowg2-pro1} to the case where $\clC=\zuDop$, $\clD=\Dop$, $\pi:\zuDop\rw\Dop$ as in Section \ref{fatdelta_first} and $F=\wt\pi^* X\in\funcat{}{\Cat}$ where $X\in\catwg{2}$ and $\wt\pi^*:\catwg{2}\rw\funcat{}{\Cat}$ is as in Definition \ref{fairtowg2Def-2}.

For each object $\eta:[n']\rw[n]$ in $\zuDop$, the map $\za_{\eta}=(\wt\pi^* X)(\nu_{\eta}^{op}):(\wt\pi^* X)_{\eta}=X_n\rw X_n$ is the identity, thus also $\zb_{\eta}=\Id$. Condition b) ii) in Proposition \ref{fairtowg2-pro1} holds trivially while condition b) i) holds by Proposition \ref{fatdels-pro00}. The result of applying Proposition \ref{fairtowg2-pro1} b) in this case is a strict functor from $\Dop$ to $\Cat$, which is precisely $D_2 X$, as immediate to check.

By Proposition \ref{fairtowg2Cor-2}, there is a morphism in $\funcat{}{\Cat}$ $S_2 X:F_2 X\rw \wt\pi^* X$ which is a levelwise equivalence of categories.

By Proposition \ref{fairtowg2-pro1} b) and the above, we therefore obtain a pseudo-natural transformation $T_2 F_2 X\rw D_2  X$ which is also a levelwise equivalence of categories.

Composing the latter with the pseudo-natural transformation $D_2 X\rw X$ of Remark \ref{D2def-Rem} (which is also a levelwise equivalence of categories), the result follows.
\end{proof}
\begin{lemma}\label{fairtowg2-lem2}
Given $Y\in\fair{2}$, there is a  pseudo-natural transformation $F_2 R_2 Y\rw Y$ in $\Ps[\zuDop,\Cat]$ which is a levelwise equivalence of categories.
\end{lemma}
\begin{proof}
By Theorem \ref{fairtowg2-the1} there is a levelwise equivalence pseudo-natural transformation in $\Ps[\zuDop,\Cat]$
\begin{equation}\label{fairtowg2-eq11}
  \pi^* T_2 Y\rw Y
\end{equation}
By the properties of the strictification functor (see Section \ref{sbs-strict-psfun}), there is a pseudo-natural transformation in $\Ps\funcat{}{\Cat}$
\begin{equation*}
  R_2 Y=\St T_2 Y \rw T_2 Y
\end{equation*}
which is a levelwise equivalence of categories. This induces a pseudo-natural transformation in $\Ps\fatcat{\zuDop}{\Cat}$
\begin{equation}\label{fairtowg2Eq-13}
  \pi^*R_2 Y\rw \pi^* T_2 Y
\end{equation}
which is a levelwise equivalence of categories. By Remark \ref{fairtowg2Rem-2} there is a pseudo-natural transformation in $\Ps\fatcat{\zuDop}{\Cat}$
\begin{equation}\label{fairtowg2Eq-14}
  \tilde\pi^* R_2 Y \rw \pi^* R_2 Y
\end{equation}
which is a levelwise equivalence of categories. On the other hand, by Proposition \ref{fairtowg2Cor-2}, there is a natural transformation in $[\zuDop,\Cat]$
\begin{equation}\label{fairtowg2Eq-15}
  F_2 R_2 Y\rw \tilde\pi^* R_2 Y
\end{equation}
which is a levelwise equivalence of categories.

Composing \eqref{fairtowg2Eq-15},  \eqref{fairtowg2Eq-14},  \eqref{fairtowg2Eq-13}, \eqref{fairtowg2-eq11} we obtain a pseudo-natural transformation in $\Ps\fatcat{\zuDop}{\Cat}$
\begin{equation}\label{fairtowg2Eq-16}
  F_2 R_2 Y= F_2\St T_2 Y\rw Y
\end{equation}
which is a levelwise equivalence of categories.

\end{proof}
The proof of Theorem \ref{compar-the1} will use twice the following remark about the strictification functor.
\begin{remark}\label{compar-rem1}
Let $\clC$ be a small category. Recall (see Section \ref{sbs-strict-psfun}) the adjunction $\St\dashv J$
\begin{equation*}
  \St: \Ps[\clC,\Cat]\rightleftarrows [\clC,\Cat]: J
\end{equation*}
where $\St$ is the strictification functor and $J$ is the inclusion. Let $X\in[\clC,\Cat]$ and suppose there is a pseudo-natural transformation $t:Z\rw JX$ in $\Ps[\clC,\Cat]$ such that $t_c$ is an equivalence of categories for all $c\in \clC$. Then by the adjunction $\St\dashv J$ this corresponds to a natural transformation $w:\St Z\rw X$ in $[\clC,\Cat]$ making the following diagram commute
\begin{equation*}
\xymatrix@R=40pt{
Z \ar[rr]^{\eta} \ar[drr]_{t} && J\St Z \ar[d]^{Jw}\\
&& JX
}
\end{equation*}
Since for all $c\in\clC$, $\eta_c$ is an equivalences of categories (see \cite{Lack}) and, by assumption, so is $t_c$ then by the above diagram $w_c$ is also an equivalences of categories.
\end{remark}
\begin{theorem}\label{compar-the1}
The functors
\begin{equation*}
  F_2:\catwg{2}\rightleftarrows \fair{2}:R_2
\end{equation*}
induce an equivalence of categories after localization with respect to the $2$-equivalences
\begin{equation*}
\wt F_2:  \catwg{2}\bsim\;\simeq\;\fair{2}\bsim : \wt R_2\;.
\end{equation*}
\end{theorem}
\begin{proof}
We are going to show that, for each $X\in\catwg{2}$, there is a 2-equivalence in $\catwg{2}$
\begin{equation}\label{compar-eq1}
  R_2 F_2 X \rw X
\end{equation}
and that this is natural in $X$. We then will show that, for each $Y\in\fair{2}$, there is a zig-zag of 2-equivalences in $\fair{2}$
\begin{equation}\label{compar-eq2}
F_2 R_2 Y \lw D\, \St F_2 R_2 Y \rw Y
\end{equation}
and this is natural in $Y$. Then \eqref{compar-eq1} and \eqref{compar-eq2} imply the result. In fact, by \eqref{compar-eq1} we have an isomorphism $R_2 F_2 X\cong X$ in $\catwg{2}\bsim$ and by naturality with respect to $X$ we have the natural isomorphism $\wt R_2\wt F_2\cong\Id_{\Sz{\catwg{2}\bbsim}}$ where $\wt R_2$ and $\wt F_2$ are the functors induced by $R_2$ and $F_2$ on the localizations. Similarly \eqref{compar-eq2} means that $F_2 R_2 Y\cong Y$ in $\fair{2}\bsim$; by the naturality with respect to $Y$ this implies that there is a natural isomorphism $\wt F_2 \wt R_2\cong\Id_{\Sz{\fair{2}\bbsim}}$, so in conclusion $\wt F_2$ and $\wt R_2$ are equivalences of categories.

Let $X\in\catwg{2}$. By Lemma \ref{fairtowg2-lem1} there is a pseudo-natural transformation in $\Ps[\zD^{op},\Cat]$
\begin{equation}\label{compar-eq3}
T_2 F_2 X\rw X
\end{equation}
which is a levelwise equivalence of categories. Applying Remark \ref{compar-rem1} (with $\clC=\zD^{op}$) to \eqref{compar-eq3} we obtain a natural transformation in $\funcat{}{\Cat}$
\begin{equation}\label{compar-eq4}
R_2 F_2 X=\St T_2 F_2 X\rw X
\end{equation}
which is a levelwise equivalence of categories, and therefore also a 2-equivalence (see Remark \ref{wg-doubcat-rem-2}).

Given a morphism $f:X\rw X'$ in $\catwg{2}$, since \eqref{compar-eq4} is a natural transformation, this induces a commuting diagram
\begin{equation}\label{compar-eq5}
\begin{gathered}
\xymatrix{
R_2 F X \ar[r] \ar_{R_2 F f}[d] & X \ar^{f}[d]\\
R_2 F X' \ar[r] & X'
}
\end{gathered}
\end{equation}
Let $Y\in\fair{2}$. By Lemma \ref{fairtowg2-lem2} there is a levelwise equivalence pseudo-natural transformation $F_2 R_2 Y\rw Y$ in $\Ps[\zuDop,\Cat]$. Applying to it Remark \ref{compar-rem1} (with $\clC=\zuDop$) we obtain a natural transformation in $\fatcat{\zuDop}{\Cat}$

 \begin{equation}\label{compar-eq6}
  \St F_2 R_2 Y\rw Y
\end{equation}
which is a levelwise equivalence of categories.

 By Remark \ref{segpseu-rem1}, since $F_2 \St T_2 Y=F_2 R_2 Y \in\fair{2}$, then $F_2 R_2 Y\in\segps\fatcat{\zuDop}{\Cat}$; thus, by Theorem \ref{fairtowg2The-1}, $\St F_2 R_2 Y\in\fairwg{2}$. By Remark \ref{strict-psfun-Rem-1} there is a natural transformation in $\fatcat{\zuDop}{\Cat}$
\begin{equation}\label{compar-eq7}
  \St F_2 R_2 Y\rw F_2 R_2 Y
\end{equation}
which is a levelwise equivalence of categories.

 Applying the functor $D$ of Lemma \ref{fairtowg2Lem-6} to \eqref{compar-eq6} and \eqref{compar-eq7} we obtain a zig-zag in $\fair{2}$
\begin{equation*}
  D F_2 R_2 Y=F_2 R_2 Y \lw D \St F_2 R_2 Y\rw D Y=Y\;.
\end{equation*}
Since $D$ preserves levelwise equivalences of categories, this is a zig-zag of levelwise equivalences and therefore (see Remark \ref{wg-doubcat-rem-2}) of 2-equivalences in $\fair{2}$.

If $f: Y\rw Y'$ is a morphism in $\fair{2}$, by the naturality of \eqref{compar-eq6} and \eqref{compar-eq7} we obtain commuting diagrams in $\fatcat{\zuDop}{\Cat}$
\begin{equation*}
\xymatrix@C=40pt{
F_2R_2 Y\ar_{F_2R_2 f}[d] & \St F_2 R_2 Y \ar[l] \ar^{\St F_2 R_2 f}[d] \ar[r] & Y \ar^{f}[d]\\
F_2 R_2 Y' & \St F_2 R_2 Y' \ar[l] \ar[r] & Y'
}
\end{equation*}
By functoriality of $D$, this give rise to  commuting diagrams in $\fair{2}$
\begin{equation}\label{compar-eq8}
\begin{gathered}
\xymatrix@C=45pt{
F_2 R_2 Y = D F_2 R_2 Y \ar_{F_2R_2 f}[d] & D \St F_2 R_2 Y \ar[l] \ar^{D\St F_2 R_2 f}[d] \ar[r] & D Y=Y \ar^{f}[d]\\
F_2 R_2 Y'=D F_2 R_2 Y' & D\St F_2 R_2 Y' \ar[l] \ar[r] & DY'=Y'
}
\end{gathered}
\end{equation}
In conclusion, both \eqref{compar-eq1} and \eqref{compar-eq2} hold, and are natural in $X$ and $Y$ respectively, as required.
\end{proof}
\begin{corollary}\label{compar-cor1}
There is an equivalence of categories
\begin{equation*}
   \fair{2}\bsim\;\simeq\;\ta{2}\bsim\;.
\end{equation*}
\end{corollary}
\begin{proof}
By Theorem \ref{compar-the1} there is an equivalence of categories $ \fair{2}\bsim\;\simeq\;\catwg{2}\bsim\;$ while by \cite[Theorem 12.2.6]{PBook2019} there is an equivalence of categories $\catwg{2}\bsim\;\simeq\;\ta{2}\bsim\;.$ Hence the result.
\end{proof}
We finally observe that the equivalence up to homotopy between $\catwg{2}$ and $\fair{2}$ specializes to an equivalence between the groupoidal versions of these models, defined as follows. We denote by $\Gpd$ the category of groupoids.
\begin{definition}\cite{PBook2019}\label{compar-def1}
The category $\gcatwg{2}$ of groupoidal weakly globular double categories is the full subcategory of $\catwg{2}$ whose objects $X$ are such that $X_k\in\Gpd$ for all $k\in\Dop$ and $\p{1}X\in\Gpd$.
\end{definition}
\begin{definition}\label{compar-def2}
The category $\gfair{2}$ of groupoidal weakly globular fair 2-categories is the full subcategory of $\fair{2}$ whose objects $X$ are such that $X_{\eta}\in\Gpd$ for all $\eta\in\zuDop$ and $\p{1}X\in\Gpd$.
\end{definition}
\begin{corollary}\label{compar-cor2}
The functors $F_2,\,R_2$ of Theorem \ref{compar-the1} restrict to functors
\begin{equation*}
  F_2:\gcatwg{2}\leftrightarrows \gfair{2}:R_2
\end{equation*}
inducing an equivalence of categories after localization with respect to the 2-equivalences:
\begin{equation*}
  \gcatwg{2}\bsim\;\simeq\,\fair{2}\bsim\;.
\end{equation*}
\end{corollary}
\begin{proof}
Let $X\in\gcatwg{2}$. Then $X_1\in\Gpd$, $X_0\in\Gpd$. So if $\eta\in\zuDop$, from the expression \eqref{fairtowg2eq-01} of $(F_2 X)_{\eta}$ we see that $(F_2 X)_{\eta}\in\Gpd$. Also $\p{1}X\in\Gpd$ hence (using Theorem \ref{StrongpseThe-1}), $\p{1}F_2 X\cong\p{1}X \in\Gpd$. We conclude that $F_2 X\in\gfair{2}$.

Let $Y\in\gfair{2}$. By Remark \ref{compar-rem1}, there is a pseudo natural transformation $R_2 Y=\St T_2 Y\rw T_2 Y$ which is a levelwise equivalence of categories. By the expression of $T_2 Y$ (Theorem \ref{fairtowg2-the1}), $(T_2 Y)_k\in\Gpd$ for all $k\in\Dop$; since a category equivalent to a groupoid is itself a groupoid, $(R_2 Y)_k\in\Gpd$ for all $k\in\Dop$. Also, $\p{1}R_2 Y\cong\p{1}T_2 Y\cong\p{1}Y\in\Gpd$. We conclude that $R_2 Y\in\gcatwg{2}$. Thus we have functors
\begin{equation*}
F_2:\gcatwg{2}\leftrightarrows \gfair{2}:R_2\;.
\end{equation*}
In the proof of Theorem \ref{compar-the1} we showed that, that for each $X\in\catwg{2}$, $Y\in\fair{2}$, there are natural zig-zags of levelwise equivalences of categories (and thus 2-equivalences)
\begin{equation*}
R_2 F_2 X\rw X\qquad F_2 R_2 Y\lw D\St F_2 R_2 Y\rw Y\;.
\end{equation*}
If $X\in\gcatwg{2}$, $Y\in\gfair{2}$ these are zig-zags of levelwise equivalences of categories (and thus 2-equivalences) in $\gcatwg{2}$ and $\gfair{2}$ respectively. Therefore there is an induced equivalence of categories after localization:
\begin{equation*}
  \gcatwg{2}\bsim\;\simeq\,\fair{2}\bsim\;.
\end{equation*}
\end{proof}




\bibliography{C:/Users/simon/OneDrive/Documents/LATEX_ARCHIVES_BIBLIO/BIBLIOGRAPHY/BIB_LIB_BOOK_B}{}
\end{document}